\documentclass[11pt]{amsart}
\usepackage{amssymb,amsmath,amsfonts,latexsym, xcolor,amsthm}

\usepackage[all,cmtip]{xy}
\usepackage{verbatim}
\usepackage{tikz}
\usetikzlibrary{cd}

\newtheorem{thm}{Theorem}[section]
\newtheorem{prop}[thm]{Proposition}

\newtheorem{cor}[thm]{Corollary}

\newtheorem{definition}[thm]{Definition}
\newtheorem{example}[thm]{Example}

\newtheorem{remark}[thm]{Remark}
\newtheorem{lemma}[thm]{Lemma}
\newtheorem{theorem}[thm]{Theorem}

\numberwithin{equation}{section}
\newcommand{\D}{\mathbb{D}}

\newcommand{\T}{\mathbb{T}}
\newcommand{\C}{\mathbb{C}}

\newcommand{\Z}{\mathbb{Z}}

\newcommand{\clb}{\mathcal{B}}

\newcommand{\cld}{\mathcal{D}}

\newcommand{\clh}{\mathcal{H}}
\newcommand{\clk}{\mathcal{K}}

\newcommand{\cls}{\mathcal{S}}

\newcommand{\clw}{\mathcal{W}}

\newcommand\restr[2]{\ensuremath{\left.#1\right|_{#2}}}

\begin{document}
\title[Wold decompositions and unitary extensions of product systems]{Twisted representations of product systems of $C^*$-correspondences: Wold decomposition and unitary extensions}
\author{Baruch Solel}
\address{Department of Mathematics, Technion—Israel Institute of Technology, Haifa 32000, Israel.}
\email{mabaruch@technion.ac.il}
\author{Mansi Suryawanshi}
\address{Department of Mathematics, Technion—Israel Institute of Technology, Haifa 32000, Israel.}
\email{suryawanshi@campus.technion.ac.il; mansisuryawanshi1@gmail.com}

\date{\today}

\begin{abstract}
We investigate Wold-type decompositions and unitary extension problems for
multivariable isometric covariant representations associated with product systems
of $C^*$-correspondences. First, we establish an operator-theoretic characterization
for the existence of a Wold decomposition for the tuple $(\sigma, T_1, T_2, \ldots, T_n)$, where each
$(\sigma,T_i)$ is an isometric covariant representation of a
$C^*$\nobreakdash-correspondence. We then introduce twisted and doubly twisted covariant representations of product
systems. For doubly twisted isometric representations, we prove the
existence of a Wold decomposition, recovering earlier results for doubly commuting
representations as special cases. We further obtain explicit descriptions of the
resulting Wold summands and develop concrete Fock-type models realizing each
component. We present non-trivial examples of these families. Finally, we construct unitary extensions via a direct-limit
procedure. As applications, we obtain unitary extensions for several previously
studied classes of operator tuples, including doubly twisted, doubly non-commuting,
and doubly commuting isometries, and for a special class of doubly twisted representations of product system.
\end{abstract}

\maketitle	
\tableofcontents

\section{Introduction}

The classical Wold decomposition theorem is a cornerstone of the structural theory of isometries on Hilbert spaces. Let $\mathcal H$ be an infinite dimensional separable Hilbert space over the complex plane $\C$, and
let $\mathcal B(\mathcal H)$ denote the algebra of all bounded linear operators
on $\mathcal H$.
An operator $V \in \mathcal B(\mathcal H)$ is called an \emph{isometry} if
\(
V^*V = I_{\mathcal H}.
\)
An isometry $V$ is called a \emph{shift} if
\(\|V^{*m}h\|\longrightarrow 0 \quad \text{as } m\to\infty \text{ for all } h\in\mathcal H,
\)
and it is called \emph{unitary} if $VV^* = I_{\mathcal H}$.
The Wold decomposition asserts that every isometry $V \in \mathcal B(\mathcal H)$
admits a unique orthogonal decomposition
\(
\mathcal H = \mathcal H_u \oplus \mathcal H_s,
\)
where $\restr{V}{\mathcal H_u}$ is unitary and $\restr{V}{\mathcal H_s}$ is  shift \cite{NagyFoias}. This canonical decomposition has far-reaching consequences across operator theory, stochastic processes~\cite{Miamee}, prediction theory~\cite{Helson}, time series analysis~\cite{Masani}, operator models~\cite{NagyFoias}, $C^*$-algebras~\cite{Coburn, Omland, Weber}.

The analytical elegance of this decomposition is highly desirable, yet its translation to multi-variable operator systems specifically, tuples of isometries $\{V_1, \dots, V_k\}$ for $k \ge 2$ acting on a common space $\mathcal{H}$ proves profoundly challenging. Early multivariable analogues typically required additional
structural hypotheses. 
In particular, Słociński established such a decomposition for a doubly commuting pair of isometries in~\cite{Slocinski80}. This result, along with its subsequent generalizations, was later used in~\cite{BCL} to construct models for
tuples of commuting isometries and to analyze the structure of the $C^*$-algebras they generate. Popovici studied the Wold decompositon for a pair of commuting isometries~\cite{Popovici10}. Popescu generalized the Wold decomposition to the noncommutative setting of row isometries $V=(V_1,\ldots,V_n)$ with mutually orthogonal ranges, obtaining a unique orthogonal decomposition into a \emph{row unitary} part and a \emph{pure} (shift-type) part modeled on the full Fock space over $\mathbb C^n$
\cite{Popescu90}. For a systematic and detailed study of doubly commuting and twisted analogues in higher-dimensional settings and related general frameworks, see~\cite{Gaspar, JP, Ko, RSS, S}.

To study operator tuples within a unified framework extending beyond the classical
Hilbert space setting, Hilbert $C^*$-modules and $C^*$-correspondences provide a
natural and effective framework~\cite{Lance,Muhly}. These ideas were subsequently
developed in depth in a series of papers by P.~Muhly and B.~Solel, who established
the theory of tensor algebras over $C^*$-correspondences.
In this context, Muhly and Solel
\cite{MS} established a Wold-type decomposition for isometric covariant
representations of a $C^*$-correspondence over a $C^*$-algebra $\mathcal A$.
They proved that every isometric covariant representation admits a unique decomposition
as a direct sum of two types of representations: one that is \emph{induced} by a representation
of $\mathcal A$, in the sense of Rieffel~\cite{Rieffel}, and another that gives rise to a
representation of the Cuntz–Pimsner algebra.
The induced isometric covariant representations share many features with shifts. To a large extent, shifts reflect the representation theory of $\mathbb C$ and in fact shifts are essentially
induced by representations of $\mathbb C$.
In the general setting, however, the situation is considerably more intricate than in the
classical case of shifts. Popescu’s Wold decomposition theorem~\cite{Popescu90},
which can be recovered as a special case by of this decomposition by considering the $C^*$-correspondence $E$ as a Hilbert space over $\mathbb C$, and the left action defined by $\varphi(\lambda)=\lambda I_E$ for
$\lambda \in\mathbb C$.
Skalski and Zacharias
\cite{Skalski} subsequently established a higher-rank version of the main result of \cite{MS}. They proved Wold decomposition for
\emph{doubly commuting} isometric representations of product systems of
$C^*$-correspondences over $\mathbb Z_+^k$, yielding the expected $2^k$-fold
orthogonal splitting. The corresponding crucial concept of a product system of a $C^*$-correspondence
over $\mathbb{N}_0^k$ was introduced in~\cite{Fow}. This notion has been further exploited in recent work by B.~Solel on dilations of
commuting completely positive maps (see~\cite{Solel,Solel2}).

While these Wold-type theorems provide canonical structural decompositions,
they do not give a criterion for \emph{when} a given multivariable isometric
representation admits such a decomposition.  In the classical setting of
tuples of isometries on a Hilbert space, this problem was resolved in
\cite[Theorem~2.1]{RSS2} via a purely operator-theoretic characterization: a tuple
$V = (V_1,\ldots,V_n)$ of isometries admits a von Neumann--Wold decomposition if and only if the
individual Wold components $\mathcal{H}_{V_i,u}$ (or equivalently $\mathcal{H}_{V_i,s}$) reduce every
$V_j$ for all $1 \leq i,j \leq k$. For a commuting triple of isometries $\{V_1,V_2,V_3\}$, the existence result was established by
Gaspar and Suciu (see~\cite[Theorem~2]{Gaspar}).

Section~\ref{Characterization for Existence of Wold Decomposition} aims to extend this
characterization from the setting of $n$-tuples of isometries on a Hilbert space to
families of the form $(\sigma, T_1, T_2, \ldots, T_n)$, where for each $i\in I_n$,
$(\sigma,T_i)$ is an isometric covariant representation of a
$C^*$\nobreakdash-correspondence $E_i$ over a $C^*$\nobreakdash-algebra
$\mathcal A$ on $\mathcal H$, and
$\sigma:\mathcal A\to\mathcal B(\mathcal H)$ is a fixed nondegenerate
$*$\nobreakdash-representation (cf. Theorem~\ref{thm: class of decom}).
In this higher-rank
framework, the analogues of the ``shift'' and ``unitary'' parts are the induced and fully coisometric
components of each covariant pair. We establish that the existence of a Wold decomposition
for a tuple $(\sigma, T_1,\ldots,T_k)$ is precisely equivalent to the mutual reducibility of these
components, thereby extending the von Neumann--Wold criterion to the categorical setting of product
systems.

In Section \ref{Doubly Twisted Representations}, we introduce the notions of \emph{twisted} and \emph{doubly twisted} representations of
a product system, which form the central objects of our study. A twisted representation extends the
usual framework of covariant representations by allowing the coordinate maps $T_{i}$ to satisfy
commutation relations that are modified by a prescribed family of unitaries
$\{U_{ij}\}_{i\neq j}$ acting on the underlying Hilbert space. This additional twisting captures
interactions between the coordinate directions that are not necessarily symmetric or commuting. As an application of the characterization theorem established in
Section~\ref{Characterization for Existence of Wold Decomposition}, we prove the
existence of a Wold decomposition for doubly twisted isometric representations of product systems (cf. Theorem~\ref{thm: decompo for doubly twisted rep}). As a special case, we recover the Wold decomposition
obtained by Skalski and Zacharias~\cite{Skalski} for doubly commuting isometric
representations.

In Section~\ref{Examples}, we present a series of illustrative examples, confirming that the setting is rich and nontrivial. A central role in this paper is played by a concrete Fock-type model
for doubly twisted covariant representations.
This model provides a unified framework in which induced and fully
coisometric components coexist in a controlled manner.
More precisely, given a subset $A\subseteq I_k$, the coordinates indexed
by $A$ give rise to an induced (shift-type) part, while the complementary
coordinates $A^{\mathrm c}$ contribute a fully coisometric part governed
by a twisted representation on an auxiliary Hilbert space.
The construction yields a large and flexible class of doubly twisted
representations, extending the classical Fock representation and its
variants to the setting of twisted product systems.
Moreover, direct sums of such models corresponding to different choices
of $A$ again produce doubly twisted representations, illustrating the
robustness of the framework.
We present this model explicitly in Example~\ref{ex:fock-model} (see~\eqref{eq:model-operators}), which
serves as a guiding example for the general theory developed in the
subsequent sections.
Then we construct explicit non-trivial examples of operator families
that give rise to twisted and doubly twisted representations, illustrating how
the corresponding relations arise naturally in different settings. In the scalar case $\mathcal A=\mathbb C$
and the automorphic case $E_i=\alpha_i\mathcal A$, the twisted relations
translate into explicit operator identities, and extend Solel's characterization
\cite{Solel} of doubly commuting representations. In addition, we provide two concrete examples of operators in the latter case.
Taken together, these results demonstrate that (doubly) twisted representations
can be constructed from concrete operator families satisfying explicit algebraic
relations.

In Section \ref{Wold Decomposition}, we obtain concrete representations of the summands arising in the Wold decomposition
of doubly twisted isometric representations of the product system.
Starting with an isometric covariant $k$-tuple
$(\sigma,T^{(1)},\ldots,T^{(k)})$, we introduce the wandering subspaces $ W_A=\displaystyle\bigcap_{i\in A}W_i,$ where $W_i := \operatorname{Ran}(I-T_{i}T_{i}^*).$
We then give explicit representations for each Wold summand $\mathcal{H}_A$,
\[
\mathcal{H}_A
=
\sum_{n\in\mathbb{Z}_+^{|A|}}
T_A^{(n)}
\left(I_{E_A^n}\otimes
\sum_{m\in\mathbb{Z}_+^{k-|A|}}
T_{I_k\setminus A}^{(m)}
\left(I_{E_{I_k\setminus A}^m}\otimes W_A\right)\right).
\]
On each summand $\mathcal H_A$ the coordinates indexed by $A$ represent induced part (shift--like) while those indexed by $A^{\mathrm c}$ are fully
coisometric. 
As a corollary, the case $A=I_k$ recovers the induced Fock part, while $A=\varnothing$ yields the
fully coisometric part. 

Motivated by the classical one variable theory, in Section \ref{Models}, we pass to a Fock space model in the multivariable, doubly twisted setting.  In the one variable case, the Wold decomposition identifies the shift part of an isometry with multiplication by $z$ on a vector--valued Hardy space via the canonical unitary, providing a concrete functional model.  This philosophy has been a guiding theme in operator theory.  
In this section, our goal is to describe an explicit model for each summand $\mathcal{H}_A$ obtained
from the Wold decomposition of a doubly twisted isometric covariant representation
$(\sigma,T_1,\ldots,T_k)$. Using the canonical unitary
$\Pi_A\colon \mathcal{H}_A\to\mathcal{F}(E_A)\otimes\mathcal{D}_A$, we transport the restricted
representation $\restr{(\sigma,\{T_i\}_{i\in I_k})}{\mathcal{H}_A}$ to the Fock space model. This yields a
new tuple of operators $(\sigma_A,\{M_{A,i}\}_{i\in I_k})$ acting on
$\mathcal{F}(E_A)\otimes\mathcal{D}_A$, unitarily equivalent to the original one on $\mathcal{H}_A$ (cf. Prposition~\ref{prop:construction of MAi}).
In this model, the components $M_{A,i}$ with $i\in A$ appear as creation operators, whereas those
for $i\in A^c$ act as fully coisometric parts.

One of the fundamental consequences of the Wold decomposition is that every isometry
admits a unitary extension. It is also well known that every contraction admits a
coisometric extension. In~\cite[Section 5]{Muhly}, P.~Muhly and B.~Solel  obtained
far-reaching generalizations of the latter result by constructing an explicit
inductive extension procedure for completely contractive covariant representations
of a $C^*$-correspondence (see also~\cite{Muhly2}, where the special case of a
$C^*$-correspondence induced by a unital injective endomorphism of $\mathcal A$
is treated). Their approach, inspired by methods from classical operator theory,
produces, under suitable technical assumptions, a fully coisometric extension of a
completely contractive representation. However, except in certain special cases, it
remains unknown whether a fully coisometric extension obtained in this manner from an
isometric representation is itself isometric.
In a different direction, Skalski and Zacharias (cf. \cite[Section~5]{Skalski}) showed 
that one can obtain extensions which are simultaneously isometric and fully
coisometric, which they termed \emph{unitary extensions}. In the one-variable case,
they established the existence of such unitary extensions and highlighted the
difficulties involved in constructing analogous extensions in higher-rank settings.

In light of these developments, it is natural to seek a systematic construction of
unitary extensions for higher-rank representations, particularly in the presence of
twisted commutation relations. This is precisely the goal of Section~\ref{Unitary Extension}.

We develop a new approach for doubly
twisted isometric covariant representations of product systems via a direct-limit construction. Starting from a compatible direct
system $\{(\mathcal{H}_m,\phi_{n,m})\}_{m\le n}$ of Hilbert spaces carrying doubly twisted isometric
covariant representations $(\sigma_m,T_{m,1},\ldots,T_{m,k})$ with twists
$\{t_{ij}\otimes U_{m,ij}\}_{i\neq j}$, we show in Theorem~\ref{thm:inductive-limit} that there exists a limiting twisted
isometric covariant representation $(\sigma_\infty,T_{\infty,1},\ldots,T_{\infty,k})$ on the direct
limit Hilbert space $\mathcal{H}_\infty$, together with unitaries $\{U_{\infty,ij}\}_{i\neq j}$,
extending all the given data. We identify a natural and essentially minimal hypothesis on the
connecting maps which ensures that the limiting representation is fully coisometric, and hence
doubly twisted.

As consequences, we obtain unitary extensions for several well-known classes of operators,
including doubly twisted isometries studied by Sarkar, Rakshit, and Suryawanshi~\cite{RSS},
doubly non-commuting tuples of  isometries introduced by Jeu and Pinto~\cite{JP}, and doubly commuting tuples of 
isometries due to Sarkar~\cite{S} and Słociński~\cite{Slocinski80}. As an application, we prove that multiplication operators $(M_{z_1},\cdots, M_{z_n})$ on $H^2(\mathbb D^n)$ admit a unitary extension.
When $\mathcal{A}$ is abelian and each $E_i=\mathcal{A}$, we further obtain an explicit model on
$\ell^2(\mathbb{Z}_+^{|A|})\otimes\mathcal{D}_A$ and show that every doubly twisted isometric covariant representation admits a doubly twisted unitary extension in this setting (cf. Example~\ref{ex: unitary-ext-commutative}).

In the next section, we collect the basic definitions and known results that will be used throughout the paper.

\section{Preliminaries} \label{Preliminaries}
 Throughout this paper, all Hilbert spaces are assumed to be separable and defined over $\mathbb{C}$ and we let  $\mathcal A$ denote the $C^*$-algebra.
We set $\mathbb Z_+:=\{0,1,2,\ldots\}$ and, for $k\in\mathbb Z_+$, we write
$I_k:=\{1,2,\ldots,k\}$.
Inner products on Hilbert spaces are conjugate-linear
in the first variable and linear in the second. Also $I_{\clh}$ denotes the identity operator on a Hilbert space $\clh.$ Given $T \in \mathcal B(\mathcal H)$, its range and kernel are defined by
$\text{Ran}T = T(\mathcal H)$ and $\ker T = \{h \in \mathcal H : Th=0\}$, respectively.

\begin{definition}
A (right) \emph{Hilbert $\mathcal A$-module} is a complex vector space $E$
equipped with a right $\mathcal A$-module structure and an
$\mathcal A$-valued inner product
$\langle\cdot,\cdot\rangle_{\mathcal A}:E\times E\to\mathcal A$
which is linear in the second variable, conjugate-linear in the first,
satisfies
\[
\langle\xi,\eta a\rangle_{\mathcal A}
=\langle\xi,\eta\rangle_{\mathcal A}a,\qquad
\langle\xi,\eta\rangle_{\mathcal A}
=\langle\eta,\xi\rangle_{\mathcal A}^*,
\]
and $\langle\xi,\xi\rangle_{\mathcal A}\ge0$ with equality if and only if
$\xi=0$. When $E$ is complete with respect to
the norm $\|\xi\|:=\|\langle\xi,\xi\rangle_{\mathcal A}\|^{1/2}$, it is called a 
\emph{Hilbert $C^*$\nobreakdash-module over $\mathcal{A}$}, 
or simply a \emph{Hilbert $\mathcal{A}$-module}. 
\end{definition}

For a Hilbert $\mathcal A$-module $E$, an operator $T:E\to E$ is called \emph{adjointable} if there exists an operator
$T^*:E\to E$ such that $\langle T\xi,\eta\rangle_{\mathcal A}
=
\langle\xi,T^*\eta\rangle_{\mathcal A},
\, \xi,\eta\in E.$
The set of all adjointable operators on  is denoted by $\mathcal L(E)$ (see \cite{Lance}).
The presence of adjointable operators allows one to equip a Hilbert
$\mathcal A$-module with a left action of $\mathcal A$, leading to the notion
of a $C^*$-correspondence.

\begin{definition}
A \emph{$C^*$-correspondence} over a $C^*$-algebra $\mathcal A$ is a right
Hilbert $\mathcal A$-module $E$ together with a nondegenerate
$*$-homomorphism $\varphi:\mathcal A\longrightarrow\mathcal L(E),$ 
called the \emph{left action}.  We write $a\cdot\xi:=\varphi(a)\xi$ for
$a\in\mathcal A$ and $\xi\in E$.
\end{definition}
 
Throughout the paper, all $C^*$-correspondences are assumed to be essential as
left $\mathcal A$-modules, i.e. the closed linear
span of $\varphi(\mathcal{A})E$ is equal to $E$, and all representations of $\mathcal A$ are assumed
to be nondegenerate. 
A good, general, up-to-date reference for Hilbert C*-modules is the monograph by Lance~\cite{Lance}.

Let $E$ and $F$ be $C^*$-correspondences over a $C^*$-algebra $\mathcal A$, with
left actions $\varphi_E$ and
$\varphi_F$ respectively. Then the balanced tensor product $E\otimes_{\mathcal A}F$ is a $C^*$-correspondence
over $\mathcal A$. It is defined as the Hausdorff completion of the algebraic
balanced tensor product, equipped with the $\mathcal A$-valued inner product
given by
\[
\langle \xi_1\otimes \eta_1,\;\xi_2\otimes \eta_2\rangle
=
\langle \eta_1,\,
\varphi_F\!\bigl(\langle \xi_1,\xi_2\rangle_E\bigr)\eta_2\rangle_F.
\]
The right and left actions of $\mathcal A$ are given by $(\xi\otimes\eta)a:=\xi\otimes(\eta a),\, a\cdot(\xi\otimes\eta):=(\varphi_E(a)\xi)\otimes\eta.$ Whenever the underlying correspondence is understood, we suppress the subscript and
write the left action as $\varphi$.
Let $E^0:=\mathcal A$. For $n\ge1$ define inductively $E^n:=E\otimes_{\mathcal A}E^{n-1}.$
Each $E^n$ is again a $C^*$-correspondence over $\mathcal A$ with left action
$\varphi^{(n)}:\mathcal A\to\mathcal L(E^n)$ given by
\[
\varphi^{(n)}(a)(\xi_1\otimes\cdots\otimes\xi_n)
=
\varphi(a)\xi_1\otimes\xi_2\otimes\cdots\otimes\xi_n,
\qquad a\in\mathcal A,
\]
where $\varphi^{(0)}=\mathrm{id}_{\mathcal A}$.
The associated \emph{Fock module} is defined by $\mathcal F(E):=\bigoplus_{n\in\mathbb Z_+}E^n,$
which is a $C^*$-correspondence over $\mathcal A$ with left action
$\varphi_\infty:\mathcal A\to\mathcal L(\mathcal F(E))$ given by
\[
\varphi_\infty(a)\left(\bigoplus_{n\in\mathbb Z_+}\xi_n\right)
:=
\bigoplus_{n\in\mathbb Z_+}\varphi^{(n)}(a)\xi_n,
\qquad a\in\mathcal A.
\]
For $\xi \in E$ and $\eta \in E^n$, the \emph{creation operator} $T_\xi$ on $\mathcal{F}(E)$ is defined by $T_\xi(\eta) := \xi \otimes \eta.$
If $E$ is a $C^*$-correspondence over $\mathcal A$ and
$(\sigma,\mathcal H)$ is a nondegenerate representation of $\mathcal A$, we
denote by $E\otimes_\sigma\mathcal H$ the Hilbert space completion of the
algebraic tensor product $E\odot\mathcal H$ with respect to the inner product
\(
\langle \xi\otimes h,\eta\otimes k\rangle
:=\langle h,\sigma(\langle\xi,\eta\rangle_{\mathcal A})k\rangle,
\,
\xi,\eta\in E,\ h,k\in\mathcal H.\)

\begin{definition}\cite{Rieffel}
Let $E$ be a $C^*$-correspondence over  $\mathcal A$ with left
action $\varphi$, and let
$\pi:\mathcal A\to\mathcal B(\mathcal H)$ be a nondegenerate representation on a
Hilbert space $\mathcal H$.  The associated \emph{induced (isometric covariant)
representation} $(\rho,S)$ of $E$ on the Hilbert space
$\mathcal F(E)\otimes_\pi\mathcal H$ is defined by
\[
\rho(a):=\varphi_\infty(a)\otimes I_{\mathcal H},
\quad a\in\mathcal A, \quad S(\xi):=T_\xi\otimes I_{\mathcal H},
\quad \xi\in E.
\]
\end{definition}
Now we define the representations considered by Muhly and Solel:
\begin{definition}\cite{MS}
A \emph{completely contractive covariant representation} of $E$ on $\mathcal H$
is a pair $(\sigma,T)$, where $\sigma:\mathcal A\to\mathcal B(\mathcal H)$ is a
nondegenerate $*$-representation and $T:E\to\mathcal B(\mathcal H)$ is a linear,
completely contractive map satisfying
\[
T(a\,\xi\,b)=\sigma(a)T(\xi)\sigma(b),
\qquad a,b\in\mathcal A,\ \xi\in E.
\]
The representation $(\sigma,T)$ is called \emph{isometric} if for every $\xi,\eta\in E$, we have $T(\xi)^*T(\eta)=\sigma(\langle\xi,\eta\rangle_{\mathcal A}).$
Associated to $(\sigma,T)$ is a contraction
$\widetilde T:E\otimes_\sigma\mathcal H\to\mathcal H$ defined by $\widetilde T(\xi\otimes h):=T(\xi)h,\, \xi\in E,\ h\in\mathcal H$,
which satisfies the covariance relation
\begin{equation}\label{eq:left_action}
\widetilde T(\varphi(a)\xi\otimes h)
=
\sigma(a)\widetilde T(\xi\otimes h),
\qquad a\in\mathcal A.
\end{equation}
For a fixed $\sigma$, the correspondence $T\leftrightarrow\widetilde T$ gives a
one-to-one correspondence between covariant representations $(\sigma,T)$ of $E$
and contractions $\widetilde T$ satisfying~\eqref{eq:left_action}
(see~\cite{Muhly}).

Moreover, $(\sigma,T)$ is isometric if and only if $\widetilde T$ is an
isometry.  The representation $(\sigma,T)$ is called \emph{fully coisometric} if
$\widetilde T\,\widetilde T^*=I_{\mathcal H}$.
\end{definition}

In addition to $\widetilde T$, we also require the generalized higher powers of it. For 
$i \in I_k$ and $l \ge 1$, we define recursively
\[
\widetilde T_i^{(1)} := \widetilde T_i, \qquad
\widetilde T_i^{(l)} := 
\widetilde T_i\bigl(I_{E_i} \otimes \widetilde T_i^{(l-1)}\bigr).
\]
Equivalently,
\(
\widetilde T_i^{(l)}
=
\widetilde T_i\,
(I_{E_i} \otimes \widetilde T_i)\cdots
\bigl(I_{E_i^{\,l-1}} \otimes \widetilde T_i\bigr):E_i^{\,l} \otimes_\sigma \mathcal H
\longrightarrow
\mathcal H.
\)

The canonical single variable analogue of the classical Wold decomposition,
corresponding to the semigroup $\mathbb Z_+$, is the
\emph{Muhly--Solel Wold decomposition}.  It provides a structural splitting for isometric covariant representations of a single $C^*$-correspondence
\cite{MS}. First, recall that
a subspace $\mathcal{K} \subset \mathcal{H}$ is called \emph{reducing} for the 
representation $(\sigma, T)$ if it reduces 
$\sigma(\mathcal{A})$, and both 
$\mathcal{K}$ and $\mathcal{K}^{\perp}$ are left invariant under all operators 
$T_i(\xi_i)$ for $i \in I_k$ and $\xi_i \in E_i$.

\begin{theorem}\cite[Theorem 2.9]{MS}\label{Thm:Wold-decomposition-sigma-T}
Let $(\sigma,T)$ be an isometric covariant representation of a
$C^*$-correspondence $E$ over a $C^*$-algebra $\mathcal A$ on a Hilbert space $\mathcal H$.  Then there exists a unique orthogonal decomposition
\[
(\sigma,T)\;\cong\;(\sigma_1,T_1)\oplus(\sigma_2,T_2),
\qquad
\mathcal H=\mathcal H_1\oplus\mathcal H_2,
\]
such that
\begin{enumerate}
\item $(\sigma_1,T_1):=\restr{(\sigma,T)}{\mathcal H_1}$ is an \emph{induced covariant
representation};
\item $(\sigma_2,T_2):=\restr{(\sigma,T)}{\mathcal H_2}$ is \emph{fully coisometric}.
\end{enumerate}

The decomposition is unique in the sense that if a closed subspace
$\mathcal K\subseteq\mathcal H$ reduces $(\sigma,T)$ and
$\restr{(\sigma,T)}{\mathcal K}$ is induced (respectively, fully coisometric), then
$\mathcal K\subseteq\mathcal H_1$ (respectively, $\mathcal K\subseteq\mathcal
H_2$).
Moreover, the summands are given explicitly by
\[
\mathcal H_1
=
\overline{\operatorname{span}}
\big\{
\widetilde T^{(n)}(\xi\otimes w)
:\ \xi\in E^n,\ w\in W,\ n\in\mathbb Z_+
\big\}, \qquad \mathcal H_2
=
\bigcap_{n\in\mathbb Z_+}
\widetilde T^{(n)}\left(E^n\otimes_\sigma\mathcal H\right),
\]
where
\(
W:=\operatorname{Ran}(I-\widetilde T\widetilde T^*)
\)
and $\widetilde T^{(0)}:=\sigma$.

\end{theorem}

It serves as the foundational
one variable case for higher rank generalizations.
A major subsequent advance was obtained by Skalski and Zacharias
\cite{Skalski}, who established a Wold decomposition for doubly commuting
isometric representations of product systems of $C^*$-correspondences over the
semigroup $\mathbb Z_+^k$. 
The basic object of this work is a product system $\mathbb E$ of $C^*$-correspondences
over $\mathbb Z_+^k$~\cite{Fow}. As explained in~\cite{Solel,Solel2}, $\mathbb E$ may be viewed
as a family of $C^*$-correspondences $\{E_1,\ldots,E_k\}$ together with unitary
isomorphisms $t_{ij} : E_i \otimes E_j \longrightarrow E_j \otimes E_i, \, i>j,$
yielding canonical identifications,  $E^n\simeq
E_1^{n_1}\otimes_{\mathcal A}\cdots\otimes_{\mathcal A}E_k^{n_k},$ for $n=(n_1,\ldots,n_k)\in\mathbb Z_+^k.$
We further assume $t_{ii}=\mathrm{id}_{E_i\otimes E_i}$,
$t_{ij}=t_{ji}^{-1}$ for $i<j$, and 
the braid (hexagon) relations
\[
(I_{E_\ell}\otimes t_{ij})(t_{i\ell}\otimes I_{E_j})(I_{E_i}\otimes t_{j\ell})
=
(t_{j\ell}\otimes I_{E_i})(I_{E_j}\otimes t_{i\ell})(t_{ij}\otimes I_{E_\ell})
\]
for all distinct $i,j,\ell$.

We now pass from the algebraic structure of a product system to its analytic
realization via compatible covariant representations on Hilbert spaces.

\begin{definition}\cite[Definitons 1.2, 2.1]{Skalski}
\label{def: rep of E}
Let $\mathbb E$ be a product system over $\mathbb Z_+^k$.
By a (covariant completely contractive) representation of $\mathbb E$
on a Hilbert space $\mathcal H$ we mean a tuple $(\sigma, T_1,\ldots,T_k),$
where $(\sigma,\mathcal H)$ is a representation of $\mathcal A$, and
$T_i:E_i\to\mathcal B(\mathcal H)$ are linear completely contractive maps
satisfying
\[
T_i(a\,\xi_i\,b)=\sigma(a)T_i(\xi_i)\sigma(b),
\qquad a,b\in\mathcal A,\ \xi_i\in E_i,
\]
and for all $1 \leq i \neq j \leq k,$
\begin{equation}\label{eq:twisted-relation}
\widetilde T_i(I_{E_i}\otimes \widetilde T_j)
=
\widetilde T_j(I_{E_j}\otimes \widetilde T_i)(t_{ij}\otimes I_{\mathcal H}).
\end{equation}
Such a representation is called \emph{isometric} if each $(\sigma,T_i)$
is isometric as a representation of $E_i$, and \emph{fully coisometric} if each
$(\sigma,T_j)$ is fully coisometric.

Further, we say that the representation $(\sigma, T_1, \ldots, T_k)$ is \emph{doubly commuting} if 
for every distinct $i,j \in I_k$,
\begin{equation}\label{eq: doubly commuting rel}
\widetilde{T_j}^{*}\,\widetilde{T_i}
= 
(I_{E_j} \otimes \widetilde{T_i})\,
(t_{ij} \otimes I_{\mathcal{H}})\,
(I_{E_i} \otimes \widetilde{T_j}^{*}).
\end{equation}
\end{definition}

\begin{remark}
$T_\xi$ denotes the creation operator on $\mathcal{F}(E)$, whereas for $i \in I_k$, $T_i:E_i\to \mathcal{B}(\clh)$ are the maps in a covariant representation; these are unrelated symbols.    
\end{remark} 
We now extend the notion of the Fock module associated with a $C^*$-correspondence
$E$ to the setting of product systems.
For a product system of
$C^*$-correspondences $\mathbb E=\{E^n\}_{n\in\mathbb Z_+^k}$ over $\mathcal A$ and
$A\subseteq I_k$, we denote by
$\mathbb E_A=\{E_A^n\}_{n\in\mathbb Z_+^{|A|}}$
the product subsystem generated by $\{E_i:i\in A\}$.
The associated \emph{Fock module} is
\[
\mathcal F(E_A):=\bigoplus_{n\in\mathbb Z_+^{|A|}}E_A^n,
\qquad E_A^0:=\mathcal A.
\]
For each $n\in\mathbb Z_+^{|A|}$, the left action
$\varphi^{(n)}:\mathcal A\to\mathcal L(E_A^n)$
is defined recursively by $\varphi^{(0)}(a)b:=ab$ on $\mathcal A$; and for $n\neq 0$, choose $i\in A$ with $n_i>0$ and identify
$E_A^n\cong E_i\otimes_{\mathcal A}E_A^{n-e_i}$, then set $\varphi^{(n)}(a)(\eta_i\otimes\zeta)
   := (\varphi_i(a)\eta_i)\otimes\zeta,
   \, a\in\mathcal A.$
This definition is independent of the choice of $i$ since each $t_{ij}$ is a bimodule isomorphism.
The induced \emph{Fock left action} of $\mathcal A$ on $\mathcal F(E_A)$ is
\[
\varphi_\infty(a):=\bigoplus_{n\in\mathbb Z_+^{|A|}}\varphi^{(n)}(a),
\qquad a\in\mathcal A.
\]
Now we extend the notion of reducing subspace to the 
representation $(\sigma, T_1, \ldots, T_k)$. First recall that given a Hilbert space $\mathcal H$ and a subset $\mathcal S \subseteq \mathcal B(\mathcal H)$,
we define the \emph{commutant} of $\mathcal S$ to be
\[
\mathcal S' = \{\, X \in \mathcal B(\mathcal H) : XS = SX \text{ for every } S \in \mathcal S \,\}.
\]
\begin{definition}\label{def:reducing subspace}
A subspace $\mathcal{K} \subset \mathcal{H}$ is called \emph{reducing} for the 
representation $(\sigma, T_1, \ldots, T_k)$ if it reduces 
$\sigma(\mathcal{A})$ (so that the projection onto $\mathcal{K}$, denoted by 
$P_{\mathcal{K}}$, lies in $\sigma(\mathcal{A})'$), and both 
$\mathcal{K}$ and $\mathcal{K}^{\perp}$ are left invariant under all operators 
$T_i(\xi_i)$ for $i \in I_k$ and $\xi_i \in E_i$. 
\end{definition}
It is easy to see that the obvious restriction procedure yields a 
representation of $E$ on $\mathcal{K}$, which is called a 
\emph{summand} of $(\sigma, T_1, \ldots, T_k)$ and will be denoted by
\(
\restr{(\sigma, T_1, \ldots, T_k)}{\mathcal{K}}.
\)
The following decomposition generalizes the single variable Wold decomposition~\cite{MS} and 
provides the canonical model for multivariable isometries satisfying the 
doubly commuting relations.

\begin{theorem}\cite[Theorem 2.4] {Skalski}\label{thm:doubly_commuting_wold}
Every doubly commuting isometric representation 
$(\sigma, T_1, \ldots, T_k)$ 
of a product system $\mathbb{E}$ on a Hilbert space $\mathcal{H}$ 
admits a unique orthogonal decomposition
\(
\mathcal{H} = \bigoplus_{\alpha \subseteq I_k} \mathcal{H}_\alpha
\)
such that, for each subset $\alpha = \{\alpha_1, \ldots, \alpha_r\} \subseteq I_k$:
\begin{enumerate}
    \item $\mathcal{H}_\alpha$ reduces $(\sigma, T_1, \ldots, T_k)$;
    \item the restricted tuple 
          $\restr{(\sigma, T_{\alpha_1}, \ldots, T_{\alpha_r})}{\mathcal{H}_\alpha}$ 
          is isomorphic to an induced representation of the 
          product subsystem 
          $\mathbb{E}_\alpha$ over $\mathbb{Z}_+^r$, 
          generated by the correspondences 
          $E_{\alpha_1}, \ldots, E_{\alpha_r}$ 
          and the associated isomorphisms 
          $t_{\alpha_i,\alpha_j}$;
    \item for each $i \in I_k \setminus \alpha$, 
          the representation 
          $\restr{(\sigma, T_i)}{\mathcal{H}_\alpha}$ 
          of $E_i$ is fully coisometric.
\end{enumerate}
\end{theorem}

A.~Talker~\cite{Talker} introduced the notion of $\lambda$--doubly commuting row isometries and established the existence of their unitary extensions using a
direct--limit construction.  Although our setting is substantially different, the underlying method—namely, the use of inductive limits of Hilbert spaces—is particularly effective.  We adapt this technique to the framework of
product systems and use it in Section~\ref{Unitary Extension} to construct
unitary extensions for various classes of representations.

First we recall a categorical construction of the direct limit of Hilbert spaces \cite{Talker}.  
Let \( \clh = \{\clh_i\}_{i \in I} \) be a family of Hilbert spaces indexed by a non-empty directed set \( I \).  
Assume that for all \( i, j \in I \) with \( i \le j \) there is an isometry 
\(\varphi_{ji} : \clh_i \to \clh_j\) such that whenever \(i \le j \le k\),
\[
\varphi_{kj} \circ \varphi_{ji} = \varphi_{ki}, \quad \text{and} \quad \varphi_{ii} = \mathrm{id}_{\clh_i}.
\]  
The direct limit \((\clh_\infty, \psi_i)\) is a Hilbert space with isometries \(\psi_i : \clh_i \to \clh_\infty\) which satisfy
\[
\psi_j \circ \varphi_{ji} = \psi_i,
\]
and has the following universal property: for every family of contractions \(\{\pi_i\}_{i \in I}\) from \(\{\clh_i\}_{i \in I}\) into a Hilbert space \(\clk\) such that $\pi_j \circ \varphi_{ji} = \pi_i,$
there exists a unique contraction \(\pi_\infty : \clh_\infty \to \clk\) such that
\[
\pi_i = \pi_\infty \circ \psi_i \quad \forall i \in I.
\]
\begin{prop}\cite[Proposition 3.3.2]{Talker}\label{prop: density in H infty}
Let \(\clh_\infty\) be the direct limit of the direct system \((\clh_i, \varphi_{ji})\).  
Then the set $\displaystyle\bigcup_{i \in I} \psi_i(\clh_i)$
is dense in \(\clh_\infty\).
\end{prop}

\begin{lemma}\cite[Lemma 3.3.3--3.3.4]{Talker}
\label{lemma: existence of Vinfnty}
Let \(\clh_\infty\) be the direct limit of the direct system \((\clh_i, \varphi_{ji})\).  
Assume that for all \(i \in I\) there exists a contraction 
\(V_i : \clh_i \to \clh_i\) such that
$V_j \circ \varphi_{ji} = \varphi_{ji} \circ V_i, \quad \text{for all } i \le j.$
Then there exists a contraction 
\(V_\infty : \clh_\infty \to \clh_\infty\) such that
$V_\infty \circ \psi_i = \psi_i \circ V_i, \quad \text{for all } i \in I.$
\medskip
Moreover, if each \(V_i\) is an isometry, then \(V_\infty\) is also an isometry.
\end{lemma}

\section{Characterization of the Wold Decomposition} \label{Characterization for Existence of Wold Decomposition}

The Wold-type theorems of Muhly--Solel~\cite{MS} 
(for single $C^*$\nobreakdash-correspondences) 
and of Skalski--Zacharias~\cite{Skalski} 
(for product systems) establish that 
every isometric representation decomposes uniquely into 
induced and fully coisometric parts.  
What remains open is to determine an intrinsic 
operator-theoretic condition guaranteeing the existence of such a decomposition.

In this section, we provide a necessary and sufficient criterion 
for the existence of a Wold decomposition for an isometric representation 
$(\sigma,T_1,\ldots,T_n)$ of a family of $C^*$\nobreakdash-correspondences 
$E_1,\ldots,E_n$, without imposing any a priori relations 
(such as double commutativity) among them.  
Our approach extends the classical result of~\cite{RSS2}, 
where it was shown that for a tuple $V=(V_1,\ldots,V_n)$ of isometries 
on a Hilbert space~$\mathcal H$, 
the existence of a von Neumann–Wold decomposition is equivalent to 
the mutual reducibility of the unitary (or shift) parts of the individual 
isometries.

Replacing $V_i$ by the covariant pairs $(\sigma,T_i)$, 
we demonstrate that the same principle governs the $C^*$\nobreakdash-correspondence 
and product-system settings:  
a Wold decomposition exists exactly when 
the induced and fully coisometric components of each 
$(\sigma,T_i)$ reduce every other $(\sigma,T_j)$.

\smallskip
To motivate this characterization, recall that in the one variable case, every 
isometric covariant representation $(\sigma,T)$ of a $C^*$-correspondence 
admits a unique orthogonal decomposition into an induced part and a fully coisometric part 
(Theorem~\ref{Thm:Wold-decomposition-sigma-T}). 
In several variables, the existence of such a decomposition depends on how the individual representations $(\sigma,T_i)$ interact. We now introduce a notion of Wold decomposition that isolates the structural property required for a representation to admit a Wold-type decomposition. This definition generalizes Definition~1.2 of~\cite{RSS2} to the present setting
(see also~\cite{Popovici10, Slocinski80, Wold}).

\begin{definition}\label{Wold-Decomposition}
Let $E_1, E_2, \cdots, E_k$ be $C^*$\nobreakdash-correspondences over a $C^*$\nobreakdash-algebra~$\mathcal{A}$. 
For each $i \in I_k$, let $(\sigma, T_i)$ be a (nondegenerate) covariant representation 
of $E_i$ on a Hilbert space~$\mathcal H$, where $\sigma : \mathcal{A} \to \mathcal B(\mathcal H)$ 
is a fixed $*$\nobreakdash-representation.  
We say that the family $(\sigma, T_1, \ldots, T_k)$ \emph{admits a Wold decomposition} 
if for each subset $A \subseteq I_k$, there exists a closed subspace $\mathcal H_A$ of $\mathcal H$ (some of these subspaces may be trivial) such that
\begin{enumerate}
    \item $\displaystyle \mathcal H = \bigoplus_{A \subseteq I_k} \mathcal H_A$;
    \item each $\mathcal H_A$ reduces $(\sigma, T_i)$ for all $i \in I_k$;
    \item $\restr{(\sigma, T_i)}{\mathcal H_A}$ is an induced representation for $i \in A$, and $\restr{(\sigma, T_j)}{\mathcal H_A}$ is fully coisometric for $j \in A^c$.
\end{enumerate}
\end{definition}

For $i \in I_k$, let $(\sigma, T_i)$ be  an isometric covariant representation of a single correspondence $E_i$ on~$\mathcal{H}$.
The Muhly--Solel Wold decomposition yields
\[
\mathcal{H} = \mathcal{H}_i^1 \oplus \mathcal{H}_i^2,
\]
where $\restr{(\sigma, T_i)}{\mathcal{H}_i^1}$ is an induced isometric covariant representation 
and $\restr{(\sigma, T_i)}{\mathcal{H}_i^2}$ is fully coisometric.
We begin by identifying a simple but powerful criterion that 
characterizes all representations admitting a Wold decomposition.

\begin{theorem}\label{thm: class of decom}
Let $E_1, E_2, \ldots, E_n$ be $C^*$\nobreakdash-correspondences over a 
$C^*$\nobreakdash-algebra~$\mathcal{A}$. 
For each $i \in I_n$, let $(\sigma, T_i)$ be an isometric covariant representation 
of $E_i$ on a Hilbert space~$\mathcal{H}$, where 
$\sigma : \mathcal{A} \to \mathcal{B}(\mathcal{H})$ is a fixed nondegenerate 
$*$\nobreakdash-representation. 
The following statements are equivalent:
\begin{enumerate}
    \item The family $(\sigma, T_1, T_2, \ldots, T_n)$ admits a Wold decomposition.
    \item For all $i, j \in I_n$, the subspace $\mathcal{H}_i^1$ reduces the covariant pair $(\sigma, T_j)$.
    \item For all $i, j \in I_n$, the subspace $\mathcal{H}_i^2$ reduces the covariant pair $(\sigma, T_j)$.
\end{enumerate}
\end{theorem}

\begin{proof}
It suffices to establish the equivalence of \textup{(1)} and \textup{(2)}. 
Suppose that $(\sigma, T_1, \ldots, T_n)$ admits a Wold decomposition. 
Then, for a fixed $i \in I_n$, we have
\[
\mathcal H
= \bigoplus_{A\subseteq I_n} \mathcal H_A
= \left(\bigoplus_{\substack{B\subseteq I_n\\ i\in B}} \mathcal H_B\right)
  \oplus
  \left(\bigoplus_{\substack{B\subseteq I_n\\ i\notin B}} \mathcal H_B\right)
= \mathcal H_i^1 \oplus \mathcal H_i^2.
\]
By Definition~\ref{Wold-Decomposition}, each $\mathcal H_B$ reduces $(\sigma, T_1, \ldots, T_n)$, and hence so does their direct sum.
Furthermore, for $i\in B$, the restriction $\restr{(\sigma, T_i)}{\mathcal H_B}$ is induced. 
Therefore, $\restr{(\sigma, T_i)}{\bigoplus_{\substack{B\subseteq I_n\\ i\in B}} \mathcal H_B}$ is an (possibly infinite) orthogonal direct sum of induced representations,
and hence itself induced. 
Similarly, for $i\notin B$, $\restr{(\sigma, T_i)}{\bigoplus_{\substack{B\subseteq I_n\\ i\notin B}} \mathcal H_B}$ is fully coisometric. 
Consequently, the uniqueness of the Wold decomposition for $(\sigma, T_i)$ implies that
\[
\mathcal H_i^1 
= \bigoplus_{\substack{B\subseteq I_n\\ i\in B}} \mathcal H_B,
\qquad
\mathcal H_i^2
= \bigoplus_{\substack{B\subseteq I_n\\ i\notin B}} \mathcal H_B.
\]
Since each $\mathcal{H}_A$ reduces $(\sigma, T_j)$ for every $j \in I_n$, 
it follows that $\mathcal{H}_i^1$ (and hence $\mathcal{H}_i^2$) reduces $(\sigma, T_j)$ for all $i, j \in I_n$.

\smallskip
Conversely, assume that $\mathcal{H}_i^2$ (and hence $\mathcal{H}_i^1$) reduces $(\sigma, T_j)$ for all $i, j \in I_n$. 
We claim that the family $(\sigma, T_1, T_2, \ldots, T_n)$ admits a Wold decomposition.
For each $A \subseteq I_n$, define
\begin{equation}\label{eqn: rep of Wold dec H A}
    \mathcal{H}_A 
    = \left(\bigcap_{i \in A} \mathcal{H}_i^1\right) 
      \cap 
      \left(\bigcap_{j \in A^c} \mathcal{H}_j^2\right).
\end{equation}
Then each $\mathcal{H}_A$ reduces $(\sigma, T_i)$ for all $i \in I_n$. 
Moreover, by the structure of $\mathcal{H}_A$, it is clear that for $i \in A$, the restriction $\restr{(\sigma, T_i)}{\mathcal{H}_A}$ is induced, 
whereas for $j \in A^c$, the restriction $\restr{(\sigma, T_j)}{\mathcal{H}_A}$ is fully coisometric.  
It remains to show that 
\[
\mathcal{H} = \bigoplus_{A \subseteq I_n} \mathcal{H}_A.
\]
Since $\displaystyle\bigoplus_{A \subseteq I_n} \mathcal{H}_A \subseteq \mathcal{H}$, 
we only need to prove the reverse inclusion.

\smallskip
We claim that for any $A \subseteq I_m \subsetneq I_n$ and $j \notin I_m$, 
if $\widetilde{A}$ denotes $A$ viewed as a subset of $I_m \cup \{j\}$, then
\[
\mathcal{H}_A 
    \subseteq 
    \mathcal{H}_{\widetilde{A} \cup \{j\}} \oplus \mathcal{H}_{\widetilde{A}}.
\]
Since $\mathcal{H}_A$ reduces $(\sigma, T_j)$, 
we may decompose the restriction $\restr{(\sigma, T_j)}{\mathcal{H}_A}$ as
\[
\mathcal{H}_A = \mathcal{H}_{A,j}^1 \oplus \mathcal{H}_{A,j}^2,
\]
where $\restr{(\sigma, T_j)}{\mathcal{H}_{A,j}^1}$ is induced 
and $\restr{(\sigma, T_j)}{\mathcal{H}_{A,j}^2}$ is fully coisometric. 
Then
\[
\mathcal{H}_{A,j}^1 
    \subseteq \mathcal{H}_A \cap \mathcal{H}_j^1 
    \subseteq \mathcal{H}_{\widetilde{A} \cup \{j\}},
\qquad
\mathcal{H}_{A,j}^2 
    \subseteq \mathcal{H}_A \cap \mathcal{H}_j^2 
    \subseteq \mathcal{H}_{\widetilde{A}}.
\]
Hence, the claim follows.
Applying this observation repeatedly to 
$A \subseteq I_m \subsetneq I_n$ and indices $j, k \notin I_m$, we obtain
\[
\mathcal{H}_A 
    \subseteq 
    \mathcal{H}_{\widetilde{A} \cup \{j,k\}}
        \oplus \mathcal{H}_{\widetilde{A} \cup \{j\}}
        \oplus \mathcal{H}_{\widetilde{A} \cup \{k\}}
        \oplus \mathcal{H}_{\widetilde{A}},
\]
where $\widetilde{A}=A$, regarded as a subset of $I_m \cup \{j,k\}$.

\smallskip
Now consider the decomposition corresponding to $(\sigma, T_1)$, and set 
$\mathcal{H}_{\{1\}} := \mathcal{H}_1^1$ and $\mathcal{H}_{\emptyset} := \mathcal{H}_1^2$. 
Then 
\[
\mathcal{H} = \mathcal{H}_{\{1\}} \oplus \mathcal{H}_{\emptyset},
\]
where $\{1\}$ and $\emptyset$ are subsets of $I_1=\{1\}$. 
Applying the above argument recursively to $\mathcal{H}_{\{1\}}$, we obtain
\[
\mathcal{H}_{\{1\}} 
    \subseteq 
    \bigoplus_{A \subseteq J} \mathcal{H}_{\widetilde{\{1\}} \cup A},
\]
where $J=\{2,\ldots,n\}$ and $\widetilde{\{1\}}=\{1\}$ viewed as a subset of $I_n$. 
Similarly,
\[
\mathcal{H}_{\emptyset} 
    \subseteq 
    \bigoplus_{A \subseteq J} \mathcal{H}_A.
\]
Therefore,
\[
\begin{split}
\mathcal{H} 
    &= \mathcal{H}_{\{1\}} \oplus \mathcal{H}_{\emptyset} \\[4pt]
    &\subseteq 
        \left(\bigoplus_{A \subseteq J} 
            \mathcal{H}_{\widetilde{\{1\}} \cup A}\right) 
        \oplus 
        \left(\bigoplus_{A \subseteq J} \mathcal{H}_A\right) \\[4pt]
    &= \bigoplus_{A \subseteq I_n} \mathcal{H}_A,
\end{split}
\]
and this completes the proof.
\end{proof}

\section{Doubly Twisted Representations} \label{Doubly Twisted Representations}
We now introduce our main object of study the \emph{doubly twisted representations} of a product system. As an application of the characterization theorem established in the preceding section, we prove the existence of a Wold-type decomposition for this family of representations. In particular, we recover the classical decomposition for doubly commuting representations. 

\begin{remark} \label{rem: tensor identities}
    We begin by recording a collection of elementary identities that will be used 
frequently throughout the paper. Their proofs follow directly from the 
functoriality of the spatial tensor product and from the definitions of the maps 
$\widetilde{T}_i$ associated with a covariant representation 
$(\sigma, T_1, \ldots, T_k)$ of a product system:
\begin{enumerate}
\item
Let $\mathcal H_1, \mathcal H_2, \mathcal H_3$, and $\mathcal K$ be Hilbert spaces. 
If $B \in \mathcal B(\mathcal H_2, \mathcal H_3)$ and 
$C \in \mathcal B(\mathcal H_1, \mathcal H_2)$, then
\[
I_{\mathcal K} \otimes (BC)
=
(I_{\mathcal K} \otimes B)(I_{\mathcal K} \otimes C).
\]

\item
Let $\mathbb E$ be a product system over $\mathbb Z_+^k$, and let 
$(\sigma, T_1, \ldots, T_k)$ be a (completely contractive) covariant representation 
of $\mathbb E$ on a Hilbert space~$\mathcal H$. 
For any unitary isomorphism
\(
t_{ji} : E_j \otimes_{\mathcal A} E_i 
\longrightarrow 
E_i \otimes_{\mathcal A} E_j,
\, i \neq j,
\)
and any operator $X \in \mathcal B(\mathcal H)$, one has
\[
(t_{ji} \otimes I_{\mathcal H})
(I_{E_j \otimes E_i} \otimes X)
=
(I_{E_i \otimes E_j} \otimes X)
(t_{ji} \otimes I_{\mathcal H}).
\]
In particular, this identity holds for 
$X = \widetilde T_i^{(n)} \widetilde T_i^{(n)\,*}$, $n \in \mathbb Z_+$. 

\end{enumerate}
\end{remark}

\begin{definition}\label{def:twisted}
Let $\mathbb{E} = \{E_i\}_{i=1}^k$ be a product system of 
$C^*$\nobreakdash-correspondences over $\mathcal A$, 
with canonical product--system isomorphisms $t_{ij} : E_i \otimes_{\mathcal A} E_j 
\longrightarrow 
E_j \otimes_{\mathcal A} E_i,
\, i \neq j.$
Let $(\sigma, \mathcal H)$ be a nondegenerate $*$\nobreakdash-representation of 
$\mathcal A$ on a Hilbert space~$\mathcal H$.
A \emph{(covariant completely contractive) twisted representation} of $\mathbb E$ 
on~$\mathcal H$ is a tuple $(\sigma, T_1, \ldots, T_k)$, where each
$T_i : E_i \to \mathcal B(\mathcal H)$ is a linear, completely contractive map 
such that for $ a,b \in \mathcal A,\ \xi_i \in E_i,$
    \[
    T_i(a\,\xi_i\,b) 
    = \sigma(a)\,T_i(\xi_i)\,\sigma(b).
    \]
Each $T_i$ induces a contraction
$\widetilde T_i:E_i\otimes_\sigma\mathcal H\to\mathcal H$,
$\widetilde T_i(\xi\otimes h)=T_i(\xi)h$, which satisfies the intertwining relation
\[
\widetilde{T_i}(\varphi_i(a)\xi \otimes h) = \sigma(a) \widetilde{T_i}(\xi \otimes h),
\, a \in \mathcal{A}.
\]
And for all $i,j \in I_k$ with $i \neq j$,
    \begin{equation}\label{eq:twisted}
    \widetilde T_i (I_{E_i} \otimes \widetilde T_j)
    =
    \widetilde T_j (I_{E_j} \otimes \widetilde T_i)
    \,(t_{ij} \otimes U_{ij}),
    \end{equation} 
where $\{U_{ij}\}_{i \neq j}$ is a family of unitaries on~$\mathcal H$ such that, 
for all $i \neq j$, all $\ell \in I_k$, and all $a \in \mathcal A$, we have
\[
U_{ji} = U_{ij}^*, \qquad U_{ij}\,\widetilde T_{\ell}
    = \widetilde T_{\ell}\,(I_{E_{\ell}} \otimes U_{ij}), \qquad
U_{ij}\,\sigma(a) = \sigma(a)\,U_{ij}.
\]

Such a representation is called \emph{isometric} if each pair $(\sigma, T_i)$ is an 
isometric covariant representation of $E_i$, and \emph{fully coisometric} if
$\widetilde T_i\,\widetilde T_i^{\,*} = I_{\mathcal H}$ for all $ i \in I_k.$
\end{definition}

\begin{definition}\label{def:doubly-twisted}
A twisted representation $(\sigma, T_1, \ldots, T_k)$ of $\mathbb E$ is said to be 
\emph{doubly twisted} if for all $i,j\in I_k$ with $i\neq j$,
\begin{equation}\label{eq:doubly}
\widetilde T_j^{\,*}\,\widetilde T_i
= (I_{E_j}\otimes \widetilde T_i)\,(t_{ij}\otimes U_{ij})\,
(I_{E_i}\otimes \widetilde T_j^{\,*}).
\end{equation}
\end{definition}

\begin{remark}\label{rem: range projections}
For a doubly twisted (completely contractive) representation, we observe the following.
\begin{enumerate}
    \item
    When each twisting unitary $U_{ij}$ is equal to the identity operator, 
    the above framework reduces to the classical setting of doubly commuting 
    representations (cf. \cite{Skalski}).

    \item
    For each $m \in I_k$ and $n \in \mathbb Z_+$, define
    \[
    P_m^{(n)} := \widetilde T_m^{(n)}\,\widetilde T_m^{(n)\,*}.
    \]
    Each operator $P_m^{(n)}$ is a positive contraction on $\mathcal H$ 
    (and in fact a projection whenever $(\sigma, T_m)$ is isometric), 
    representing the range projection of $\widetilde T_m^{(n)}$. 
    The family $\{P_m^{(n)}\}_{n \ge 0}$ forms a decreasing sequence of contractions; 
    in particular,
    \[
    P_m^{(0)} = I_{\mathcal H}
    \quad\text{and}\quad
    P_m^{(n+1)} \le P_m^{(n)} 
    \quad \text{for all } n \in \mathbb Z_+.
    \]
\end{enumerate}
\end{remark}

We now turn to the proof of the existence of a Wold decomposition for doubly twisted representations. 
As a first step, we verify certain commutation relations between the twisting unitaries and the range 
projections $P_m^{(n)}$. 
The following lemmas record these identities and form the technical backbone of the proof that follows.

\begin{lemma}\label{lemma:intertwining of U_{ij}}
Let $(\sigma, T_1, \ldots, T_k)$ be a doubly twisted representation of a product system $\mathbb{E}$ on a Hilbert space $\mathcal{H}$. 
Then for all $i,j,m \in I_k$ with $i\ne j$ and for every $n\in\mathbb Z_+$, we have
\begin{enumerate}
    \item $U_{ij}\, \widetilde{T}_m^{(n)}
= \widetilde{T}_m^{(n)}\, (I_{E_m^{\,n}} \otimes U_{ij}).$
    \item $U_{ij}\, P_m^{(n)} = P_m^{(n)}\, U_{ij}$, where $P_m^{(n)} := \widetilde{T}_m^{(n)}\widetilde{T}_m^{(n)*}$;
    
    \item $(t_{ij} \otimes U_{ij})\left(I_{E_i \otimes E_j} \otimes P_m^{(n)}\right)
    = \left(I_{E_j \otimes E_i} \otimes P_m^{(n)}\right)(t_{ij} \otimes U_{ij})$.
\end{enumerate}
\end{lemma}

\begin{proof}
\begin{enumerate}
\item We prove the identity by induction on $n$. 
For $n=1$, we already have $U_{ij}\, \widetilde{T}_m
= \widetilde{T}_m (I_{E_m} \otimes U_{ij}).$
Assume the identity for $n-1$, that is,
\[
U_{ij}\, \widetilde{T}_m^{(n-1)}
= \widetilde{T}_m^{(n-1)}\left(I_{E_m^{\,n-1}} \otimes U_{ij}\right).
\]
We now prove the result for $n$. By Remark~\ref{rem: tensor identities} together with the case $n=1$,
\[
U_{ij}\, \widetilde{T}_m^{(n)} 
= U_{ij}\, \widetilde{T}_m \left(I_{E_m} \otimes \widetilde{T}_m^{(n-1)}\right)
=\widetilde{T}_m \left(I_{E_m} \otimes U_{ij}\,\widetilde{T}_m^{(n-1)}\right).
\]
Using the induction hypothesis, the right hand side simplifies to $\widetilde T_m\left(I_{E_m}\otimes \widetilde T_m^{(n-1)}
\left(I_{E_m^{\,n-1}}\otimes U_{ij}\right)\right),$
and hence Remark \ref{rem: tensor identities} implies
\[
U_{ij}\,\widetilde T_m^{(n)}
=
\widetilde T_m^{(n)}\bigl(I_{E_m^{\,n}}\otimes U_{ij}\bigr).
\]

\item
From the relation $(1)$ above, we obtain
\[
(I_{E_m^{\,n}} \otimes U_{ij})\, \widetilde{T}_m^{(n)*}
= \widetilde{T}_m^{(n)*}\, U_{ij}.
\]
Using these two relations, we get
\[
U_{ij}\, \widetilde{T}_m^{(n)}\, \widetilde{T}_m^{(n)*}
= \widetilde{T}_m^{(n)} (I_{E_m^{\,n}} \otimes U_{ij})\, \widetilde{T}_m^{(n)*}
= \widetilde{T}_m^{(n)}\, \widetilde{T}_m^{(n)*}\, U_{ij}.
\]

\item The assertion follows directly from (2) above and Remark \ref{rem: tensor identities}(2).
\end{enumerate}
\end{proof}

With this lemma in hand, we now apply our characterization of the existence of the Wold decomposition (Theorem~\ref{thm: class of decom}) to the setting of doubly twisted representations of a product system.  
Consequently, this leads to a natural extension of the classical Wold decomposition for a single isometric covariant representation $(\sigma, T)$ to the present multi-variable twisted framework, as stated below.

\begin{theorem}\label{thm: decompo for doubly twisted rep}
Every doubly twisted isometric covariant representation 
$(\sigma, T_1, \ldots, T_k)$
of the product system~$\mathbb{E}$ on a Hilbert space~$\mathcal{H}$ admits a unique Wold decomposition.
\end{theorem}

\begin{proof}
It suffices to show that for all \(i, j \in I_k\), the subspace \(\mathcal{H}_2^i\) reduces \(T_j\); that is, $T_j(e)\mathcal{H}_2^i \subseteq \mathcal{H}_2^i$ and $ 
T_j(e)^*\mathcal{H}_2^i \subseteq \mathcal{H}_2^i$
for all $e \in E_j.$ As 
$\mathcal H_2^i :=\displaystyle \bigcap_{l\ge0} \operatorname{Ran}(\widetilde T_i^{(l)}\widetilde T_i^{(l)*}),$
it is enough to prove that for any $l \in \Z_+$,
\[
T_j(e)\widetilde{T}_i^{(l)} \widetilde{T}_i^{(l)*}(h)
= \widetilde{T}_i^{(l)} \widetilde{T}_i^{(l)*} T_j(e)(h),
\qquad \forall\, h \in \mathcal{H}.
\]
Thus it suffices to establish the operator identity
\begin{equation}\label{eq:TiTj commutation}
\widetilde{T}_j(I_{E_j} \otimes \widetilde{T}_i^{(l)} \widetilde{T}_i^{(l)*})
= \widetilde{T}_i^{(l)} \widetilde{T}_i^{(l)*} \widetilde{T}_j.
\end{equation}

\noindent
We proceed by induction on \(l\). For $l=1,$ the twisted relations ~\eqref{eq:twisted} and ~\eqref{eq:doubly} imply
\[
\widetilde{T}_j (I_{E_j} \otimes \widetilde{T}_i \widetilde{T}_i^*) 
= \widetilde{T}_i (I_{E_i} \otimes \widetilde{T}_j)(t_{ji} \otimes U_{ji})(I_{E_j} \otimes \widetilde{T}_i^*) 
= \widetilde{T}_i \widetilde{T}_i^* \widetilde{T}_j.
\]
Assume that \eqref{eq:TiTj commutation} holds for some \(l = n\), that is,
\[
\widetilde{T}_j(I_{E_j} \otimes \widetilde{T}_i^{(n)} \widetilde{T}_i^{(n)*})
= \widetilde{T}_i^{(n)} \widetilde{T}_i^{(n)*} \widetilde{T}_j.
\]
We show it holds for \(l = n+1\).  Using the recursive identity
\(\widetilde{T}_i^{(n+1)} = \widetilde{T}_i (I_{E_i} \otimes \widetilde{T}_i^{(n)})\),
we compute:
\begin{align*}
\widetilde{T}_j(I_{E_j} \otimes \widetilde{T}_i^{(n+1)} \widetilde{T}_i^{(n+1)*})
&= \widetilde{T}_j(I_{E_j} \otimes \widetilde{T}_i (I_{E_i} \otimes \widetilde{T}_i^{(n)} \widetilde{T}_i^{(n)*}) \widetilde{T}_i^*) \\
&= \widetilde{T}_j(I_{E_j} \otimes \widetilde{T}_i)
   (I_{E_j \otimes E_i} \otimes \widetilde{T}_i^{(n)} \widetilde{T}_i^{(n)*})
   (I_{E_j} \otimes \widetilde{T}_i^*) \\
&=
   \widetilde{T}_i (I_{E_i} \otimes \widetilde{T}_j)
   (t_{ji} \otimes U_{ji})
   (I_{E_j \otimes E_i} \otimes \widetilde{T}_i^{(n)} \widetilde{T}_i^{(n)*})
   (I_{E_j} \otimes \widetilde{T}_i^*) \quad\text{(by \eqref{eq:twisted})}
 \\
&=
   \widetilde{T}_i (I_{E_i} \otimes \widetilde{T}_j)
   (I_{E_i \otimes E_j} \otimes \widetilde{T}_i^{(n)} \widetilde{T}_i^{(n)*})
   (t_{ji} \otimes U_{ji})
   (I_{E_j} \otimes \widetilde{T}_i^*) \quad\text{(by Lemma \ref{lemma:intertwining of U_{ij}})}\\
&= \widetilde{T}_i (I_{E_i} \otimes \widetilde{T}_j (I_{E_j} \otimes \widetilde{T}_i^{(n)} \widetilde{T}_i^{(n)*}))
   (t_{ji} \otimes U_{ji})(I_{E_j} \otimes \widetilde{T}_i^*) \\
&=
   \widetilde{T}_i (I_{E_i} \otimes \widetilde{T}_i^{(n)} \widetilde{T}_i^{(n)*} \widetilde{T}_j)
   (t_{ji} \otimes U_{ji})(I_{E_j} \otimes \widetilde{T}_i^*) \quad \text{(induction hypothesis)} \\
&= \widetilde{T}_i (I_{E_i} \otimes \widetilde{T}_i^{(n)})(I_{E_i} \otimes \widetilde{T}_i^{(n)*})
   (I_{E_i} \otimes \widetilde{T}_j)(t_{ji} \otimes U_{ji})(I_{E_j} \otimes \widetilde{T}_i^*) \\
&= \widetilde{T}_i^{(n+1)} \widetilde{T}_i^{(n+1)*} \widetilde{T}_j  \quad \text{ (by \eqref{eq:doubly})}.
\end{align*}
This completes the induction. Therefore, by Theorem~\ref{thm: class of decom}, we conclude that every doubly twisted isometric representation admits a unique decomposition.
\end{proof}

When each twisting unitary $U_{ij}$ is the identity operator, 
Theorem~\ref{thm: decompo for doubly twisted rep} immediately reduces to the case of 
doubly commuting representations. 
Hence, we recover the following classical result by Skalski and Zacharias~\cite{Skalski} 
for doubly commuting isometric representations as a corollary of the above theorem. 
\begin{cor}\label{thm:skalski-zacharias-decomposition}
  Every doubly commuting isometric covariant representation $(\sigma, T_1, \ldots, T_k)$
of a product system~$\mathbb{E}$ on a Hilbert space~$\mathcal{H}$ 
admits a Wold decomposition.  
\end{cor}

One can also obtain the above result as a direct consequence of the characterization theorem~\ref{thm: class of decom}.
It is important to note that the notion of a Wold decomposition used in our framework
is formulated coordinatewise: for each $i \in A \subseteq I_k$,
the pair $\restr{(\sigma, T_i)}{\mathcal{H}_A}$ acts as an induced covariant representation of $E_i$.
In contrast, Skalski and Zacharias~\cite{Skalski} describe their decomposition in terms 
of the subsystems $\mathbb{E}_\alpha$ generated by the correspondences 
$E_{\alpha_1}, \dots, E_{\alpha_r}$, and state that the restriction $\restr{(\sigma, T^{(\alpha_1)}, \dots, T^{(\alpha_r)})}{\mathcal{H}_\alpha}$
is isomorphic to an induced representation of the product subsystem $\mathbb{E}_\alpha$
over $\mathbb{Z}_+^r$.
When the twisting unitaries $U_{ij}$ are all identities, these two notions coincide: 
an induced representation of the subsystem $\mathbb{E}_\alpha$ 
is precisely the tensor product of the single variable induced representations 
$(\sigma, T_i)$ for $i \in \alpha$. That is, the Fock module factorizes canonically as
\[
\mathcal F(\mathbb E_\alpha)
\;\cong\;
\mathcal F(E_{\alpha_1})\otimes\cdots\otimes\mathcal F(E_{\alpha_r}),
\]
where $\otimes$ denotes the internal tensor product of Hilbert 
$\mathcal A$–modules.
Under this identification, the $i$-th creation operator on 
$\mathcal F(\mathbb E_\alpha)$ corresponds to the creation operator on the 
$i$-th factor. 
Hence the induced representation of the subsystem $\mathbb E_\alpha$ from $\sigma$ 
is unitarily equivalent to the tensor product of the single variable induced 
representations:
\[
\mathbb E_\alpha(\sigma)
\;\cong\;
E_{\alpha_1}(\sigma)\otimes\cdots\otimes E_{\alpha_r}(\sigma).
\]
Thus, our formulation recovers the Skalski--Zacharias decomposition as a special case of the general twisted setting.

\section{Examples} \label{Examples}
In this section, we construct explicit non-trivial examples of twisted and
doubly twisted representations arising from operator families. Recall that, for $i \neq j \in I_k$, $t_{ij} : E_i \otimes E_j \;\longrightarrow\; E_j \otimes E_i,$
flips the two tensor factors.
For any $n \in \Z_+$, to interchange the tensor blocks $E_i \otimes E_j^{\,n}$ and 
$E_j^{\,n} \otimes E_i$, 
we define an isomorphism
\(
t_{ij}^{(n)} : 
E_i \otimes E_j^{\,n}
\;\longrightarrow\;
E_j^{\,n} \otimes E_i
\) recursively by
\begin{equation}\label{eq:tijn_interchanged}
t_{ij}^{(n)}
\;:=\;
\prod_{k=1}^{n}
\left(
I_{E_j^{\,n-k}} \otimes t_{ij} \otimes I_{E_j^{\,k-1}}
\right).
\end{equation}

\begin{lemma}\label{lemma:recursive-def-tij-n}
Let $1 \le i \ne j \le k$. Then the maps
$t_{ij}^{(n)} \colon E_i \otimes E_j^{\,n} \longrightarrow E_j^{\,n} \otimes E_i$
satisfy the recursive relation
\[
t_{ij}^{(n+1)}
=
\left(I_{E_j^{\,n}} \otimes t_{ij}\right)
\left(t_{ij}^{(n)} \otimes I_{E_j}\right),
\qquad n \ge 1.
\]
\end{lemma}

\begin{proof}
By definition,
\begin{align*}
 t_{ij}^{(n+1)}
&=\prod_{k=1}^{n+1}
\left(I_{E_j^{\,n+1-k}}\otimes t_{ij}\otimes I_{E_j^{\,k-1}}\right) =\left(I_{E_j^{\,n}}\otimes t_{ij}\right)
\prod_{k=1}^{n}
\left(I_{E_j^{\,n-k}}\otimes t_{ij}\otimes I_{E_j^{\,k-1}}\right)\\
&=\left(I_{E_j^{\,n}}\otimes t_{ij}\right)
\left(t_{ij}^{(n)}\otimes I_{E_j}\right)  \end{align*}
which is the desired recursion.
\end{proof}

\begin{remark}
The lemma shows that the family $\{t_{ij}^{(n)} \}_{n\geq 1}$ satisfies a natural recursion: 
each higher power $t_{ij}^{\,(n+1)}$ is obtained by adjoining $t_{ij}$ on the left of $t_{ij}^{(n)}$, up to a suitable tensoring with identity operators.  
This identity is often used inductively to manipulate products involving the intertwiners $t_{ij}$.
\end{remark}
In general, for arbitrary $m,n \in \mathbb{Z_+}$, we define \(
t_{ij}^{(m,n)} : 
E_i^{\,m} \otimes E_j^{\,n}
\;\longrightarrow\;
E_j^{\,n} \otimes E_i^{\,m}
\) by
\begin{equation}\label{eq:tijmn}
t_{ij}^{(m,n)}
\;:=\;
\prod_{k=1}^{m}
\left(
I_{E_i^{\,k-1}} \otimes t_{ij}^{(n)} \otimes I_{E_i^{\,m-k}}
\right).
\end{equation}
The maps $t_{ij}^{(m,n)}$ may also be defined recursively in the
second index.  Namely, for $n\ge1$,
\begin{equation} \label{eq: sec def of tijmn}
   t_{ij}^{(m,n)}
\;=\;
\left(I_{E_j^{\,n-1}} \otimes t_{ij}^{(m,1)}\right)
\;\left(t_{ij}^{(m,n-1)} \otimes I_{E_j}\right)
:\;
E_i^{\,m} \otimes E_j^{\,n}
\,\longrightarrow\,
E_j^{\,n} \otimes E_i^{\,m}. 
\end{equation}

For the iterated family $\{t_{ij}^{(n)}\}$, we now establish the corresponding braid (or block hexagon) relations.
\begin{prop}\label{prop: braid at nth level}
For all distinct $i>j>\ell$ and all $n\ge 1$, the iterated flips
$t_{i\ell}^{(n)}:E_i\otimes E_\ell^{\,n}\to E_\ell^{\,n}\otimes E_i$ and
$t_{j\ell}^{(n)}:E_j\otimes E_\ell^{\,n}\to E_\ell^{\,n}\otimes E_j$
satisfy
\[
\left(I_{E_\ell^{\,n}}\otimes t_{ij}\right)\,
\left(t_{i\ell}^{(n)}\otimes I_{E_j}\right)\,
\left(I_{E_i}\otimes t_{j\ell}^{(n)}\right)
\;=\;
\left(t_{j\ell}^{(n)}\otimes I_{E_i}\right)\,
\left(I_{E_j}\otimes t_{i\ell}^{(n)}\right)\,
\left(t_{ij}\otimes I_{E_\ell^{\,n}}\right),
\]
as maps
\(
E_i\otimes E_j\otimes E_\ell^{\,n}\to E_\ell^{\,n}\otimes E_j\otimes E_i.
\)
\end{prop}

\begin{proof}
We proceed by induction on $n$.
For $n=1$, the identity reduces to the usual hexagon relation for the
product--system isomorphisms.
Assume that the identity holds for some $n \ge 1$.
We prove it for $n+1$.
Write
\[
E_i \otimes E_j \otimes E_\ell^{\,n+1}
=
\bigl(E_i \otimes E_j \otimes E_\ell^{\,n}\bigr)\otimes E_\ell.
\]
Applying the induction hypothesis to the factor
$E_i \otimes E_j \otimes E_\ell^{\,n}$ yields the commuting block on the left-hand
side of the diagram below.
Applying the case $n=1$ to the remaining factor
$E_i \otimes E_j \otimes E_\ell$ produces the commuting block on the right.
Combining these two blocks and applying Lemma \ref{lemma:recursive-def-tij-n} completes the proof.
\[
\begin{tikzcd}[row sep=2em, column sep=2em]
E_i \otimes E_j \otimes E_\ell^{(n+1)}
\ar[r]
\ar[d]
&
E_\ell^{(n)} \otimes E_i \otimes E_j \otimes E_\ell
\ar[r]
\ar[d]
&
E_\ell^{(n)} \otimes E_i \otimes E_\ell \otimes E_j
\ar[r]
&
E_\ell^{(n+1)} \otimes E_i \otimes E_j
\ar[d]
\\
E_j \otimes E_i \otimes E_\ell^{(n+1)}
\ar[r]
&
E_\ell^{(n)} \otimes E_j \otimes E_i \otimes E_\ell
\ar[r]
&
E_\ell^{(n)} \otimes E_j \otimes E_\ell \otimes E_i
\ar[r]
&
E_\ell^{(n+1)} \otimes E_j \otimes E_i
\end{tikzcd}
\]

\end{proof}
Let us introduce a convenient notation. 
For remaining section, let us fix a set \(A=\{i_1<i_2<\cdots<i_p\}\subseteq I_k\) 
and a multi-index \(n=(n_1,\dots,n_p)\in\mathbb Z_+^{|A|}\).
For this fixed \(A\) and \(n\), and for \(1\le j\le m\le p\), 
we define (omitting explicit reference to \(A\) and \(n\) hereafter):
\begin{equation}\label{eq:notational-convenience}
\begin{aligned}
E_{j,m}
&:= E_{i_j}^{\,n_j}\otimes E_{i_{j+1}}^{\,n_{j+1}}\otimes\cdots
   \otimes E_{i_{m-1}}^{\,n_{m-1}}\otimes E_{i_m}^{\,n_m},
   \qquad 
   E_{j,j}:=E_{i_j}^{\,n_j},\\[4pt]
\widetilde{T}_{j,m}
&:= \widetilde{T}_{i_j}^{(n_j)}
   \left(I_{E_{j,j}}\otimes \widetilde{T}_{i_{j+1}}^{(n_{j+1})}\right)
   \cdots
   \left(I_{E_{j,m-1}}\otimes \widetilde{T}_{i_m}^{(n_m)}\right)
   : E_{j,m}\otimes\mathcal H\longrightarrow\mathcal H.
\end{aligned}
\end{equation}
When \(j=1\) and \(m=p\), we use the standard shorthand
\(\widetilde{T}_A^{(n)} (= \widetilde{T}_{1,p}\))
and \(E_A^{n} (= E_{1,p})\).
For any $l \in I_k,$ and $i_j \in A,$
define $\theta_{l,i_j}^{(n_j)}
\;:\;
E_{1,\,j-1} \otimes E_{l} \otimes E_{i_j}^{\,n_j} 
\otimes E_{\,j+1,\,p}
\;\longrightarrow\;
E_{1,\,j} \otimes E_{l} 
\otimes E_{\,j+1,\,p}$ by
\begin{equation}\label{eq:theta-ij-im}
\theta_{l,i_j}^{(n_j)}
\;:=\;
I_{E_{1,\,j-1}} \,\otimes\,
t_{li_j}^{(n_j)} \,\otimes\,
I_{E_{\,j+1,\,p}}.
\end{equation}
We set \(W_{l,i_j} := \theta_{l,i_j}\otimes U_{li_j}\) and for a given $n$, we define the unitary operator $D_j[W_{l,i_j}]
:\;
E_{1,\,j-1}\otimes E_{l}\otimes E_{i_j}^{\,n_j}\otimes E_{\,j+1,\,p} \otimes \clh
\;\longrightarrow\;
E_{1,\,j}\otimes E_{l}\otimes E_{\,j+1,\,p} \otimes \clh$
by
\begin{equation}\label{eq:Dm-theta-U-im}
D_j[W_{l,i_j}]
\;:=\;
\theta_{l,i_j}^{(n_j)}\otimes U_{li_j}^{\,n_j}.
\end{equation}

We now present our first example, which serves as a model for the class of
doubly twisted representations introduced above.
The construction produces a doubly twisted covariant representation
associated with a fixed subset $A \subseteq I_k$, where the coordinates in $A$
give rise to an induced part, while those in $A^{\mathrm c}$ contribute a fully
coisometric part.
Using this model, one can generate a wide class of doubly twisted
representations.
Moreover, taking direct sums of such models over different choices of $A$
again yields a doubly twisted representation.

\begin{example}\label{ex:fock-model}
Let $\mathbb{E}=\{E(\mathbf n)\}_{\mathbf n\in\mathbb Z_+^k}$ be a product system of
$C^*$-correspondences over $\mathcal A$ with the left action given by  $\varphi_i:\mathcal A\to\mathcal L(E_i)$. For $i\neq j$, let
\(
t_{ij}:
E_i\otimes_{\mathcal A}E_j\longrightarrow E_j\otimes_{\mathcal A}E_i
\)
be the canonical unitary flip.
Let $\cld$ be a Hilbert space and
 $\sigma:\mathcal A\to\mathcal B(\mathcal D)$ be a nondegenerate
$*$-representation. Also, let $\{U_{ij}\}_{i\neq j}\subset\mathcal B(\mathcal D)$ be a family of
pairwise commuting unitaries satisfying
\[
U_{ji}=U_{ij}^*,\qquad
U_{ij}\sigma(a)=\sigma(a)U_{ij}\quad (a\in\mathcal A).
\]
Fix $m\in I_k$. Let $A=I_m=\{1,\ldots,m\},$ therefore $A^{\mathrm c}
= I_k\setminus I_m=\{m+1,\ldots,k\}.$
Let 
\(
(\sigma,W_{m+1},\ldots,W_k)
\) be a \emph{fully coisometric twisted covariant representation}
of the subsystem $\mathbb E_{A^{\mathrm c}}$ on $\mathcal D$ with corresponding twists $\{t_{ij} \otimes U_{ij}: i, j \in A^{\mathrm{c}},\, i\neq j\}$.
Consider the Hilbert space $\mathcal H:=\mathcal F(E_A)\otimes_\sigma\mathcal D,
$ and define a representation $\tau:\mathcal A\to\mathcal B(\mathcal H)$ by
\[
\tau(a):=\varphi_{(\infty)}(a)\otimes I_{\mathcal D},
\text{ where }
\varphi_{(\infty)}(a)
=
\bigoplus_{\mathbf n\in\mathbb Z_+^m}\varphi^{(\mathbf n)}(a).
\]
Since $\varphi_{(\infty)}$ and $\sigma$ are nondegenerate, so is $\tau$.
For $x_i\in E_i$, $i\in I_k$, define completely contractive operators $M_i:E_i\to\mathcal B(\mathcal H)$ by
\begin{equation}\label{eq:model-operators}
M_i(x_i)=
\begin{cases}
T_{x_i}\otimes I_{\mathcal D},
& i=1,\\[6pt]
D_{i-1}[W_{i,i-1}]
\cdots
D_1[W_{i1}]
\,(T_{x_i}\otimes I_{\mathcal D}),
& 2\le i\le m,\\[6pt]
(I_{\mathcal F(E_{A})}\otimes\widetilde W_i)\,
D_m[W_{im}]
\cdots
D_1[W_{i1}]
\,(T_{x_i}\otimes I_{\mathcal D}),
& m+1\le i\le k.
\end{cases}
\end{equation}
For $i\neq j$, define $U_{A,ij}:=I_{\mathcal F(E_{A})}\otimes U_{ij}.$
Then, for all $\ell \in I_k$ and $a \in \mathcal A,$
\[
U_{A,ji}=U_{A,ij}^*, \qquad
U_{A,ij}\,\widetilde M_{\ell}
=
\widetilde M_{\ell}(I_{E_{\ell}}\otimes U_{A,ij}),
\qquad
U_{A,ij}\tau(a)=\tau(a)U_{A,ij},
\]
and
$(\tau,M_1,\ldots,M_k)$ is a doubly twisted covariant isometric representation
of the product system $\mathbb E$ on $\mathcal H$ with twists
$\{t_{ij}\otimes U_{A,ij}\}_{i \neq j}$. Moreover, $(\tau,M_i)$ is induced for $i \in A,$ whereas $(\tau,M_i)$ is fully coisometric for each $i \in A^{\mathrm c}$.

In Theorem \ref{thm:verification of ex 1} below we verify the details of the example. 
\end{example}

\begin{remark}
The above construction is presented for the simpler case
$A = I_m = \{1,\ldots,m\}$ only for notational convenience.
For an arbitrary subset $A \subseteq I_k$, the same model is obtained
by reindexing the product system and the twisting unitaries accordingly.
All statements and proofs carry over verbatim.
\end{remark}

\begin{lemma}\label{lem:commutation-with-phi-infty}
Let $a\in\mathcal A$. Fix $\ell\in A$ and $i\in I_k$ with $i\neq \ell$.
Let $n_\ell\ge1$. Then for
\(
D_\ell[W_{i,\ell}]
=
\theta_{i,\ell}^{(n_\ell)}\otimes U_{i\ell}^{\,n_\ell},
\)
we have
\[
(\varphi_{(\infty)}(a)\otimes I_{\mathcal D})\,D_\ell[W_{i,\ell}]
=
D_\ell[W_{i,\ell}]\,(\varphi_{(\infty)}(a)\otimes I_{\mathcal D}).
\]
\end{lemma}

\begin{proof}
Since $U_{i\ell}^{\,n_\ell}$ acts on $\mathcal D$ only, it suffices to show that
$\theta_{i,\ell}^{(n_\ell)}$ commutes with $\varphi_{(\infty)}(a)$ on
$\mathcal F(E_A)$.
By $\mathcal A$--bilinearity of $t_{i\ell}:E_i\otimes_{\mathcal A}E_\ell\to
E_\ell\otimes_{\mathcal A}E_i$, we have, for all $a\in\mathcal A$,
\[
(\varphi_\ell(a)\otimes I_{E_i})\,t_{i\ell}
=
t_{i\ell}\,(\varphi_i(a)\otimes I_{E_\ell}).
\]
Iterating this identity (equivalently, using Lemma~\ref{lemma:recursive-def-tij-n})
yields, for every $n_\ell\ge1$,
\[
(\varphi_\ell^{(n_\ell)}(a)\otimes I_{E_i})\,t_{i\ell}^{(n_\ell)}
=
t_{i\ell}^{(n_\ell)}\,(\varphi_i(a)\otimes I_{E_\ell^{\,n_\ell}}).
\]
Tensoring with identities on the remaining factors gives
\[
\varphi^{(\mathbf n)}(a)\,\theta_{i,\ell}^{(n_\ell)}
=
\theta_{i,\ell}^{(n_\ell)}\,\varphi^{(\mathbf n)}(a)
\quad\text{on each summand }E(\mathbf n)\subseteq \mathcal F(E_A),
\]
and hence, since $\varphi_{(\infty)}(a)=\bigoplus_{\mathbf n}\varphi^{(\mathbf n)}(a)$,
\[
\varphi_{(\infty)}(a)\,\theta_{i,\ell}^{(n_\ell)}
=
\theta_{i,\ell}^{(n_\ell)}\,\varphi_{(\infty)}(a).
\]
Combining this with the commutation of $U_{i\ell}^{\,n_\ell}$ with
$\varphi_{(\infty)}(a)\otimes I_{\mathcal D}$ proves the claim.
\end{proof}

\begin{thm}\label{thm:verification of ex 1}
The representation $(\tau,M_1,\ldots,M_k)$ 
of the product system $\mathbb E$ on $\mathcal H$, with twists
$\{\,t_{ij}\otimes U_{A,ij}\,\}_{i\neq j}$, satisfies the following property:
\begin{enumerate}
    \item For all $i\in I_k$, $M_i$ satisfies the covariance relation.
    \item For $i,j\in I_k$ with $i\neq j$, the maps $\widetilde M_i,\widetilde M_j$ satisfy the
 twisted relation:
\begin{equation}\label{eq:twisted-ij-Im}
\widetilde M_i\bigl(I_{E_i}\otimes \widetilde M_j\bigr)
=
\widetilde M_j\bigl(I_{E_j}\otimes \widetilde M_i\bigr)\,
\bigl(t_{ij}\otimes U_{A,ij}\bigr)
\qquad\text{on }E_i\otimes E_j\otimes_\tau \mathcal H .
\end{equation}
\item For $i,j\in I_k$ with $i\neq j$, the maps $\widetilde M_i,\widetilde M_j$ satisfy the
doubly twisted relation:
\begin{equation}\label{eq:doublytwisted-ij-Im}
\widetilde M^*_j \widetilde M_i 
= \bigl(I_{E_j}\otimes \widetilde M_i\bigr)
\bigl(t_{ij}\otimes U_{A,ij}\bigr)
\bigl(I_{E_i}\otimes \widetilde M^*_j\bigr)
\qquad\text{on }E_i\otimes_\tau \mathcal H .
\end{equation}
\end{enumerate}
\end{thm}

\begin{proof}
 \begin{enumerate}
     \item To verify the covariance property, fix $a,b\in\mathcal A$ and $x_i \in E_i$ for some
$i\in I_k$. We first observe that, for any creation operator,
\[
T_{a\cdot x\cdot b}
=
\varphi_{(\infty)}(a)\,T_x\,\varphi_{(\infty)}(b).
\]  
Consequently, $M_1(a\cdot x_1\cdot b)
=
\tau(a)\,M_1(x_1)\,\tau(b).$
Now, for $1<i\le m$, and  $x_i\in E_i$, let
\(
M_i(x_i)=Q_i\,(T_{x_i}\otimes I_{\mathcal D}),
\) where
\(
Q_i=D_{i-1}[W_{i,i-1}]\cdots D_1[W_{i1}].
\) Lemma \ref{lem:commutation-with-phi-infty} implies that
$Q_i$ commutes with $\varphi_{(\infty)}(a)\otimes I_{\mathcal D}$.
Therefore, for $a,b\in\mathcal A$, using covariance of creation operators,
\begin{align*}
M_i(a\cdot x_i\cdot b)
&=Q_i\bigl((\varphi_{(\infty)}(a)\,T_{x_i}\,\varphi_{(\infty)}(b))\otimes I_{\mathcal D}\bigr)\\
&=(\varphi_{(\infty)}(a)\otimes I_{\mathcal D})\,
Q_i\,(T_{x_i}\otimes I_{\mathcal D})\,
(\varphi_{(\infty)}(b)\otimes I_{\mathcal D})\\
&=\tau(a)\,M_i(x_i)\,\tau(b).
\end{align*}
Now we verify covariance for $m+1\le i\le k$.
Set
\(
R_i:=D_m[W_{im}]\cdots D_1[W_{i1}],
\) 
so that $M_i(x_i)=(I\otimes \widetilde W_i)\,R_i\,(T_{x_i}\otimes I_{\mathcal D})$.
$R_i$ commutes with $\varphi_{(\infty)}(a)\otimes I_{\mathcal D}$ by Lemma \ref{lem:commutation-with-phi-infty}. Moreover, $\varphi_{(\infty)}(a)$ acts on the
$\mathcal F(E_A)$ tensor factor, while $\widetilde W_i$ acts on the
$\mathcal D$--factor, and hence they commute. Therefore, for $a,b\in\mathcal A$, using covariance of creation operators, we obtain
\begin{align*}
M_i(a\cdot x_i\cdot b)
&=(I\otimes \widetilde W_i)\,R_i\bigl((\varphi_{(\infty)}(a)\,T_{x_i}\,\varphi_{(\infty)}(b))\otimes I_{\mathcal D}\bigr)\\
&=(\varphi_{(\infty)}(a)\otimes I_{\mathcal D})\,(I\otimes \widetilde W_i)\,R_i\,
(T_{x_i}\otimes I_{\mathcal D})\,(\varphi_{(\infty)}(b)\otimes I_{\mathcal D})\\
&=\tau(a)\,M_i(x_i)\,\tau(b).
\end{align*}
\item 
Fix $i,j\in I_m$ with $i\le j$ (the case $i\ge j$ is symmetric).
First observe that, by Definition~\ref{eq:model-operators} and the fact that
the unitaries $\{U_{pq}\}_{p\neq q}$ commute, the $\mathcal D$--component
of both
\(
\widetilde M_i\bigl(I_{E_i}\otimes \widetilde M_j\bigr) \) and \(
\widetilde M_j\bigl(I_{E_j}\otimes \widetilde M_i\bigr)
\bigl(t_{ij}\otimes U_{A,ij}\bigr)
\)
coincide.
Indeed, $U_{A,ij}=I_{\mathcal F(E_A)}\otimes U_{ij}$ acts only on the
$\mathcal D$--tensor factor, and the resulting products of the unitaries
$\{U_{pq}\}_{p\neq q}$ coincide by commutativity.
Therefore, to verify the twisted relation it suffices to compare the
actions of the two sides on the Fock space component
$\mathcal F(E_{I_m})$.
Let $\mathcal L$ and $\mathcal R$ denote the corresponding Fock-space
components of the left-hand side and the right-hand side, respectively. Explicitly
\[
\mathcal{L}= \theta^{(n_{i-1})}_{i,i-1} \cdots \theta^{(n_1)}_{i,1}
\left( I_{E_i} \otimes
      \theta^{(n_{j-1})}_{j,j-1} \cdots \theta^{(n_1)}_{j,1}\right),
\]
and
\begin{align*}
\mathcal{R}
&= \theta^{(n_{j-1})}_{j,j-1} \cdots \theta^{(n_i+1)}_{j,i}
\cdots \theta^{(n_1)}_{j,1}
\left( I_{E_j} \otimes
      \theta^{(n_{i-1})}_{i,i-1} \cdots \theta^{(n_1)}_{i,1}\right)
( t_{ij}\otimes I_{\mathcal F(E_{A})}).
\end{align*}
The equality $\mathcal L=\mathcal R$ is encoded by the commutative hexagon, which commutes by the block hexagon (braid) relation for the flips $t_{ij}$.
\begin{equation}\label{eq:hexagon-diagram}
\begin{tikzcd}[column sep=large, row sep=large]
E_i \otimes E_j \otimes E_{I_m}^{\boldsymbol n}
\arrow[r]
\arrow[d]
&
E_i \otimes E_{I_m}^{\boldsymbol n+\mathbf e_j}
\arrow[r]
&
E_{I_m}^{\boldsymbol n+\mathbf e_i+\mathbf e_j}
\\
E_j \otimes E_i \otimes E_{I_m}^{\boldsymbol n}
\arrow[r]
&
E_j \otimes E_{I_m}^{\boldsymbol n+\mathbf e_i}
\arrow[r]
&
E_{I_m}^{\boldsymbol n+\mathbf e_i+\mathbf e_j}.
\end{tikzcd}
\end{equation}

For a more explicit argument, note that $\mathcal L$ acts by first inserting
the $E_j$--block into $E_A^{\mathbf n}$ and then inserting the $E_i$--block.
Since for $p\neq q$ the operators $\theta_{i,p}$ and $\theta_{j,q}$ act on
distinct tensor components, this action may be realized by alternately applying
$\theta_{j,\ell}^{(n_\ell)}$ and $\theta_{i,\ell}^{(n_\ell)}$ for
$\ell=1,\dots,i-1$ as follows:
\begin{equation*}\label{eq:hexagon-diagram}
\begin{tikzcd}[column sep=small, row sep=large]
E_i \otimes E_j \otimes E_{A}^{n}
\arrow[r]
&
E_i \otimes E_{i_1}^{n_1} \otimes E_j \otimes E_{2,m} 
\arrow[r]
&
 E_{i_1}^{n_1} \otimes E_i \otimes E_j \otimes E_{2,m} 
\arrow[r]
&
\cdots
\arrow[r]
&
 E_{A}^{n+e_i+e_j}.
\end{tikzcd}
\end{equation*}

Formally, and with a slight abuse of notation (suppressing identity operators
on tensor components on which the maps do not act), we may write
\begin{align*}
\mathcal L
&=
\theta^{(n_{j-1})}_{j,j-1}\cdots \theta^{(n_i)}_{j,i}
\Bigl[
  \theta^{(n_{i-1})}_{i,i-1}\,
  \theta^{(n_{i-1})}_{j,i-1}
\Bigr]
\cdots
\Bigl[
  \theta^{(n_1)}_{i,1}\,
  \theta^{(n_1)}_{j,1}
\Bigr]\\
&=
\theta^{(n_{j-1})}_{j,j-1}\cdots \theta^{(n_i)}_{j,i}
\Bigl[
 t_{ji}\,
  \theta^{(n_{i-1})}_{j,i-1}\,
  \theta^{(n_{i-1})}_{i,i-1}\,
 t_{ij}
\Bigr]
\cdots
\Bigl[
  t_{ji}\,
  \theta^{(n_1)}_{j,1}\,
  \theta^{(n_1)}_{i,1}\,
  t_{ij}
\Bigr] \quad (\text{Proposition} \ref{prop: braid at nth level})\\
&=
\theta^{(n_{j-1})}_{j,j-1}\cdots \theta^{(n_i+1)}_{j,i}
\Bigl[
  \theta^{(n_{i-1})}_{j,i-1}\,
  \theta^{(n_{i-1})}_{i,i-1}\,
\Bigr]
\cdots
\Bigl[
  \theta^{(n_1)}_{j,1}\,
  \theta^{(n_1)}_{i,1}\,
\Bigr] t_{ij}\\
&=
\theta^{(n_{j-1})}_{j,j-1}\cdots \theta^{(n_i+1)}_{j,i}
\cdots \theta^{(n_1)}_{j,1}
\left(
I_{E_j}\otimes
\theta^{(n_{i-1})}_{i,i-1}\cdots \theta^{(n_1)}_{i,1}
\right)
t_{ij}\\
&=\mathcal{R},
\end{align*}
where the second last equality follows because $\theta^{(n_l)}_{i,l}\theta^{(n_k)}_{j,k}=\theta^{(n_k)}_{j,k}\theta^{(n_l)}_{i,l}, \, \forall k < l,$ as they act on different components. 

Now consider the case $i\in A$ and $j\in A^{\mathrm c}$.
Since $\widetilde W_j$ acts trivially on the
Fock space $\mathcal F(E_A)$, it commutes with the
operators $D_{i-r}[W_{ir}]$ acting on the Fock-space component. Therefore,
\begin{align*}
\widetilde M_i\bigl(I_{E_i}\otimes \widetilde M_j\bigr)
(x_i \otimes x_j \otimes \xi \otimes h)
&=
\prod_{r=1}^{i-1} D_{i-r}[W_{ir}]\,
(I_{\mathcal F(E_A)}\otimes \widetilde W_j)
\prod_{r=1}^{m} D_{m+1-r}[W_{jr}]
(x_i \otimes x_j \otimes \xi \otimes h) \\
&=
(I_{\mathcal F(E_A)}\otimes \widetilde W_j)
\prod_{r=1}^{i-1} D_{i-r}[W_{ir}]
\prod_{r=1}^{m} D_{m+1-r}[W_{jr}]
(x_i \otimes x_j \otimes \xi \otimes h).
\end{align*}

As in the previous case, by Definition~\ref{eq:model-operators} and the
commutativity of the unitaries $\{U_{pq}\}_{p\neq q}$, the $\mathcal D$--components
of
\(
\widetilde M_i\bigl(I_{E_i}\otimes \widetilde M_j\bigr)
\)
and
\(
\widetilde M_j\bigl(I_{E_j}\otimes \widetilde M_i\bigr)
(t_{ij}\otimes U_{A,ij})
\)
coincide.
Therefore, it suffices to compare the induced maps on the Fock-space component
$\mathcal F(E_A)$.
Let $\mathcal L$ and $\mathcal R$ denote the corresponding Fock-space maps.
Then
\[
\mathcal{L}
=
\theta^{(n_{i-1})}_{i,i-1} \cdots \theta^{(n_1)}_{i,1}
\Bigl(
I_{E_i}\otimes
\theta^{(n_m)}_{j,m}\cdots \theta^{(n_1)}_{j,1}
\Bigr),
\]
while
\[
\mathcal{R}
=
\theta^{(n_m)}_{j,m}\cdots \theta^{(n_i+1)}_{j,i}
\cdots \theta^{(n_1)}_{j,1}
\Bigl(
I_{E_j}\otimes
\theta^{(n_{i-1})}_{i,i-1}\cdots \theta^{(n_1)}_{i,1}
\Bigr)
(t_{ij}\otimes I_{\mathcal F(E_A)}).
\]
The equality $\mathcal L=\mathcal R$ follows exactly as in the case
$i,j\in A$, using the braid (hexagon) relation for the flips $t_{ij}$.
This completes the proof for $i\in A$ and $j\in A^{\mathrm c}$.

Finally, let $i,j\in A^{\mathrm c}$.
Since $\widetilde W_j$ acts trivially on the Fock space
$\mathcal F(E_A)$, it commutes with the operators
$D_{i-r}[W_{ir}]$ acting on the Fock-space component.
Consequently, in the expression
$\widetilde M_i\bigl(I_{E_i}\otimes \widetilde M_j\bigr)$,
the factor $(I_{\mathcal F(E_A)}\otimes \widetilde W_j)$ may be moved past
$\displaystyle\prod_{r=1}^{i-1} D_{i-r}[W_{ir}]$ without affecting the Fock-space part.

It follows that both
\(
\widetilde M_i\bigl(I_{E_i}\otimes \widetilde M_j\bigr)
\)
and
\(
\widetilde M_j\bigl(I_{E_j}\otimes \widetilde M_i\bigr)
(t_{ij}\otimes U_{A,ij})
\)
act identically on the Fock-space component $\mathcal F(E_A)$, and hence the
verification of the twisted relation reduces entirely to the corresponding
twisted relation for $(\widetilde W_i,\widetilde W_j)$ on $\mathcal D$ along with hexagon relations.
This completes the proof in the case $i,j\in A^{\mathrm c}$.

\item We verify the doubly twisted relations for
$(\sigma,\widetilde M_1,\ldots,\widetilde M_m)$.
If at least one of the indices $i$ or $j$ belongs to $A^{\mathrm c}$,
the desired relations follow directly from the corresponding twisted relations.
In the remaining case $i,j\in A$, note that
$U_{A,ij}=I_{\mathcal F(E_A)}\otimes U_{ij}$ and that the family
$\{U_{pq}\}_{p\neq q}$ consists of commuting unitaries.
Consequently, the $\mathcal D$--components of the two sides coincide, and it
suffices to compare the induced maps on the Fock space $\mathcal F(E_A)$. Notice that,
\begin{align*}
\langle \widetilde M^*_j(\xi \otimes h), x_j \otimes \eta \otimes h_1 \rangle
&=\langle \xi \otimes h,\widetilde M_j( x_j \otimes \eta \otimes h_1) \rangle\\
&=\langle \xi \otimes h,  D_{j-1}[W_{j j-1}] \cdots D_{1}[W_{j 1}]( x_j \otimes \eta \otimes h_1) \rangle\\
&=\langle  D_{1}[W_{1j}]  \cdots D_{j-1}[W_{j-1j}](\xi \otimes h),   x_j \otimes \eta \otimes h_1 \rangle\\
\end{align*}
Fix $i\neq j$. Both sides define bounded maps $E_i\otimes H\to E_j\otimes H$,
so it suffices to compare their matrix coefficients against elementary tensors.
Let $x_i\in E_i$, $x_j\in E_j$, $\xi,\eta\in\mathcal F(E_A)$, and $h,h_1\in\mathcal D$.
By the defining property of adjoints,
\[
\bigl\langle \widetilde M_j^{*}\widetilde M_i(x_i\otimes \xi\otimes h),\,
x_j\otimes \eta\otimes h_1\bigr\rangle
=
\bigl\langle \widetilde M_i(x_i\otimes \xi\otimes h),\,
\widetilde M_j(x_j\otimes \eta\otimes h_1)\bigr\rangle.
\]
The resulting Fock-space identities are exactly the two
routes in the block-hexagon diagram, and therefore coincide by the braid
relation for the flips $t_{ij}$ (Proposition~\ref{prop: braid at nth level}).
Thus the two operators agree on elementary tensors, and the desired identity
follows.
\end{enumerate}
\end{proof}

We now address the problem of constructing concrete examples of doubly twisted
representations from prescribed families of operators.
To this end, we develop two frameworks: the automorphic case
$E_i=\alpha_i\mathcal A$ and the scalar case $\mathcal A=\mathbb C$.
In each setting, we provide explicit characterizations of (doubly) twisted
representations in terms of concrete operator relations, thereby extending
Solel’s theory of doubly commuting representations~\cite{Solel} to the
doubly twisted setting.

\begin{example}\label{ex2:A=C}
Let $\mathcal A=\mathbb C$, and let
$\mathbb E=\{E_i\}_{i=1}^k$ be a product system of $C^*$-correspondences over
$\mathbb Z_+^k$.  We identify each fiber $E_i$ with a Hilbert space
$\mathcal H_i$ equipped with a fixed orthonormal basis
$\{e^i_\alpha : \alpha\in J_i\}$.
Let $t_{ij}:\mathcal H_i\otimes\mathcal H_j \longrightarrow
\mathcal H_j\otimes\mathcal H_i$
denote the unitary flip maps determining the product system structure.
Let $\mathcal H$ be a Hilbert space, and let $\{U_{ij}\}_{i\neq j}\subseteq\mathcal B(\mathcal H)$
be a family of unitaries.
Suppose that for each $i\in I_k$ we are given a family of operators
$\{S^i_\alpha : \alpha\in J_i\}\subseteq\mathcal B(\mathcal H)$ such that $\displaystyle \sum_{\alpha\in J_i} S^i_\alpha (S^i_\alpha)^* \le I_{\mathcal H}$, and for all $i,j\in I_k$ and $\alpha\in J_i$, $\beta\in J_j$,
\begin{equation}\label{eq:S_twisted_assumed}
S^i_\alpha S^j_\beta
=
\sum_{\gamma\in J_j,\ \delta\in J_i}
\big\langle e^j_\gamma\otimes e^i_\delta,\,
t_{ij}(e^i_\alpha\otimes e^j_\beta)\big\rangle\,
S^j_\gamma S^i_\delta\, U_{ij};
\end{equation}
\begin{equation}\label{eq:S_doubly_twisted_assumed}
(S^j_\beta)^* S^i_\alpha
=
\sum_{\delta\in J_i,\ \gamma\in J_j}
\big\langle e^j_\beta\otimes e^i_\delta,\,
t_{ij}(e^i_\alpha\otimes e^j_\gamma)\big\rangle\,
S^i_\delta\, U_{ij}\,(S^j_\gamma)^* .
\end{equation}
Then there exists a doubly twisted representation
$(\sigma,T_1,\ldots,T_k)$ of the product system $\mathbb E$ on $\mathcal H$
such that
\[
S^i_\alpha h
=
\widetilde T_i(e^i_\alpha\otimes h),
\qquad h\in\mathcal H.
\]
Here $\sigma:\mathbb C\to B(\mathcal H)$ is given by $\sigma(\lambda)=\lambda I_{\mathcal H}$.
We verify the details of the above example in the Proposition \ref{prop:construct_doubly_twisted}.
\end{example}

\begin{remark}
In applications, we will assume that the family $\{U_{ij}\}_{i\neq j}$
consists of pairwise commuting unitaries satisfying $U_{ji}=U_{ij}^*$.
These additional assumptions are not required for the construction below,
but are part of the standard notion of a (doubly) twisted representation.
\end{remark}

\begin{prop}\label{prop:construct_doubly_twisted}
In the setting of Example~\ref{ex2:A=C}, suppose that the families
$\{S^i_\alpha : \alpha\in J_i\}\subseteq\mathcal B(\mathcal H)$
satisfy \eqref{eq:S_twisted_assumed} and \eqref{eq:S_doubly_twisted_assumed}.
Then there exists a doubly twisted representation
$(\sigma,T_1,\ldots,T_k)$ of the product system $\mathbb E$ on $\mathcal H$
such that
\[
\widetilde T_i(e^i_\alpha\otimes h)=S^i_\alpha h,
\qquad h\in\mathcal H.
\]
\end{prop}

\begin{proof}
For each $i$, define a linear operator
$\widetilde T_i:\mathcal H_i\otimes\mathcal H\to\mathcal H$ on elementary tensors by
\begin{equation} \label{eq:def_Ttilde_converse}
  \widetilde T_i(\xi\otimes h)
:=\sum_{\alpha\in J_i}\langle e^i_\alpha,\xi\rangle\, S^i_\alpha h,
\qquad \xi\in\mathcal H_i,\ h\in\mathcal H.  
\end{equation}
Let $R_i:\displaystyle\bigoplus_{\alpha\in J_i}\mathcal H\to\mathcal H$ be the row operator
$R_i((h_\alpha)_\alpha)=\sum_\alpha S^i_\alpha h_\alpha$.
Since $\{S^i_\alpha\}_{\alpha\in J_i}$ is a row contraction, we have
$R_iR_i^*\le I_{\mathcal H}$, and hence $\|R_i\|\le 1$.
Let
$U_i:\mathcal H_i\otimes\mathcal H\to\displaystyle\bigoplus_{\alpha\in J_i}\mathcal H$
be the canonical unitary
$U_i(\xi\otimes h)=(\langle e^i_\alpha,\xi\rangle h)_\alpha$.
Then $\widetilde T_i=R_iU_i$, so $\widetilde T_i$ is a contraction.

Fix $i,j\in I_k$. We prove the twisted relation for $\widetilde T_i, \, \widetilde T_j$
on elementary tensors, hence on all of $\mathcal H_i\otimes\mathcal H_j\otimes\mathcal H$ by linearity
and continuity.
Let $\xi\in\mathcal H_i$, $\eta\in\mathcal H_j$, and $h\in\mathcal H$. Using \eqref{eq:def_Ttilde_converse},
\begin{align*}
\widetilde T_i(I_{\mathcal H_i}\otimes \widetilde T_j)(\xi\otimes\eta\otimes h)
&=\widetilde T_i\!\left(\xi\otimes \sum_{\beta\in J_j}\langle e^j_\beta,\eta\rangle\, S^j_\beta h\right) \\
&=\sum_{\alpha\in J_i}\sum_{\beta\in J_j}
   \langle e^i_\alpha,\xi\rangle\,\langle e^j_\beta,\eta\rangle\, S^i_\alpha S^j_\beta h\\
&=\sum_{\alpha,\beta}\langle e^i_\alpha,\xi\rangle\,\langle e^j_\beta,\eta\rangle
   \sum_{\gamma,\delta}
   \big\langle e^j_\gamma\otimes e^i_\delta,\,
      t_{ij}(e^i_\alpha\otimes e^j_\beta)\big\rangle\,
   S^j_\gamma S^i_\delta\,U_{ij}h \quad[\text{by }\eqref{eq:S_twisted_assumed}] \\
   &=\sum_{\gamma,\delta}
   \Big\langle e^j_\gamma\otimes e^i_\delta,\,
      t_{ij}(\xi\otimes\eta)\Big\rangle\,
   S^j_\gamma S^i_\delta\,U_{ij}h.
\end{align*}
Writing $t_{ij}(\xi\otimes\eta)=\sum_{\gamma,\delta} c_{\gamma\delta}\, e^j_\gamma\otimes e^i_\delta$
with $c_{\gamma\delta}=\langle e^j_\gamma\otimes e^i_\delta,\,t_{ij}(\xi\otimes\eta)\rangle$,
the last expression becomes
\begin{align*}
\widetilde T_i(I_{\mathcal H_i}\otimes \widetilde T_j)(\xi\otimes\eta\otimes h)
&=\sum_{\gamma,\delta} c_{\gamma\delta}\, S^j_\gamma S^i_\delta\,U_{ij}h \\
&=\widetilde T_j\!\left(\sum_{\gamma,\delta} c_{\gamma\delta}\,
      e^j_\gamma\otimes \widetilde T_i(e^i_\delta\otimes U_{ij}h)\right) \\
&=\widetilde T_j(I_{\mathcal H_j}\otimes \widetilde T_i)\,(t_{ij}\otimes U_{ij})
   (\xi\otimes\eta\otimes h).
\end{align*}
Hence $(\sigma,T_1,\ldots,T_k)$ is a twisted representation of $\mathbb E$.

Assume now that \eqref{eq:S_doubly_twisted_assumed} holds.
We show that this condition is equivalent to the doubly twisted covariance relation.
It suffices to check the equality on elementary tensors.
For this, first notice that, for each $j\in I_k$ and $h\in\mathcal H$, 
\begin{align*}
\langle \xi\otimes h',\, \widetilde T_j^{*}h\rangle
&=\left\langle \sum_{\gamma\in J_j}\langle e^j_\gamma,\xi\rangle\, S^j_\gamma h',\, h\right\rangle 
=\sum_{\gamma\in J_j}\langle \xi, e^j_\gamma\rangle\, \langle h',\, (S^j_\gamma)^* h\rangle
=\left\langle \xi\otimes h',\,
\sum_{\gamma\in J_j} e^j_\gamma\otimes (S^j_\gamma)^*h \right\rangle,
\end{align*}
with convergence in the strong operator topology.
Since finite linear combinations of elementary tensors are dense in
$\mathcal H_j\otimes\mathcal H$, 
we obtain \(
\widetilde T_j^{*}h
=
\displaystyle\sum_{\gamma\in J_j} e^j_\gamma\otimes (S^j_\gamma)^* h.
\)
Now we compute,
\begin{align*}
(I_{\mathcal H_j}\otimes \widetilde T_i)\,(t_{ij}\otimes U_{ij})\,
(I_{\mathcal H_i}\otimes \widetilde T_j^{*})(e^i_\alpha\otimes h)
&=(I_{\mathcal H_j}\otimes \widetilde T_i)\,(t_{ij}\otimes U_{ij})
  \left(
    \sum_{\gamma\in J_j} e^i_\alpha\otimes e^j_\gamma\otimes (S^j_\gamma)^* h
  \right) \\
&=(I_{\mathcal H_j}\otimes \widetilde T_i)
  \left(
    \sum_{\gamma\in J_j} t_{ij}(e^i_\alpha\otimes e^j_\gamma)
      \otimes U_{ij}(S^j_\gamma)^* h
  \right) \\
&=\sum_{\delta\in J_i,\ \gamma\in J_j}
   \langle e^j_\beta\otimes e^i_\delta,\,
          t_{ij}(e^i_\alpha\otimes e^j_\gamma)\rangle\,
   e^j_\beta\otimes
   \widetilde T_i(e^i_\delta\otimes U_{ij}(S^j_\gamma)^* h).
\end{align*}
Applying $\widetilde T_i(e^i_\delta\otimes k)=S^i_\delta k$ and summing over $\beta$,
this expression becomes
\[
\sum_{\beta\in J_j} e^j_\beta\otimes
\left(
\sum_{\delta\in J_i,\ \gamma\in J_j}
\langle e^j_\beta\otimes e^i_\delta,\,
        t_{ij}(e^i_\alpha\otimes e^j_\gamma)\rangle\,
S^i_\delta U_{ij}(S^j_\gamma)^* h
\right).
\]
By \eqref{eq:S_doubly_twisted_assumed}, the inner sum equals $(S^j_\beta)^* S^i_\alpha h$,
and hence the above expression reduces to
\[
\sum_{\beta\in J_j} e^j_\beta\otimes (S^j_\beta)^* S^i_\alpha h
=\widetilde T_j^{*} S^i_\alpha h
=\widetilde T_j^{*}\widetilde T_i(e^i_\alpha\otimes h).
\]
Therefore, the representation $(\sigma,T_1,\ldots,T_k)$ is doubly twisted.
\end{proof}
The preceding proposition shows how a family of operators
$\{S^i_\alpha\}$ satisfying the coordinate relations
\eqref{eq:S_twisted_assumed}--\eqref{eq:S_doubly_twisted_assumed}
gives rise to a doubly twisted representation of the product system $\mathbb E$.
We now establish the converse construction.
We first compute the adjoint of the operator $\widetilde T_j$.

\begin{lemma}\label{lem:adjoint-Ttilde}
For each $j\in I_k$, the adjoint $\widetilde T_j^*:\mathcal H\to \mathcal H_j\otimes\mathcal H$
is given by
\[
\widetilde T_j^{*} h
=
\sum_{\gamma\in J_j} e^j_\gamma \otimes (S^j_\gamma)^* h,
\qquad h\in\mathcal H,
\]
where the series converges in the strong operator topology.
\end{lemma}

\begin{proof}
Let $x\in\mathcal H_j$ and $h,h'\in\mathcal H$.
Using the expansion of $x$ with respect to the fixed orthonormal basis
$\{e^j_\gamma\}_{\gamma\in J_j}$, we have
\(
\widetilde T_j(x\otimes h)
=
\displaystyle\sum_{\gamma\in J_j}\langle e^j_\gamma,x\rangle\, S^j_\gamma h.
\)
Consequently, for any $h' \in \clh,$
\begin{align*}
\langle x\otimes h,\, \widetilde T_j^* h' \rangle
=\left\langle \sum_{\gamma\in J_j}\langle e^j_\gamma,x\rangle\, S^j_\gamma h,\,
h'\right\rangle 
=\left\langle x\otimes h,\,
\sum_{\gamma\in J_j} e^j_\gamma \otimes (S^j_\gamma)^* h'
\right\rangle.
\end{align*}
Therefore, for all $h\in\mathcal H,\, \widetilde T_j^{*} h
=
\displaystyle\sum_{\gamma\in J_j} e^j_\gamma \otimes (S^j_\gamma)^* h.$
The series
on the right-hand side converges in the strong operator topology.
\end{proof}

\begin{theorem}\label{thm:matrix-characterization}
In the setting of Example~\ref{ex2:A=C},
$(\sigma,T_1,\ldots,T_k)$ is a twisted representation of the product system
$\mathbb E$ with respect to the twists $\{t_{ij}\otimes U_{ij}\}_{i\neq j}$
if and only if there exists a family of operators $\{S^{i}_\alpha : i\in I_k,\ \alpha\in J_i\}
\subseteq\mathcal B(\mathcal H)$
such that for all $i,j\in I_k$ and $\alpha\in J_i$, $\beta\in J_j$,
\[
\sum_{\alpha\in J_i} S^{i}_\alpha (S^{i}_\alpha)^* \le I_{\mathcal H},
\qquad
S^{i}_\alpha S^{j}_\beta
=
\sum_{\gamma\in J_j,\ \delta\in J_i}
\big\langle e^j_\gamma \otimes e^i_\delta,\,
t_{ij}(e^i_\alpha \otimes e^j_\beta)\big\rangle\,
S^{j}_\gamma S^{i}_\delta\, U_{ij}.
\]
Moreover, $(\sigma,T_1,\ldots,T_k)$ is a \emph{doubly twisted representation}
if and only if, in addition,
\[
(S^{j}_\beta)^* S^{i}_\alpha
=
\sum_{\delta\in J_i,\ \gamma\in J_j}
\big\langle e^j_\beta \otimes e^i_\delta,\,
t_{ij}(e^i_\alpha \otimes e^j_\gamma)\big\rangle\,
S^{i}_\delta\, U_{ij}\,(S^{j}_\gamma)^*.
\]
\end{theorem}

\begin{proof}
The ``if'' direction follows from Proposition~\ref{prop:construct_doubly_twisted}.
Conversely, suppose that $(\sigma,T_1,\ldots,T_k)$ is a twisted representation
of $\mathbb E$ on $\mathcal H$ with respect to the twists
$\{t_{ij}\otimes U_{ij}\}_{i\neq j}$.
For each $i\in I_k$ and $\alpha\in J_i$, define
\[
S^i_\alpha h := \widetilde T_i(e^i_\alpha\otimes h),
\qquad h\in\mathcal H.
\]
Since each $\widetilde T_i$ is a contraction, the family
$\{S^i_\alpha\}_{\alpha\in J_i}$ is a row contraction, that is, $\sum_{\alpha\in J_i} S^i_\alpha (S^i_\alpha)^* \le I_{\mathcal H}.$
Fix $i,j\in I_k$ and $\alpha\in J_i$, $\beta\in J_j$.
Using the twisted covariance relation for $(\sigma,T_1,\ldots,T_k)$ and 
expanding 
$t_{ij}(e^i_{\alpha} \otimes e^j_{\beta})= \sum_{\gamma \in J_j, \, \delta \in J_i} \langle e^j_\gamma \otimes e^i_{\delta}, \,
t_{ij}(e^i_\alpha \otimes e^j_\beta) \rangle e^j_\gamma \otimes e^i_{\delta},$ with respect to the fixed orthonormal bases, 
we obtain
\[
S^i_\alpha S^j_\beta
=
\sum_{\gamma\in J_j,\ \delta\in J_i}
\big\langle e^j_\gamma\otimes e^i_\delta,\,
t_{ij}(e^i_\alpha\otimes e^j_\beta)\big\rangle\,
S^j_\gamma S^i_\delta\, U_{ij},
\]
with convergence in the strong operator topology.
Thus, the relations \eqref{eq:S_twisted_assumed} hold.

Now assume that $(\sigma,T_1,\ldots,T_k)$ is a doubly twisted representation.
For $\beta\in J_j$, define an operator
$L^j_{\beta} : \mathcal H \to \mathcal H_j \otimes \mathcal H$ by $L^j_{\beta}(h) := e^j_{\beta} \otimes h,
\, h\in\mathcal H.$
Then $S^j_{\beta} = \widetilde T_j L^j_{\beta}$, and hence $(S^j_{\beta})^{*} = (L^j_{\beta})^{*}\,\widetilde T_j^{*}.$
For $x\in\mathcal H_j$ and $h\in\mathcal H$, we have $(L^j_{\beta})^{*}(x\otimes h)
=
\langle e^j_{\beta}, x\rangle\, h.$
Using the doubly twisted covariance relation, together with the above
expressions for the adjoints of $L^j_{\beta}$ and $\widetilde T_j^{*}$ (cf. Lemma \ref{lem:adjoint-Ttilde}) and the
expansion of $t_{ij}(e^i_\alpha\otimes e^j_\beta)$ with respect to the fixed
orthonormal bases, we obtain
\[
(S^j_\beta)^* S^i_\alpha
=
\sum_{\delta\in J_i,\ \gamma\in J_j}
\big\langle e^j_\beta\otimes e^i_\delta,\,
t_{ij}(e^i_\alpha\otimes e^j_\gamma)\big\rangle\,
S^i_\delta\, U_{ij}\,(S^j_\gamma)^*.
\]
Thus, the relation \eqref{eq:S_doubly_twisted_assumed} holds.
\end{proof}

Consequently, the above example shows that, after fixing orthonormal bases of the
fibers $\mathcal H_i$, there is a bijective correspondence between
the twisted (respectively, doubly twisted) representations of $\mathbb E$, and
the families of operators $\{S^i_\alpha\}$ satisfying
\eqref{eq:S_twisted_assumed} (respectively, together with
\eqref{eq:S_doubly_twisted_assumed}). Thus the abstract notion of a
(doubly) twisted representation becomes equivalent to a concrete operator model.

We now specialize to the automorphic case, where each fiber of the product
system is induced by a $*$-automorphism of the coefficient algebra.
In this setting, twisted and doubly twisted representations admit a concrete
description in terms of operator tuples satisfying natural covariance and
(commutation) relations. 

\begin{example}\label{ex2}
Let $\mathcal A$ be a $C^*$-algebra and let
$\alpha_1,\ldots,\alpha_k\in\text{Aut}(\mathcal A)$ be commuting $*$-automorphisms.
Let $\mathbb E$ denote the associated automorphic product system over
$\mathbb Z_+^k$, whose fibers are $E_i={}_{\alpha_i}\mathcal A$.
Thus each $E_i=\mathcal A$ as a vector space, with right action by multiplication,
left action $\varphi_i(a)c=\alpha_i(a)c$, and inner product
$\langle c_1,c_2\rangle=c_1^*c_2$.
Let $\sigma:\mathcal A\to\mathcal B(\mathcal H)$ be a nondegenerate
$*$-representation, and let $\{U_{ij}\}_{i\neq j}\subseteq\mathcal B(\mathcal H)$
be a family of pairwise commuting unitaries such that $U_{ji}=U_{ij}^*$ and
$U_{ij}\sigma(a)=\sigma(a)U_{ij}$ for all $a\in\mathcal A$.
Suppose $S_1,\ldots,S_k\in\mathcal B(\mathcal H)$ are contractions such that
for all $a\in\mathcal A$ and $i\neq j$,
\[
\sigma(a)S_i=S_i\sigma(\alpha_i(a)),
\qquad
S_iS_j=U_{ij}S_jS_i,
\qquad
S_i^*S_j=U_{ij}^*S_jS_i^*.
\]
Define operators $T_i:E_i\to\mathcal B(\mathcal H)$ by
\[
T_i(a):=S_i\sigma(a),\qquad a\in E_i={}_{\alpha_i}\mathcal A.
\]
Then $(\sigma,T_1,\ldots,T_k)$ is a doubly twisted representation of the product
system $\mathbb E$ on $\mathcal H$ with respect to the twists
$\{t_{ij}\otimes U_{ij}\}_{i\neq j}$, where (under the natural identifications
$E_i\otimes_{\mathcal A}E_j\simeq {}_{\alpha_j\circ\alpha_i}\mathcal A$ and using the fact that automorphisms commute) the
correspondence isomorphism $t_{ij}:E_i\otimes_{\mathcal A}E_j\longrightarrow E_j\otimes_{\mathcal A}E_i$
is given by
\[
t_{ij}(a\otimes b)=\alpha_i^{-1}\alpha_j(a)\otimes b,
\qquad a,b\in\mathcal A.
\]
\end{example}

We verify the details of this construction in Theorem~\ref{thm:char-for-Ei-alphaA},
where we establish a necessary and sufficient criterion characterizing
(doubly) twisted representations of automorphic product systems in terms of
concrete operator relations.

\begin{thm}\label{thm:char-for-Ei-alphaA}
In the setting of Example~\ref{ex2},
$(\sigma,T_1,\ldots,T_k)$ is  a \emph{doubly twisted representation} 
with respect to the twists $\{t_{ij}\otimes U_{ij}\}_{i\neq j}$
if and only if there exist contractions
$S_1,\ldots,S_k\in\mathcal B(\mathcal H)$ such that for all
$a\in\mathcal A$ and $i\neq j$,
\[
\sigma(a)S_i=S_i\sigma(\alpha_i(a)),
\quad 
S_iS_j=U_{ij}S_jS_i,
\quad
S_i^*S_j=U_{ij}^*S_jS_i^*,
\]
with $T_i(a)=S_i\sigma(a).$ 
\end{thm}

\begin{proof}
Let $(\sigma,T_1,\ldots,T_k)$ be a (respectively, doubly) twisted representation
of $\mathbb E$.
For $i\neq j$, the interior tensor product
$E_i\otimes_{\mathcal A}E_j$ is canonically identified with
${}_{\alpha_j\circ\alpha_i}\mathcal A$.
The correspondence isomorphisms
\(
t_{ij}:E_i\otimes_{\mathcal A}E_j \longrightarrow E_j\otimes_{\mathcal A}E_i
\)
satisfy the coherence relations required of a product system and therefore
determine the product system structure on $\mathbb E$
(see, Section~4 (4.3) of~\cite{Solel}).

Fix $i\in I_k$ and let $(u_\lambda)$ be a contractive approximate unit in
$\mathcal A$. Since $(\sigma,T_1,\ldots,T_k)$ is a completely contractive
representation of $\mathbb E$, each operator $T_i(u_\lambda)$ is a contraction.
Hence the net $\{T_i(u_\lambda)\}$ is uniformly bounded, and we define
\[
S_i := \text{\emph{SOT-}}\!\lim_{\lambda} T_i(u_\lambda)\ \in \mathcal B(\mathcal H).
\]
The limit exists and satisfies $\|S_i\|\le 1$.
Moreover, for every $b\in\mathcal A$ we have
\[
T_i(u_\lambda)\sigma(b)
= T_i(u_\lambda b)
\longrightarrow T_i(b)
\quad \text{in norm}.
\]
Passing to the limit yields $T_i(b)=S_i\sigma(b),\, b\in\mathcal A.$
Moreover, for all $a,b\in\mathcal A$, we compute
\[
S_i\sigma(\alpha_i(b))\sigma(a)
= T_i(\alpha_i(b)a)
= T_i(\varphi_i(b)a)
= \sigma(b)T_i(a)
= \sigma(b)S_i\sigma(a).
\]
Since $\sigma$ is nondegenerate, we obtain the covariance relation
\begin{equation}\label{eq:comm-sigma-Si}
\sigma(a)\,S_i = S_i\,\sigma(\alpha_i(a)),
\qquad a\in\mathcal A,\ i=1,\ldots,k.
\end{equation}
Let $a\otimes b\otimes h\in E_i\otimes E_j\otimes\mathcal H$.
Using the definition of $\widetilde T_i$ and $\widetilde T_j$, we compute
\begin{align*}
\widetilde T_i (I_{E_i}\otimes \widetilde T_j)
(a\otimes b\otimes h)
= T_i(a)T_j(b)h 
= S_i\sigma(a)\,S_j\sigma(b)h 
= S_iS_j\,\sigma(\alpha_i(a)b)h.
\end{align*}
On the other hand,
\begin{align*}
\widetilde T_j (I_{E_j}\otimes \widetilde T_i)
(t_{ij}\otimes U_{ij})(a\otimes b\otimes h)
&=
\widetilde T_j (I_{E_j}\otimes \widetilde T_i)
\bigl(\alpha_i^{-1}\alpha_j(a)\otimes b\otimes U_{ij}h\bigr) \\
&=
S_j\sigma(\alpha_i^{-1}\alpha_j(a))\,S_i\sigma(b)\,U_{ij}h \\
&=
S_jS_i\,\sigma(\alpha_j(a)b)\,U_{ij}h .
\end{align*}
The twisted relation implies
\begin{equation}\label{eq:Si-Sj-twisted}
S_i S_j\,\sigma(\alpha_i(a)b)h
=
S_j S_i\,\sigma(\alpha_j(a)b)\,U_{ij}h,
\qquad a,b\in\mathcal A,\ h\in\mathcal H.
\end{equation}
As $U_{ij}\,\widetilde T_\ell
=
\widetilde T_\ell\,(I_{E_\ell}\otimes U_{ij}),$ for $i,j,\ell \in I_k,\ i\neq j,$ therefore, for $a\in E_\ell$ and $h\in\mathcal H$,
\begin{align*}
U_{ij}T_\ell(a)h
= U_{ij}\,\widetilde T_\ell(a\otimes h) = \widetilde T_\ell(a\otimes U_{ij}h) = T_\ell(a)\,U_{ij}h .
\end{align*}
Consequently, for all $a\in\mathcal A$ and $h\in\mathcal H,\ U_{ij}S_\ell\sigma(a)h
= S_\ell\sigma(a)U_{ij}h.$
Since $\sigma$ is nondegenerate, this implies $U_{ij}S_\ell = S_\ell U_{ij},$ for $ 1\le i\neq j,\ \ell\le k.$
Combining this with \eqref{eq:Si-Sj-twisted} and using again the nondegeneracy
of $\sigma$, we conclude that
\[
S_i S_j = S_j S_i U_{ij},
\qquad 1\le i\neq j\le k.
\]
Since $\sigma$ is nondegenerate, for any contractive approximate unit
$(u_\lambda)$ in $\mathcal A$ we have $\sigma(u_\lambda)\to I_{\mathcal H}$
in the strong operator topology. Consequently,
$S_j\sigma(u_\lambda)\to S_j$ strongly.
Taking $\ell=j$ in the preceding relation yields $U_{ij}S_j\sigma(u_\lambda)=S_j\sigma(u_\lambda)U_{ij},$ for all $\lambda.$
Passing to strong limits on both sides, we obtain $U_{ij}S_j=S_jU_{ij}.$
Combining this with the relation \eqref{eq:Si-Sj-twisted}, it follows that
\begin{equation*}\label{eq:SiSj-twisted-rel}
S_iS_j = U_{ij}S_jS_i,
\qquad 1\le i\neq j\le k.
\end{equation*}
Recall that for $E_j = {}_{\alpha_j}\mathcal A$, one has the adjoint formula
\[
T_{e_j}^{*} h
= \mathrm{SOT}\text{-}\!\lim_{\lambda}\,
\mu_{\lambda} \otimes S_j^{*} h,
\qquad h \in \mathcal H,
\]
where $(\mu_{\lambda})$ is a contractive approximate unit of $\mathcal A$
(cf.~\cite{Solel}).
Let $a\in\mathcal A$ and $h\in\mathcal H$. Using the adjoint formula for
$\widetilde T_j^*$ and the fact that $T_i(a)=S_i\sigma(a)$, we compute
\begin{align*}
\widetilde T_j^{\,*}\widetilde T_i(a\otimes h)
&= \widetilde T_j^{\,*}\bigl(S_i\sigma(a)h\bigr) \\
&= \text{\emph{SOT-}}\lim_{\lambda}
\mu_\lambda \otimes S_j^{*}S_i\sigma(a)h \\
&= \text{\emph{SOT-}}\lim_{\lambda}
\mu_\lambda \otimes \sigma(\alpha_j\alpha_i^{-1}(a))S_j^{*}S_i h \\
&= \alpha_j\alpha_i^{-1}(a)\otimes S_j^{*}S_i h .
\end{align*}

Next, we compute
\begin{align*}
&(I_{E_j}\otimes \widetilde T_i)
(t_{ij}\otimes U_{ij})
(I_{E_i}\otimes \widetilde T_j^{\,*})(a\otimes h) \\
&\qquad=
(I_{E_j}\otimes \widetilde T_i)
(t_{ij}\otimes U_{ij})
\Bigl(a\otimes \text{\emph{SOT-}}\lim_\lambda \mu_\lambda\otimes S_j^{*}h\Bigr) \\
&\qquad=
(I_{E_j}\otimes \widetilde T_i)
\Bigl(\alpha_i^{-1}\alpha_j(a)\otimes
\text{\emph{SOT-}}\lim_\lambda \mu_\lambda\otimes U_{ij}S_j^{*}h\Bigr) \\
&\qquad=
\alpha_i^{-1}\alpha_j(a)\otimes
\text{\emph{SOT-}}\lim_\lambda \widetilde T_i(\mu_\lambda\otimes U_{ij}S_j^{*}h) \\
&\qquad=
\alpha_i^{-1}\alpha_j(a)\otimes
\text{\emph{SOT-}}\lim_\lambda T_i(\mu_\lambda)U_{ij}S_j^{*}h \\
&\qquad=
\alpha_i^{-1}\alpha_j(a)\otimes S_iU_{ij}S_j^{*}h .
\end{align*}
The doubly twisted relation implies that
\begin{equation*}\label{eq:SjSi-doubly}
S_j^{*}S_i
=
S_iU_{ij}S_j^{*}
=
U_{ij}^{*} S_j S_i^{*},
\qquad 1\le i\neq j\le k.
\end{equation*}

Conversely, suppose that $\sigma$ is a nondegenerate representation of
$\mathcal A$ on $\mathcal H$, and that $S_1,\ldots,S_k\in \mathcal{B}(\mathcal H)$ satisfy the given conditions.
For $i \in I_k$, define operators $T_i:E_i\to \mathcal{B}(\mathcal H)$ by
\[
T_i(a)=S_i\sigma(a),\qquad a\in E_i={}_{\alpha_i}\mathcal A.
\]
Then each $T_i$ is completely contractive, and the covariance relation
$\sigma(a)T_i(b)=T_i(\varphi_i(a)b)$ follows immediately from the defining
relation between $\sigma$ and $S_i$.
Moreover, using the relations among the $S_i$ and the definition of the maps
$t_{ij}$, a direct computation shows that the operators
$(\sigma,T_1,\ldots,T_k)$ satisfy the twisted (respectively, doubly twisted)
commutation relations.
Hence $(\sigma,T_1,\ldots,T_k)$ is a (doubly) twisted representation of the
product system $\mathbb{E}$.
\end{proof}

We now present an explicit example of operators $S_i$ satisfying the hypotheses
of Theorem~\ref{thm:char-for-Ei-alphaA}. By the theorem, this yields a
(doubly) twisted representation of the
associated product system.

\begin{example}
Let $\mathcal A=\mathbb C^3$ and $\mathcal H=\mathbb C^3$.
For $i=1,2$, define automorphisms $\alpha_i\in \text{Aut}(\mathcal A)$ by
\[
\alpha_1(a_1,a_2,a_3)=(a_2,a_3,a_1),
\qquad
\alpha_2(a_1,a_2,a_3)=(a_3,a_1,a_2).
\]
Note that $\alpha_1$ and $\alpha_2$ commute.
Consider the representation $\sigma:\mathcal A\to\mathcal B(\mathcal H)$ given by
\[
\sigma(a_1,a_2,a_3)=
\begin{bmatrix}
a_1 & 0 & 0\\
0 & a_2 & 0\\
0 & 0 & a_3
\end{bmatrix}.
\]
Define operators $S_1,S_2\in\mathcal B(\mathcal H)$ by
\[
S_1=
\begin{bmatrix}
0 & 0 & 1\\
1 & 0 & 0\\
0 & 1 & 0
\end{bmatrix},
\qquad
S_2=
\begin{bmatrix}
0 & 1 & 0\\
0 & 0 & 1\\
1 & 0 & 0
\end{bmatrix}.
\]
Then $S_1$ and $S_2$ are permutation matrices, hence unitaries.
A direct computation shows that
\[
\sigma(a)S_i=S_i\sigma(\alpha_i(a)),
\qquad a\in\mathcal A,\ i=1,2.
\]
Moreover, $S_1S_2=S_2S_1=I$. Since $S_1$ and $S_2$ commute and are unitary, it follows that
$S_2^*S_1=S_1S_2^*$.
Therefore, the pair $(S_1,S_2)$ together with $\sigma$ satisfies all the hypotheses of
Theorem~\ref{thm:char-for-Ei-alphaA} (with $U_{12}=I$), and hence determines a (doubly) twisted representation of the associated product system $\mathbb E$.
\end{example}
Recall that the unit polydisc in $\mathbb C^k$ is $\mathbb D^k := \{ z=(z_1,\ldots,z_k)\in\mathbb C^k : |z_i|<1,\ 1\leq i \leq k \}.$
For a Hilbert space $\mathcal D$, the $\mathcal D$-valued Hardy space on the polydisc $\mathbb D^k$, $k \geq 1$ is
\[
H^2_{\mathcal D}(\mathbb D^k)
=
H^2(\mathbb D^k)\otimes \mathcal D
=
\left\{
f(z)=\sum_{n\in\mathbb Z_+^k} z^n \eta_n :
\eta_n\in\mathcal D,\ 
\sum_{n\in\mathbb Z_+^k} \|\eta_n\|^2 < \infty
\right\}.
\]
The monomials $\{z^n\eta : n\in\mathbb Z_+^k,\ \eta\in\mathcal D\}$ form an
orthonormal basis for $H^2_{\mathcal D}(\mathbb D^k)$.
\begin{example}
Fix $\lambda \in \T$, and define $D[\lambda] \in \mathcal{B}(H^2(\D))$ by
\[
D[\lambda] z^m = \lambda^m z^m \qquad (m \in \Z_+).
\]
Consider the Hilbert space $\clh = H^2(\D^2) \oplus H^2(\D^2)$, and isometries $V_1= \mbox{diag}(M_z \otimes I_{H^2(\D)}, D[\lambda] \otimes M_z)$ and $V_2= \mbox{diag}(D[\lambda] \otimes M_z, M_z \otimes I_{H^2(\D)})$ on $\clh$. Then  $(V_1, V_2)$ is a doubly twisted isometry on $\clh$ with $U = \mbox{diag}\left(\bar{\lambda} I_{H^2(\D^2)}, {\lambda}I_{H^2(\D^2)}\right)$ (see \cite[Example~2.1]{RSS}).

 Let \( \mathcal{A}= M_2(\mathbb{C}) \). Define automorphisms \( \alpha_1, \alpha_2 \in \mathrm{Aut}(\mathcal{A}) \) by
\[
\alpha_1(a) = U_1 a U_1^*, \quad \alpha_2(a) = U_2 a U_2^*,
\]
where
\[
U_1 = \begin{bmatrix} 1 & 0 \\ 0 & -1 \end{bmatrix}, \quad U_2 = \begin{bmatrix} 0 & 1 \\ 1 & 0 \end{bmatrix}.
\]
Let \( \mathcal{K} = \mathbb{C}^2 \otimes \mathcal{H} \). Define a representation \( \sigma: \mathcal{A} \to \mathcal{B}(\mathcal{K}) \) by \( \sigma(a) = a \otimes I \). Also, define 
\[
S_1 := U_1 \otimes V_1, \text{~and~} S_2 := U_2 \otimes V_2 \in \mathcal{B}(\mathcal{K}).
\]
Clearly, each \( S_i \) is an isometry. Let $W = \begin{bmatrix} -1 & 0 \\ 0 & -1 \end{bmatrix}$, then
\[
S_1 S_2 = (U_1 \otimes V_1)(U_2 \otimes V_2)=U_1U_2 \otimes V_1V_2= W U_2 U_1 \otimes U V_2 V_1= (W  \otimes U )S_2 S_1.
\]
And 
\[
S^*_1 S_2 = (U_1 \otimes V_1)^*(U_2 \otimes V_2)=U^*_1U_2 \otimes V^*_1V_2= W U_2 U^*_1 \otimes U ^*V_2 V^*_1= (W  \otimes U )^*S_2 S^*_1.
\]
Therefore, the pair $(S_1, S_2)$ is a doubly twisted isometry on $\clh$ with the corresponding twist $W \otimes U.$
Also, $\sigma(a) S_i = a U_i \otimes V_i,$
and
\(
S_i\sigma(\alpha_i(a))= U_i \alpha_i(a) \otimes V_i=U_i U_i a U^*_i \otimes V_i=  a U_i \otimes V_i,
\)
where the last equality follows because $U^2_i= I$ and $U^*_i=U_i$. Thus, the covariance relations \( \sigma(a) S_i = S_i \sigma(\alpha_i(a)) \) are satisfied. This example therefore produces a non-trivial doubly twisted representation
of the product system $\mathbb E$.
\end{example}

In the setup of Theorem~\ref{thm:char-for-Ei-alphaA}, the existence of a
\emph{doubly twisted representation} $(\sigma,T_1,\ldots,T_k)$ with respect to the
twists $\{t_{ij}\otimes U_{ij}\}_{i\neq j}$ is equivalent to the existence of
contractions $S_1,\ldots,S_k\in\mathcal B(\mathcal H)$ as described there
with $T_i(a)=S_i\sigma(a).$ Now we show that if the doubly twisted  representation is isometric, then it
admits explicit Hardy–space models and recover the decomposition of
\cite{RSS}.

\begin{thm}\label{thm:model-automorphic-iff}
Assume the setup of Theorem~\ref{thm:char-for-Ei-alphaA}.
Then $(\sigma,T_1,\ldots,T_k)$ is a \emph{doubly twisted isometric} representation of
the associated automorphic product system $\mathbb E$ on a Hilbert space
$\mathcal H$ with respect to the twists $\{t_{ij}\otimes U_{ij}\}_{i\neq j}$
if and only if there exist an orthogonal decomposition $\mathcal H=\displaystyle\bigoplus_{B\subseteq I_k}\mathcal H_B$
and a unitary $\pi:\mathcal H\longrightarrow
\displaystyle\bigoplus_{B\subseteq I_k} H^2_{\mathcal D_B}(\mathbb D^{|B|})$
such that the tuple
\[
(M_1,\ldots,M_k):=(\pi S_1\pi^*,\ldots,\pi S_k\pi^*)
\]
is the doubly twisted model tuple described in
{\rm\cite[Proposition~4.3]{RSS}} and, for
$\sigma'(a):=\pi\sigma(a)\pi^*$, one has
\[
\sigma'(a)M_i = M_i\sigma'(\alpha_i(a)),
\qquad a\in\mathcal A,\ i\in I_k.
\]
\end{thm}

\begin{proof}
Assume that the representation is isometric. By Theorem~\ref{thm:char-for-Ei-alphaA}, there exist contractions
$S_1,\ldots,S_k\in\mathcal B(\mathcal H)$ such that for all
$a\in\mathcal A$ and $i\neq j$,
\[
\sigma(a)S_i=S_i\sigma(\alpha_i(a)),
\quad
S_iS_j=U_{ij}S_jS_i,
\quad
S_i^*S_j=U_{ij}^*S_jS_i^*.
\]
Since $(\sigma,T_1,\ldots,T_k)$ is isometric, each $S_i$ is an isometry, hence
$(S_1,\ldots,S_k)$ is a $k$-tuple of doubly twisted isometries in the sense of
{\rm\cite{RSS}}.
By {\rm\cite[Theorem~3.6]{RSS}}, $\mathcal H$ admits an orthogonal decomposition
$\mathcal H=\bigoplus_{B\subseteq I_k}\mathcal H_B$ with
\[
\mathcal H_B=\bigoplus_{n\in\mathbb Z_+^{|B|}} S_B^n(\mathcal D_B),
\qquad
\mathcal D_B
=
\bigcap_{l\in\mathbb Z_+^{k-|B|}} S_{I_k\setminus B}^l
\Bigl(\bigcap_{i\in B}\ker S_i^*\Bigr),
\]
and for each $B\subseteq I_k$ there is a unitary
$\pi_B:\mathcal H_B\to H^2_{\mathcal D_B}(\mathbb D^{|B|})$
such that $\pi_B(S_B^n\eta)=z^n\eta$ for all $\eta\in\mathcal D_B$.
Let $B=\{p_1,\ldots,p_m\}$ and $B^c=\{q_1,\ldots,q_{k-m}\}$.
Define $M_{B,t}$ on $H^2_{\mathcal D_B}(\mathbb D^{|B|})$ by
\begin{equation}\label{eq: tilde V_t}
M_{B,t}=
\begin{cases}
M_{z_1}, & t=p_1,\\[3pt]
M_{z_i}\,D_1[U_{p_ip_1}]\cdots D_{i-1}[U_{p_ip_{i-1}}],
& t=p_i,\ 1<i\le m,\\[3pt]
D_1[U_{q_jp_1}]\cdots D_m[U_{q_jp_m}]
\bigl(I_{H^2(\D^m)}\otimes (S_{q_j}|_{\mathcal D_B})\bigr),
& t=q_j,\ 1\le j\le k-m.
\end{cases}
\end{equation}
By {\rm\cite[Proposition~4.3]{RSS}}, $\restr{S}{\mathcal H_B}$ is unitarily equivalent
to $(M_{B,1},\ldots,M_{B,k})$.
Set $M_i:=\displaystyle\bigoplus_{B\subseteq I_k}M_{B,i}$ and
$\pi:=\displaystyle\bigoplus_{B\subseteq I_k}\pi_B$. Then $\pi S_i = M_i\pi$ for all $i$.
Define $\sigma'(a):=\pi\sigma(a)\pi^*$.
Using $\sigma(a)S_i=S_i\sigma(\alpha_i(a))$, we obtain
\[
\sigma'(a)M_i
= \pi\sigma(a)\pi^* M_i
= \pi\sigma(a)S_i\pi^*
= \pi S_i\sigma(\alpha_i(a))\pi^*
= M_i\,\pi\sigma(\alpha_i(a))\pi^*
= M_i\sigma'(\alpha_i(a)),
\]
proving our claim.
The converse follows from Theorem~\ref{thm:char-for-Ei-alphaA}.
\end{proof}

\section{The Wold Decomposition} \label{Wold Decomposition}
Given an \emph{isometric} covariant representation \((\sigma, T)\) of a
\(C^*\)-correspondence \(E\) on a Hilbert space \(\mathcal{H}\),
there exists a decomposition $\mathcal{H} = \mathcal{H}_1 \oplus \mathcal{H}_2,$
where \(\restr{(\sigma, \widetilde{T})}{\mathcal{H}_1}\) is an induced covariant representation
and \(\restr{(\sigma, \widetilde{T})}{\mathcal{H}_2}\) is fully coisometric.
Let $\mathcal{W} := \operatorname{Ran}\!\left(I - \widetilde{T}\widetilde{T}^{*}\right).$
Then, by Theorem~\ref{Thm:Wold-decomposition-sigma-T}, we have
\begin{equation}\label{eq:Projection on H1 and H2}
P_{\mathcal{H}_1}
= \mathrm{SOT}\text{-}\!\sum_{n \in \mathbb{Z}_+}
  \widetilde{T}^{(n)} (I_{E^{(n)}} \otimes P_{\mathcal{W}}) \widetilde{T}^{(n)*},
\qquad
P_{\mathcal{H}_2}
= \mathrm{SOT}\text{-}\!\lim_{n \to \infty}
  \widetilde{T}^{(n)} \widetilde{T}^{(n)*}.
\end{equation}

Now consider a doubly twisted isometric covariant representation $(\sigma, T^{(1)}, \dots, T^{(k)})$ of $\mathbb{E}$ on $\mathcal{H}$.  
Then, by Theorem~\ref{thm: decompo for doubly twisted rep}, this representation admits a Wold-type decomposition, that is, $\mathcal{H} =\displaystyle \bigoplus_{A \subseteq I_k} \mathcal{H}_A.$
We now focus on describing the representations corresponding to each subspace $\mathcal{H}_A$.
For each $i \in I_k$, define 
\[
\mathcal{W}_i := \operatorname{Ran}\left(I - \widetilde{T}_i \widetilde{T}_i^{*}\right).
\]
Let $P_{\mathcal W_i}:=I-\widetilde T_i\widetilde T_i^*$ denote the orthogonal projection onto $\mathcal W_i$.

\begin{lemma}\label{lemma:  commutativity of Pwi}
The family $\{P_{\mathcal W_i}\}_{i \in I_k}$ consists of commuting orthogonal projections.
\end{lemma}

\begin{proof}
For $i,j \in I_k$, by the twisted and doubly twisted relations we have
\begin{align*}
 (\widetilde T_i \widetilde T_i^{*})(\widetilde T_j \widetilde T_j^{*})
&= \widetilde T_i (I_{E_i} \otimes \widetilde T_j) (t_{ji} \otimes U_{ji})
  (I_{E_j} \otimes \widetilde T_i^{*}) \widetilde T_j^{*}\\
&=\widetilde T_j (I_{E_j} \otimes \widetilde T_i)
  (I_{E_j} \otimes \widetilde T_i^{*}) \widetilde T_j^{*}\\
  &=\widetilde T_j (I_{E_j} \otimes \widetilde T_i) (t_{ij} \otimes U_{ij})
  (I_{E_i} \otimes \widetilde T_j^{*}) \widetilde T_i^{*}\\
&= (\widetilde T_j \widetilde T_j^{*})(\widetilde T_i \widetilde T_i^{*}).   
\end{align*}

Hence the projections $P_{\mathcal W_i} = I - \widetilde T_i \widetilde T_i^{*}$ commute pairwise:
\[
(I - \widetilde T_i \widetilde T_i^{*})(I - \widetilde T_j \widetilde T_j^{*})
= (I - \widetilde T_j \widetilde T_j^{*})(I - \widetilde T_i \widetilde T_i^{*}).
\qedhere
\]
\end{proof}

Let $\mathcal W_{\varnothing} = \mathcal H$. For each non-empty subset \(A \subseteq I_k\), the above lemma allows us to define
\begin{equation}\label{eq:WA-and-PwA-as-product}
\mathcal{W}_A
:= \operatorname{Ran}\!\Bigg(\prod_{i \in A} (I - \widetilde T_i \widetilde T_i^{*})\Bigg)
= \bigcap_{i \in A} \operatorname{Ran}(I - \widetilde T_i \widetilde T_i^{*})
= \bigcap_{i \in A} \mathcal{W}_i,
\qquad
P_{\mathcal{W}_A} = \prod_{i \in A} P_{\mathcal{W}_i}.
\end{equation}
To analyze the structure of the subspaces $\mathcal H_A$, we first observe that
if $\mathcal K$ is a closed subspace of $\mathcal H$ that \emph{reduces}
(see Definition~\ref{def:reducing subspace})
a doubly twisted covariant representation
$(\sigma,T_1,\dots,T_k)$ of $\mathbb E$ on $\mathcal H$,
then the restricted tuple
\(
\restr{(\sigma,T_1,\dots,T_k)}{\mathcal K}
\)
is again a doubly twisted covariant representation of $\mathbb E$ on $\mathcal K$.
Also, since both $\mathcal K$ and $\mathcal K^{\perp}$ are left invariant under
each operator $T_i(\xi_i)$, for $i\in I_k$ and $\xi_i\in E_i$, it follows that
\[
\widetilde T_i^{*}\mathcal K \subseteq E_i\otimes\mathcal K,
\qquad
\widetilde T_i^{*}\mathcal K^{\perp}
\subseteq E_i\otimes\mathcal K^{\perp}.
\]
Moreover, the twisted relation yields
\(
t_{ij}\otimes U_{ij}
= (I_{E_j}\otimes \widetilde T_i^{*})\,\widetilde T_j^{*}\,
  \widetilde T_i\,(I_{E_i}\otimes \widetilde T_j).
\)
Consequently, if $\mathcal K$ reduces
$(\sigma,T_1,\dots,T_k)$, then
\begin{equation}\label{eq: reducing under Uij}
(t_{ij}\otimes U_{ij})(E_i\otimes E_j\otimes \mathcal K)
\subseteq E_j\otimes E_i\otimes \mathcal K,
\quad
(t_{ij}\otimes U_{ij})(E_i\otimes E_j\otimes \mathcal K^{\perp})
\subseteq E_j\otimes E_i\otimes \mathcal K^{\perp}.
\end{equation}

\begin{lemma}\label{lemma:WA-reduces-Tj}
Let \(A \subsetneq I_k\). Then \(\mathcal{W}_A\) reduces \(\widetilde T_j\) for every \(j \in A^{\mathrm{c}}\).
\end{lemma}

\begin{proof}
For \(i \in A\) and \(j \in A^{\mathrm{c}}\), the doubly twisted relations yield
$$(\widetilde T_i \widetilde T^*_i)\widetilde T_j 
= \widetilde T_i (I_{E_i} \otimes \widetilde T_j) (t_{ji} \otimes U_{ji})  (I_{E_j} \otimes  \widetilde T^*_i)
=\widetilde T_j (I_{E_j} \otimes \widetilde T_i)(I_{E_j} \otimes  \widetilde T^*_i) 
.$$
Therefore, $P_{\mathcal W_i}\widetilde T_j
= \widetilde T_j (I_{E_j} \otimes P_{\mathcal W_i}),$ and Equation \eqref{eq:WA-and-PwA-as-product} implies that \(\mathcal W_A\) reduces \(\widetilde T_j\), as
\begin{equation} \label{eq:Tj-commutes-with-PwA}
    P_{\mathcal W_A}\widetilde T_j
= \widetilde T_j (I_{E_j} \otimes P_{\mathcal W_A}).
\end{equation}
\end{proof}
For a doubly twisted isometric representation 
\((\sigma, T_1, \dots, T_k)\) of \(\mathbb{E}\) on \(\mathcal{H}\),
let \(\mathcal{H}_i^1\) and \(\mathcal{H}_i^2\) be the reducing subspaces such that
\(\restr{(\sigma,  T_i)}{\mathcal{H}_i^1}\) is an induced covariant representation,
and \(\restr{(\sigma,  T_i)}{\mathcal{H}_i^2}\) is fully coisometric.
Then we have:
\begin{lemma}\label{lemma:commutativity of PHj with PwA}
Let \(A\) be a nonempty proper subset of \(I_k\).
Then, for every \(j \in A^c\),
\[
P_{\mathcal{H}_j^1},\; P_{\mathcal{H}_j^2} \in \{P_{\mathcal{W}_A}\}'.
\]
\end{lemma}

\begin{proof}
From Equation~\ref{eq:Tj-commutes-with-PwA}, for any \(n \in \mathbb{Z}_+\) and \(j \in A^{\mathrm{c}}\),
\begin{equation}\label{eq:commutativity PwA Tjn}
P_{\mathcal{W}_A}\widetilde{T}_j^{(n)}
= \widetilde{T}_j^{(n)}\left(I_{E_j^{n}} \otimes P_{\mathcal{W}_A}\right),
\text{ and~}
\widetilde{T}_j^{(n)*}P_{\mathcal{W}_A}
= \left(I_{E_j^{n}} \otimes P_{\mathcal{W}_A}\right)\widetilde{T}_j^{(n)*}.
\end{equation}
Consequently,
\[
\widetilde{T}_j^{(n)}\widetilde{T}_j^{(n)*}P_{\mathcal{W}_A}
= \widetilde{T}_j^{(n)}\left(I_{E_j^{n}}\otimes P_{\mathcal{W}_A}\right)\widetilde{T}_j^{(n)*}
= P_{\mathcal{W}_A}\widetilde{T}_j^{(n)}\widetilde{T}_j^{(n)*}.
\]
\eqref{eq:Projection on H1 and H2} implies that \(P_{\mathcal{H}_j^2} \in \{P_{\mathcal{W}_A}\}'\).
Since \(P_{\mathcal{H}_j^1} = I - P_{\mathcal{H}_j^2}\),
it follows that \(P_{\mathcal{H}_j^1} \in \{P_{\mathcal{W}_A}\}'\).
\end{proof}

\begin{lemma}\label{lemma:commutativity-PHi1-PHj2}
For any \(i, j \in I_k\), we have $P_{\mathcal{H}_i^{1}} P_{\mathcal{H}_j^{2}}
= P_{\mathcal{H}_j^{2}} P_{\mathcal{H}_i^{1}}.$
\end{lemma}

\begin{proof}
For each \(i \in I_k\), clearly $P_{\mathcal{H}_i^{1}} P_{\mathcal{H}_i^{2}} = 0 = P_{\mathcal{H}_i^{2}} P_{\mathcal{H}_i^{1}}.$
Assume \(i \neq j\).  
By Theorem~\ref{thm: decompo for doubly twisted rep}, the tuple
\((\sigma, T^{(1)}, \ldots, T^{(k)})\) admits a Wold decomposition.
Then Theorem~\ref{thm: class of decom} implies that
\(\mathcal{H}_j^{1}\) reduces \(\widetilde{T}_i\); that is, $P_{\mathcal{H}_j^{1}}\widetilde{T}_i
= \widetilde{T}_i (I_{E_i} \otimes P_{\mathcal{H}_j^{1}}).$
Hence for any $m \in \mathbb{Z}_+$,
\begin{equation} \label{eq: Tmi reduces H1j}
   P_{\clh^1_j} \,\widetilde{T}_i^{(m)} 
   = \widetilde{T}_i^{(m)} \left(I_{E_i^{ m}} \otimes P_{\clh^1_j}\right),
   \qquad
   \widetilde{T}_i^{(m)*} P_{\clh^1_j} 
   = \left(I_{E_i^{ m}} \otimes P_{\clh^1_j}\right)\widetilde{T}_i^{(m)*}.
\end{equation}
Using \eqref{eq:Projection on H1 and H2} and~\eqref{eq: Tmi reduces H1j}, together with Lemma~\ref{lemma:commutativity of PHj with PwA}, we obtain
\[
P_{\mathcal{H}_i^{1}} P_{\mathcal{H}_j^{1}}
= \left(\sum_{m \in \mathbb{Z}_+}
  \widetilde{T}_i^{(m)}(I_{E_i^{m}} \otimes P_{\mathcal{W}_i})\widetilde{T}_i^{(m)*} \right)P_{\mathcal{H}_j^{1}} =P_{\mathcal{H}_j^{1}} P_{\mathcal{H}_i^{1}}.
\]
A similar argument yields $P_{\mathcal{H}_i^{1}} P_{\mathcal{H}_j^{2}}
= P_{\mathcal{H}_j^{2}} P_{\mathcal{H}_i^{1}},$ and $P_{\mathcal{H}_i^{2}} P_{\mathcal{H}_j^{2}}
= P_{\mathcal{H}_j^{2}} P_{\mathcal{H}_i^{2}}.$
\end{proof}
Recall that $(1)$ of Lemma \ref{lemma:intertwining of U_{ij}} says that for all $i,j,m \in I_k$ with $i\ne j$ and for every $n\in\mathbb Z_+$, $U_{ij}\, \widetilde{T}_m^{(n)}
= \widetilde{T}_m^{(n)}\, (I_{E_m^{\,n}} \otimes U_{ij}).$ This further implies $ \widetilde{T}_m^{(n)*} U_{ij}
= (I_{E_m^{\,n}} \otimes U_{ij})\widetilde{T}_m^{(n)*}.$ Also for $m,n \in \Z_+,$ recall the map \(
t_{ij}^{(m,n)}\) \eqref{eq:tijmn} , that flips 
$E_i^{\,m} \otimes E_j^{\,n}$ to $
E_j^{\,n} \otimes E_i^{\,m}.$

\begin{lemma}\label{lemma: commutativity of Ti *m and Tj n}
Let $(\sigma,T_1,\dots,T_k)$ be a doubly twisted representation of the product
system $\{E_i\}_{i=1}^k$.  
For $i\neq j$ and $p,q\in\mathbb Z_+$,
\begin{equation}\label{eq:DTR-relation}
    \widetilde{T}_i^{(p)*}\,\widetilde{T}_j^{\,(q)}
= 
\left(I_{E_i^{\,p}} \otimes \widetilde{T}_j^{\,(q)}\right)
  \left(t_{ji}^{(q,p)} \otimes U_{ji}^{qp}\right)
  \left(I_{E_j^{\,q}} \otimes \widetilde{T}_i^{(p)*}\right).
\end{equation}
\end{lemma}

\begin{proof}
We prove the identity in (\ref{eq:DTR-relation}) by induction on $r(p,q):=p+q.$
If $p=0$ or $q=0$, the identity is immediate since
$t_{ji}^{(q,0)}=I_{E_j^q}$, $U_{ji}^{q0}=I$, and similarly for $t_{ji}^{(0,p)}$.
When $p=q=1$, the statement reduces to the defining twisted covariance
relation:
\[
\widetilde{T}_i^*\widetilde{T}_j
=
\left(I_{E_i}\otimes \widetilde{T}_j\right)
(t_{ji}\otimes U_{ji})
\left(I_{E_j}\otimes \widetilde{T}_i^*\right).
\]
Assume (\ref{eq:DTR-relation}) holds for every pair $(p',q')$ with $r(p',q')<r(p,q)$.
We treat the case $p\ge1$; the case $q\ge1$ is symmetric.
Since $E_i^p = E_i\otimes E_i^{p-1}$, $\widetilde{T}_i^{(p)*}
=
\left(I_{E_i^{p-1}}\otimes \widetilde{T}_i^*\right)\,
\widetilde{T}_i^{(p-1)*}.$
Apply the induction hypothesis to the pair $(p-1,q)$:
\[
\widetilde{T}_i^{(p-1)*}\widetilde{T}_j^{\,(q)}
=
\left(I_{E_i^{p-1}}\otimes \widetilde{T}_j^{\,(q)}\right)
\left(t_{ji}^{(q,p-1)}\otimes U_{ji}^{q(p-1)}\right)
\left(I_{E_j^{q}}\otimes \widetilde{T}_i^{(p-1)*}\right).
\]
Substituting this and using Equation \eqref{eq: sec def of tijmn} yields
\begin{align*}
&\widetilde{T}_i^{(p)*}\widetilde{T}_j^{\,(q)}\\
&=
\left(I_{E_i^{p-1}}\otimes \widetilde{T}_i^*\right)
\left(I_{E_i^{p-1}}\otimes \widetilde{T}_j^{\,(q)}\right)
\left(t_{ji}^{(q,p-1)}\otimes U_{ji}^{q(p-1)}\right)
\left(I_{E_j^{q}}\otimes \widetilde{T}_i^{(p-1)*}\right) \\
&=
\left(I_{E_i^{p}}\otimes \widetilde{T}_j^{\,(q)}\right)
\left(I_{E_i^{p-1}}\otimes  t_{ji}^{(q,1)}\otimes U_{ji}^{q} \right)
\left(I_{E_i^{p-1} \otimes E^q_j}\otimes \widetilde{T}_i^*\right)
\left(t_{ji}^{(q,p-1)}\otimes U_{ji}^{q(p-1)}\right)
\left(I_{E_j^{q}}\otimes \widetilde{T}_i^{(p-1)*}\right) \\
&=
\left(I_{E_i^{p}}\otimes \widetilde{T}_j^{\,(q)}\right)
\left(I_{E_i^{p-1}}\otimes  t_{ji}^{(q,1)}\otimes U_{ji}^{q} \right)
\left(t_{ji}^{(q,p-1)}\otimes I_{E_i} \otimes U_{ji}^{q(p-1)}\right)
\left(I_{ E_j^q \otimes E_i^{p-1}}\otimes \widetilde{T}_i^*\right)
\left(I_{E_j^{q}}\otimes \widetilde{T}_i^{(p-1)*}\right) \\
&=
\left(I_{E_i^{p}}\otimes \widetilde{T}_j^{\,(q)}\right)
\left(  t_{ji}^{(q,p)}\otimes U_{ji}^{qp} \right)
\left(I_{E_j^{q}}\otimes \widetilde{T}_i^{(p)*}\right).
\end{align*}
\end{proof}

\begin{lemma}\label{lemma: Simplify}
Let $A \subseteq I_k$ be nonempty, and 
$m = (m_1, m_2, \ldots, m_{|A|}) \in \mathbb{Z}_+^{|A|}$. 
Then
\[
\prod_{i \in A} 
\left(
\widetilde{T}_i^{(m_i)}
  (I_{E_i^{\,m_i}} \otimes P_{\mathcal{W}_i})
  \widetilde{T}_i^{(m_i)*}
\right)
\;=\;
\widetilde{T}_A^{(m)}
  \left(I_{E_A^{\,m}} \otimes P_{\mathcal{W}_A}\right)
  \widetilde{T}_A^{(m)*}.
\]
\end{lemma}
\begin{proof}
We prove the result for \(A = \{i,j\}\); the general case follows by induction.  
Fix \(m = (m_i, m_j) \in \mathbb{Z}_+^2\).  
We need to show that
\[
\left(\widetilde{T}_i^{(m_i)}(I_{E_i^{m_i}} \otimes P_{\mathcal{W}_i})\widetilde{T}_i^{(m_i)*}\right)
\left(\widetilde{T}_j^{(m_j)}(I_{E_j^{m_j}} \otimes P_{\mathcal{W}_j})\widetilde{T}_j^{(m_j)*}\right)
= \widetilde{T}_A^{(m)}(I_{E_A^{m}} \otimes P_{\mathcal{W}_A})\widetilde{T}_A^{(m)*}.
\]
Let X denote the left-hand side of the above equation. Then by Lemma~\ref{lemma: commutativity of Ti *m and Tj n},
\begin{align*}
X
&= \widetilde{T}_i^{(m_i)}\left(I_{E_i^{m_i}} \otimes P_{\mathcal{W}_i}\right)
\left(I_{E_i^{m_i}} \otimes \widetilde{T}_j^{(m_j)}\right)
\left(t_{ji}^{(m_j,m_i)} \otimes U_{ji}^{m_j m_i}\right)
\left(I_{E_j^{m_j}} \otimes \widetilde{T}_i^{(m_i)*}\right)
\left(I_{E_j^{m_j}} \otimes P_{\mathcal{W}_j}\right)\widetilde{T}_j^{(m_j)*}\\
&= \widetilde{T}_A^{(m)}\left(I_{E_A^{m}} \otimes P_{\mathcal{W}_i}\right)
\left(t_{ji}^{(m_j,m_i)} \otimes U_{ji}^{m_j m_i}\right)
\left(I_{E_j^{m_j}} \otimes \widetilde{T}_i^{(m_i)*}\right)
\left(I_{E_j^{m_j}} \otimes P_{\mathcal{W}_j}\right)\widetilde{T}_j^{(m_j)*} \quad (\text{by }\eqref{eq:commutativity PwA Tjn})\\
&=\widetilde{T}_A^{(m)}\left(I_{E_A^{m}} \otimes P_{\mathcal{W}_i}\right)
\left(t_{ji}^{(m_j,m_i)} \otimes U_{ji}^{m_j m_i}\right)\left(I_{E_j^{m_j} \otimes E_i^{m_i}} \otimes P_{\clw_j}\right)
    \left(I_{E_j^{m_j}} \otimes \widetilde{T}_i^{(m_i)*}\right)
    \widetilde{T}_j^{(m_j)*}
    \quad(\text{by }\eqref{eq:commutativity PwA Tjn})\\
&=\widetilde{T}_A^{(m)}\left(I_{E_A^{m}} \otimes P_{\mathcal{W}_i}\right)
\left(t_{ji}^{(m_j,m_i)} \otimes U_{ji}^{m_j m_i}\right)\left(I_{E_j^{m_j} \otimes E_i^{m_i}} \otimes P_{\clw_j}\right)
     \left(\widetilde{T}_j^{(m_j)}(I_{E_j^{m_j}} \otimes \widetilde{T}_i^{(m_i)})\right)^* \\
&=\widetilde{T}_A^{(m)}\left(I_{E_A^{m}} \otimes P_{\mathcal{W}_i}\right)
\left(t_{ji}^{(m_j,m_i)} \otimes U_{ji}^{m_j m_i}\right)\left(I_{E_j^{m_j} \otimes E_i^{m_i}} \otimes P_{\clw_j}\right)
     \left(t_{ji}^{\,{m_j}{m_i}} \otimes U_{ji}^{\,{m_j}{m_i}}\right)^*
     \left(\widetilde{T}_A^{(m)*}\right)\\
     &= \widetilde{T}_A^{(m)}\left(I_{E_A^{m}} \otimes P_{\mathcal{W}_i}\right)
\left(I_{E_A^{m}} \otimes P_{\mathcal{W}_j}\right)
(\widetilde{T}_A^{(m)*})
\qquad \text{(Lemma~\ref{lemma:intertwining of U_{ij}})}\\
&= \widetilde{T}_A^{(m)}(I_{E_A^{m}} \otimes P_{\mathcal{W}_A})(\widetilde{T}_A^{(m)*})
\quad (\text{by }\eqref{eq:WA-and-PwA-as-product}).
\end{align*}
Thus, the desired equality holds.
\end{proof}

\begin{lemma}\label{twisted intersection}
Let $\cls$ be a reducing subspace for $(\sigma, T_1, \ldots, T_k)$. Then
\[
\bigcap_{m \in \mathbb{Z}_+^{|A|}} \widetilde{T}_A^{(m)}(I_{E_A^{m}} \otimes \cls)
= \bigcap_{i \in A}\; \bigcap_{m_i \in \mathbb{Z}_+}
\widetilde{T}_i^{(m_i)}(I_{E_i^{m_i}} \otimes \cls).
\]
\end{lemma}

\begin{proof}
It suffices to prove the case $A = \{1,2\}$; the general case follows by induction.  
Let $(\sigma, T_1, T_2)$ be a doubly twisted representation with twist $U$.  
By \eqref{eq: reducing under Uij},
\[
\bigcap_{m \in \mathbb{Z}_+^{|A|}} \widetilde{T}_A^{(m)}\left(I_{E_A^{m}} \otimes \cls\right)
\subseteq 
\bigcap_{i \in A}\; \bigcap_{m_i \in \mathbb{Z}_+}
\widetilde{T}_i^{(m_i)}\left(I_{E_i^{m_i}} \otimes \cls\right).
\]
For the reverse inclusion, take
\[
y \in
\left(\!\bigcap_{m_1 \in \mathbb{Z}_+} \widetilde{T}_1^{(m_1)}\left(I_{E_1^{m_1}} \otimes \cls\right)\!\right)
\cap
\left(\!\bigcap_{m_2 \in \mathbb{Z}_+} \widetilde{T}_2^{(m_2)}\left(I_{E_2^{m_2}} \otimes \cls\right)\!\right).
\]
Then for $m=(m_1,m_2)\in\mathbb{Z}_+^2$, there exist $e_i\in E_i^{m_i}$ and $x_i\in\cls$ such that  
$y = \widetilde{T}_1^{(m_1)}(e_1\!\otimes\!x_1)
   = \widetilde{T}_2^{(m_2)}(e_2\!\otimes\!x_2)$.
Hence Lemma~\ref{lemma: commutativity of Ti *m and Tj n} implies
\[
e_1 \otimes x_1
= \widetilde{T}_1^{(m_1)*}\widetilde{T}_2^{(m_2)}\left(e_2 \otimes x_2\right)
= \left(I_{E_1^{m_1}} \otimes \widetilde{T}_2^{(m_2)}\right)
  \left(t_{21}^{(m_2,m_1)} \otimes U_{21}^{m_2m_1}\right)
  \left(I_{E_2^{m_2}} \otimes \widetilde{T}_1^{(m_1)*}\right)(e_2 \otimes x_2).
\] 
Since $\cls$ reduces $\widetilde{T}_1$,  
\(\left(I_{E_2^{m_2}} \otimes \widetilde{T}_1^{(m_1)*}\right)(e_2 \otimes x_2)
   \in E_2^{m_2} \otimes E_1^{m_1} \otimes \cls.\)
By \eqref{eq: reducing under Uij},  
\[
\left(t_{21}^{(m_2,m_1)} \otimes U_{21}^{m_2m_1}\right)
 \left(I_{E_2^{m_2}} \otimes \widetilde{T}_1^{(m_1)*}\right)(e_2 \otimes x_2)
 \in E_1^{m_1} \otimes E_2^{m_2} \otimes \cls.
\]
Let this element be $\eta \otimes x$ with $\eta \in E_1^{m_1} \otimes E_2^{m_2}$ and $x \in \cls$.  
Then
\[
y = \widetilde{T}_1^{(m_1)}\left(I_{E_1^{m_1}} \otimes \widetilde{T}_2^{(m_2)}\right)(\eta \otimes x) = \widetilde{T}_A^{(m)}(\eta \otimes x).
\]
As $m$ was arbitrary, therefore, $y \in \bigcap_{m \in \mathbb{Z}_+^2} \widetilde{T}_A^{(m)}\left(I_{E_A^{m}} \otimes \cls\right).$
Hence,
\[
\bigcap_{i \in A}\; \bigcap_{m_i \in \mathbb{Z}_+}
\widetilde{T}_i^{(m_i)}\left(I_{E_i^{m_i}} \otimes \cls\right)
\subseteq
\bigcap_{m \in \mathbb{Z}_+^{|A|}}
\widetilde{T}_A^{(m)}\left(I_{E_A^{m}} \otimes \cls\right).
\]
\end{proof}

Theorem~\ref{thm: decompo for doubly twisted rep} guarantees a Wold decomposition for any 
doubly twisted isometric covariant representation 
\((\sigma, T_1, \ldots, T_k)\) of \(\mathbb{E}\) on a Hilbert space \(\clh\).  
We now describe the representations of the summands \(\clh_A\) in the Wold decomposition $\clh =\displaystyle \bigoplus_{A \subseteq I_k} \clh_A.$

\begin{theorem}\label{thm: representation of summands}
For each subset \(A \subseteq I_k\), we have
\[
\clh_A
= \bigoplus_{n \in \mathbb{Z}_+^{|A|}}
\widetilde{T}_A^{(n)}
\left(I_{E_A^{n}} \otimes
   \bigcap_{m \in \mathbb{Z}_+^{k - |A|}}
   \widetilde{T}_{I_k \setminus A}^{(m)}
   \left(I_{E_{I_k \setminus A}^{m}} \otimes \clw_A\right)
\right).
\]
\end{theorem}

\begin{proof}
From (\ref{eqn: rep of Wold dec H A}), we have $\clh_A = 
\Big[\bigcap_{i \in A} \clh_i^1\Big]
\cap
\Big[\bigcap_{j \in A^c} \clh_j^2\Big].$ Fix a non-empty subset $A $ of $ I_k$, the Lemma~\ref{lemma:commutativity-PHi1-PHj2} gives
\[
\clh_A = 
\left(\prod_{i \in A} P_{\clh_i^1}\right)
\left(\prod_{j \in A^c} P_{\clh_j^2}\right)\clh.
\]
For each \(i \in A\),
\begin{align*}
\prod_{i \in A} P_{\clh_i^1}
&= \sum_{n \in \mathbb{Z}_+^{|A|}}
   \prod_{i \in A}
   \widetilde{T}_i^{(n_i)}\left(I_{E_i^{n_i}} \otimes P_{\clw_i}\right)
   \widetilde{T}_i^{(n_i)*}  \qquad (\text{by }\eqref{eq:Projection on H1 and H2})\\
&= \sum_{n \in \mathbb{Z}_+^{|A|}}
   \widetilde{T}_A^{(n)}\left(I_{E_A^{n}} \otimes P_{\clw_A}\right)
   \widetilde{T}_A^{(n)*}
   \qquad (\text{by }\text{Lemma~\ref{lemma: Simplify}}).
\end{align*}

If \(A = I_k\), then
\(
\clh_{I_k}
=\displaystyle \prod_{i \in I_k} P_{\clh_i^1}(\clh)
= \bigoplus_{n \in \mathbb{Z}_+^k}
  \widetilde{T}^{(n)}(I_{E^{n}} \otimes \clw_{I_k}).
\)

Now suppose \(A \subsetneq I_k\).  
Then
\begin{align*}
\left(\prod_{i \in A} P_{\clh_i^1}\right)
\left(\prod_{j \in A^c} P_{\clh_j^2}\right)
&= \left(\sum_{n \in \mathbb{Z}_+^{|A|}}
   \widetilde{T}_A^{(n)}
   \left(I_{E_A^{n}} \otimes P_{\clw_A}\right)
   \widetilde{T}_A^{(n)*} \right)  \prod_{j \in A^c} P_{\clh_j^2}\\
&= \sum_{n \in \mathbb{Z}_+^{|A|}}
   \widetilde{T}_A^{(n)}
   \left(I_{E_A^{n}} \otimes P_{\clw_A}\right)
   \left(I_{E_A^{n}} \otimes \prod_{j \in A^c} P_{\clh_j^2}\right)
   \widetilde{T}_A^{(n)*} \quad (\text{by }\eqref{eq: Tmi reduces H1j}) \\
&= \sum_{n \in \mathbb{Z}_+^{|A|}}
   \widetilde{T}_A^{(n)}
   \left(I_{E_A^{n}} \otimes
   \prod_{j \in A^c} P_{\clh_j^2}\right)
   \left(I_{E_A^{n}} \otimes P_{\clw_A}\right)
   \widetilde{T}_A^{(n)*} \quad(\text{by  Lemma~\ref{lemma:commutativity of PHj with PwA}}).\\
\end{align*}
Therefore, Lemma~\ref{twisted intersection} implies 
\begin{align*}
 \clh_A
&=
\left(
  \sum_{n \in \mathbb{Z}_+^{|A|}}
  \widetilde{T}_A^{(n)}
  \left(I_{E_A^{n}} \otimes
   \left(\prod_{j \in A^c} P_{\clh_j^2}\right)( P_{\clw_A})
\right)\right) \clh\\  
&= \sum_{n \in \mathbb{Z}_+^{|A|}}
\widetilde{T}_A^{(n)}
\left(I_{E_A^{n}} \otimes
   \bigcap_{m \in \mathbb{Z}_+^{k - |A|}}
   \widetilde{T}_{I_k \setminus A}^{(m)}
   \left(I_{E_{I_k \setminus A}^{m}} \otimes \clw_A\right)
\right).\\
\end{align*}

Finally, when \(A = \emptyset\), 
\(\clw_A = \clh\), and hence $\clh_{\emptyset}
= \bigcap_{m \in \mathbb{Z}_+^{k}}
  \widetilde{T}^{(m)}(I_{E^{m}} \otimes \clh).$
\end{proof}

Let $\mathcal D_A
:=\displaystyle \bigcap_{m\in\mathbb Z_+^{k-|A|}}
\widetilde T_{I_k\setminus A}^{(m)}
\left(I_{E_{I_k\setminus A}^{m}}\otimes \mathcal W_A\right).$
Then the above theorem allows us to express
\[
\clh_A 
= \bigoplus_{n \in \mathbb{Z}_+^{|A|}}
\widetilde{T}_A^{(n)}
\left(I_{E_A^{n}} \otimes \cld_A\right).
\]
The Wold decomposition (Theorem~\ref{Wold-Decomposition}) ensures that each 
$\mathcal H_A$ reduces $\sigma(\mathcal A)$.

\section{Model Operators on the Fock Space} \label{Models}
Recall the one-variable situation, suppose \(V \in \clb(\clh)\) is an isometry. Then, by the classical Wold decomposition, we have the orthogonal splitting $\clh = \clh_{\{1\}} \oplus \clh_{\emptyset},$
where $\clh_{\{1\}} = \displaystyle\bigoplus_{j=0}^{\infty} V^j (\ker V^*),$ and $
\displaystyle\clh_{\emptyset} = \bigcap_{j=0}^{\infty} V^j \clh.$ Let $\clw= \ker V^*.$
With the \emph{canonical unitary} $\Pi : \clh_{\{1\}} \longrightarrow H^2_{\clw}(\D),\, 
\Pi(V^m \eta) = z^m \eta,\, m \in \Z_+, \ \eta \in \clw,$
we get the intertwining relation $\Pi \restr{V}{\clh_{\{1\}}} =M_z \Pi.$
This illustrates the classical model for an isometry on a Hilbert space, where the Wold decomposition separates the shift and unitary parts.

Given a doubly twisted isometric covariant representation 
\((\sigma, T_1, \ldots, T_k)\) of the product system \(\mathbb{E}\) on a Hilbert space \(\clh\), 
we obtain the Wold-type orthogonal decomposition.
Analogously, for the multivariable doubly twisted representation, if we restrict to 
\(\restr{(\sigma, T_1, \ldots, T_k)}{\clh_A}\),
we again obtain a doubly twisted isometric covariant representation such that \(\restr{(\sigma, T_i)}{\clh_A}\) is induced for \(i \in A\), and 
\(\restr{(\sigma, T_j)}{\clh_A}\) is fully coisometric for \(j \in A^c\).
In this section, our aim is to construct a unitarily equivalent model for each such summand, consisting of creation operators (corresponding to the induced part) and unitary operators (corresponding to the coisometric part) acting on the appropriate Fock space. 

In order to build the model operators on the Fock space, we begin by defining several key maps and proving a few lemmas.
\begin{lemma}\label{lemma:commutativity-Uijn-Tlm}
Let $i,j,l\in I_k$ with $i \neq j.$ Then for every $m,n\in\mathbb Z_+$, $U_{ij}^{\,n}\,\widetilde{T}_l^{(m)}
\;=\;
\widetilde{T}_l^{(m)}\,\left(I_{E_l^{\,m}}\otimes U_{ij}^{\,n}\right).$
\end{lemma}

\begin{proof}
We proceed in two stages.
For each fixed $n\in\mathbb Z_+$, first we prove that
\begin{equation}\label{eq:UnTl}
U_{ij}^{\,n}\,\widetilde{T}_l \;=\; \widetilde{T}_l\,\left(I_{E_l}\otimes U_{ij}^{\,n}\right).
\end{equation}
We argue by induction on $n$. The case $n=1$ follows from the definition.
Assume \eqref{eq:UnTl} holds for some $n\ge1$. Then
\[
U_{ij}^{\,n+1}\,\widetilde{T}_l
= U_{ij}\,\left(U_{ij}^{\,n}\,\widetilde{T}_l\right)
= U_{ij}\,\widetilde{T}_l\,\left(I_{E_l}\otimes U_{ij}^{\,n}\right)
= \widetilde{T}_l\,\left(I_{E_l}\otimes U_{ij}\right)\,\left(I_{E_l}\otimes U_{ij}^{\,n}\right)
= \widetilde{T}_l\,\left(I_{E_l}\otimes U_{ij}^{\,n+1}\right),
\]
which proves \eqref{eq:UnTl} for all $n$.

Now for each fixed $n\in\mathbb Z_+$, we prove the claim for all $m\in\mathbb Z_+$ by induction on $m$.
For $m=1$, this is exactly \eqref{eq:UnTl}.
Assume for some $m\ge1$ that
\[
U_{ij}^{\,n}\,\widetilde{T}_l^{(m)}
= \widetilde{T}_l^{(m)}\,\left(I_{E_l^{\,m}}\otimes U_{ij}^{\,n}\right).
\]
Using the recursion $\widetilde{T}_l^{(m+1)}=\widetilde{T}_l^{(m)}\,\left(I_{E_l^{\,m}}\otimes \widetilde{T}_l\right)$, we compute
\[
\begin{aligned}
U_{ij}^{\,n}\,\widetilde{T}_l^{(m+1)}
&= U_{ij}^{\,n}\,\widetilde{T}_l^{(m)}\left(I_{E_l^{\,m}}\otimes \widetilde{T}_l\right) \\
&= \widetilde{T}_l^{(m)}\left(I_{E_l^{\,m}}\otimes U_{ij}^{\,n}\right)\,\left(I_{E_l^{\,m}}\otimes \widetilde{T}_l\right)
\qquad\text{(by induction hypothesis)} \\
&= \widetilde{T}_l^{(m)}\,\left(I_{E_l^{\,m}}\otimes (U_{ij}^{\,n}\,\widetilde{T}_l)\right) \\
&= \widetilde{T}_l^{(m)}\,\left(I_{E_l^{\,m}}\otimes \widetilde{T}_l\,(I_{E_l}\otimes U_{ij}^{\,n})\right)
\qquad\text{(by \eqref{eq:UnTl})} \\
&= \widetilde{T}_l^{(m)}\,(I_{E_l^{\,m}}\otimes \widetilde{T}_l)\,(I_{E_l^{\,m+1}}\otimes U_{ij}^{\,n}) \\
&= \widetilde{T}_l^{(m+1)}\,\left(I_{E_l^{\,m+1}}\otimes U_{ij}^{\,n}\right).
\end{aligned}
\]
This completes the induction on $m$. Since $n$ was arbitrary, the statement holds for all $m,n\in\mathbb Z_+$.
\end{proof}

\begin{lemma}\label{lemma:commutativity-tij-Uij-Tl}
Let $i,j,l\in I_k$ with $i \neq j$, and $n,m\in\mathbb Z_+$. 
Then the following commutation relation holds:
\[
\left(t_{ij}^{(n)} \otimes U_{ij}^{\,n}\right)
\left(I_{E_i\otimes E_j^{\,n}} \otimes \widetilde T_l^{\,(m)}\right)
\;=\;
\left(I_{E_j^{\,n}\otimes E_i} \otimes \widetilde T_l^{\,(m)}\right)
\left(t_{ij}^{(n)} \otimes I_{E_l^{\,m}} \otimes U_{ij}^{\,n}\right).
\]
\end{lemma}

\begin{proof}
Using the elementary identity 
\((A\otimes B)(C\otimes D)=(AC)\otimes(BD)\) whenever the tensor products are well-defined, we compute:
\[
\begin{aligned}
\left(t_{ij}^{(n)} \otimes U_{ij}^{\,n}\right)
\left(I_{E_i\otimes E_j^{\,n}} \otimes \widetilde T_l^{\,(m)}\right)
&=\;
t_{ij}^{(n)} (I_{E_i\otimes E_j^{\,n}}) \otimes 
\left(U_{ij}^{\,n} \widetilde T_l^{\,(m)}\right)\\[2mm]
&=\;
 t_{ij}^{(n)} \otimes 
\left(\widetilde T_l^{\,(m)}\left(I_{E_l^{\,m}}\otimes U_{ij}^{\,n}\right)\right)
\qquad\text{(by Lemma~\ref{lemma:commutativity-Uijn-Tlm})}\\[2mm]
&=\;
\left(I_{E_j^{\,n}\otimes E_i} \otimes \widetilde T_l^{\,(m)}\right)
\left(t_{ij}^{(n)} \otimes I_{E_l^{\,m}} \otimes U_{ij}^{\,n}\right),
\end{aligned}
\]
which establishes the desired identity.
\end{proof}

The techniques used in the proof of the preceding lemma, together with repeated applications of Lemma~\ref{lemma:commutativity-Uijn-Tlm}, yield the following result.

\begin{lemma}\label{lemma:commutativity-tij-Uij-TA}
Let $A\subseteq I_k$ and $q\in\mathbb Z_+^{|A|}$. 
Then, for all $i,j\in I_k$ with $i\neq j$ and for each $n\in\mathbb Z_+$, we have
\[
\left(t_{ij}^{(n)} \otimes U_{ij}^{\,n}\right)
\left(I_{E_i \otimes E_j^{\,n}} \otimes \widetilde{T}_A^{(q)}\right)
\;=\;
\left(I_{E_j^{\,n} \otimes E_i} \otimes \widetilde{T}_A^{(q)}\right)
\left(t_{ij}^{(n)} \otimes I_{E_A^{q          }} \otimes U_{ij}^{\,n}\right).
\]
\end{lemma}
\begin{lemma}\label{lemma:commutativity-Ti-Tj-n}
For every $n\in\mathbb Z_+$ and $i,j\in I_k$ with $i\neq j$, we have
\[
\widetilde{T}_i\left(I_{E_i}\otimes \widetilde{T}_j^{(n)}\right)
\;=\;
\widetilde{T}_j^{(n)}\left(I_{E_j^{\,n}}\otimes \widetilde{T}_i\right)
\left(t_{ij}^{(n)}\otimes U_{ij}^{\,n}\right).
\]
\end{lemma}

\begin{proof}
We proceed by induction on $n$. For $n=1$, we have 
\(
\widetilde{T}_i(I_{E_i}\otimes\widetilde{T}_j)
= \widetilde{T}_j(I_{E_j}\otimes\widetilde{T}_i)
(t_{ij}\otimes U_{ij})
\) by definition.
Assume the statement holds for some $n\ge1$. Using 
$\widetilde{T}_j^{(n+1)}=\widetilde{T}_j^{(n)}\left(I_{E_j^{\,n}}\otimes\widetilde{T}_j\right)$, we have
\[
\begin{aligned}
\widetilde{T}_i\left(I_{E_i}\otimes \widetilde{T}_j^{(n+1)}\right)
&=\widetilde{T}_i\left(I_{E_i}\otimes \widetilde{T}_j^{(n)}\right)
   \left(I_{E_i\otimes E_j^{\,n}}\otimes \widetilde{T}_j\right)\\[2pt]
&=\widetilde{T}_j^{(n)}\left(I_{E_j^{\,n}}\otimes \widetilde{T}_i\right)
   \left(t_{ij}^{(n)}\otimes U_{ij}^{\,n}\right)
   \left(I_{E_i\otimes E_j^{\,n}}\otimes \widetilde{T}_j\right)
   \quad\text{(by induction hypothesis)}\\[2pt]
&=\widetilde{T}_j^{(n)}\left(I_{E_j^{\,n}}\otimes \widetilde{T}_i\right)
   \left(I_{E_j^{\,n}}\otimes I_{E_i}\otimes \widetilde{T}_j\right)
   \left(t_{ij}^{(n)}\otimes I_{E_j}\otimes U_{ij}^{\,n}\right)
   \quad\text{(by Lemma~\ref{lemma:commutativity-tij-Uij-Tl})}\\[2pt]
&= \widetilde T_j^{(n)} \, \left( I_{E_j^{\, n}} \otimes \widetilde T_j \right) 
   \left( I_{E_j^{\, n}} \otimes I_{E_j} \otimes \widetilde T_i \right) 
   \left( I_{E_j^{ \,n}} \otimes t_{ij} \otimes U_{ij} \right) 
   \left( t_{ij}^{(n)} \otimes I_{E_j} \otimes U_{ij}^n \right) \\[2pt]
&=\widetilde{T}_j^{(n+1)}\left(I_{E_j^{\,n+1}}\otimes \widetilde{T}_i\right)
   \left(t_{ij}^{(n+1)}\otimes U_{ij}^{\,n+1}\right)
   \quad\text{(by Lemma~\ref{lemma:recursive-def-tij-n})}.
\end{aligned}
\]
This completes the induction.
\end{proof}

Using the definition \ref{eq:Dm-theta-U-im}, we get the following relation:
\begin{lemma}\label{lemma:commutativity-Tij-TA}
Let \(A=\{i_1<i_2<\cdots<i_p\}\subseteq I_k\) and \(n=(n_1,\dots,n_p)\in\mathbb Z_+^{|A|}\).
For \(2\le j\le p\), 
\[
\widetilde{T}_{i_j}\left(I_{E_{i_j}}\otimes \widetilde{T}_A^{(n)}\right)
\;=\;
\widetilde{T}_A^{(n+e_j)}
\prod_{r=1}^{j-1}
D_{\,j-r}[W_{i_j,i_{\,j-r}}].
\]
Similarly, if $A^c=\{l_1 <l_2<\cdots <l_{k-p}\},$ and $m=(m_1, \cdots, m_{k-p})\in \Z_+^{|A^c|}.$ Then for $2 \leq j \leq k-p,$
\[
\widetilde{T}_{l_j}\left(I_{E_{l_j}}\otimes \widetilde{T}_{A^c}^{(m)}\right)
\;=\;
\widetilde{T}_{A^c}^{(m+e_j)}
\prod_{r=1}^{j-1}
D_{\,j-r}[W_{l_j,l_{\,j-r}}].
\]
\end{lemma}
\begin{proof}
We prove the statement by induction on 
$r\in\{1,2,\dots,j-1\}$, 
the number of factors $\widetilde T_{i_1}^{(n_1)}, \dots, \widetilde T_{i_r}^{(n_r)}$ 
that $\widetilde T_{i_j}$ must cross inside $\widetilde T_A^{(n)}$. We claim that if $\widetilde T_{i_j}$ has crossed the first $r$ factors, then
\[
\widetilde T_{i_j}\left(I_{E_{i_j}}\otimes \widetilde T_A^{(n)}\right)
=
\widetilde T_{1,r}\left(I_{E_{1,r}}\otimes \widetilde T_{i_j}\left( I_{E_{i_j}}\otimes \widetilde T_{r+1,p}\right)\right)\prod_{m=1}^{r} D_{\,r+1-m}[W_{i_j,i_{\,r+1-m}}]
\]
For $r=1$, Lemma~\ref{lemma:commutativity-Ti-Tj-n} implies that 
\begin{align*}
 \widetilde T_{i_j} \left( I_{E_{i_j}} \otimes \widetilde T_{A}^{(n)} \right) 
&= \widetilde T_{i_1}^{(n_1)}  \, \left( I_{E_{i_1}^{\,n_1}} \otimes \widetilde T_{i_j} \right) 
\left( t_{i_ji_1}^{(n_1)} \otimes U_{i_ji_1}^{\,n_1} \right)
\left( I_{E_{i_j}} \otimes I_{E_{i_1}^{\,n_1}} \otimes  \widetilde T_{2,p} \right) \\[4pt]
&= \widetilde T_{i_1}^{(n_1)}  \, \left( I_{E_{i_1}^{\,n_1}} \otimes \widetilde T_{i_j} \right) 
\left(  I_{E_{i_1}^{\,n_1}} \otimes I_{E_{i_j}} \otimes\widetilde T_{2,p} \right)
\left( t_{i_ji_1}^{(n_1)}  \otimes I_{E_{2,p}}  \otimes U_{i_ji_1}^{\,n_1} \right)
 \quad (\text{by Lemma } \ref{lemma:commutativity-tij-Uij-TA})\\[4pt]
&= \widetilde T_{i_1}^{(n_1)}  \, \left( I_{E_{i_1}^{\,n_1}} \otimes \widetilde T_{i_j} (   I_{E_{i_j}} \otimes\widetilde T_{2,p} )\right)
D_1[W_{i_j,i_1}] \quad ({\text{by }\eqref{eq:theta-ij-im},~\eqref{eq:Dm-theta-U-im}} )\\
\end{align*}
Assume the relation holds if $\widetilde T_{i_j}$ crosses the first $r$ factors.
We now commute $\widetilde T_{i_j}$ past the next block $\widetilde T_{i_{r+1}}^{(n_{r+1})}$.
By Lemma~\ref{lemma:commutativity-Ti-Tj-n}, and using \eqref{eq:notational-convenience}, we get
\[
\widetilde T_{1,r}\left(I_{E_{1,r}}\otimes \widetilde T_{i_j}\left(I_{E_{i_j}}\otimes \widetilde T_{i_{r+1}}^{(n_{r+1})}\right)\right)
=
\widetilde T_{1,{r+1}}\left(I_{E_{1,r+1}}\otimes \widetilde T_{i_j}\right)\left( I_{E_{1,r}}\otimes t^{(n_{r+1})}_{i_j i_{r+1}} \otimes  U^{n_{r+1}}_{i_j i_{r+1}}\right).
\]
Again Lemma \ref{lemma:commutativity-tij-Uij-TA} implies that
\[
\begin{aligned}
&\widetilde T_{1,r}\left(I_{E_{1,r}}\otimes \widetilde T_{i_j}\left( I_{E_{i_j}}\otimes \widetilde T_{r+1,p}\right)\right)\\
&=
\widetilde T_{1,{r+1}}\left(I_{E_{1,r+1}}\otimes \widetilde T_{i_j}\right)\left( I_{E_{1,r}}\otimes \left(t^{(n_{r+1})}_{i_j i_{r+1}} \otimes  U^{n_{r+1}}_{i_j i_{r+1}}\right)\left(I_{E_{i_j}}\otimes I_{E^{n_{r+1}}_{i_{r+1}}}\otimes \widetilde T_{r+2,p}\right)\right)\\
&=
\widetilde T_{1,{r+1}}\left(I_{E_{1,r+1}}\otimes  \widetilde T_{i_j} \left( I_{E_{i_j}}\otimes \widetilde T_{r+2,p}\right)\right)\left(I_{E_{1,r}}\otimes t^{(n_{r+1})}_{i_j i_{r+1}} \otimes I_{E_{r+2,p}} \otimes  U^{n_{r+1}}_{i_j i_{r+1}}\right)\\
&=
\widetilde T_{1,r+1}\left(I_{E_{1,r+1}}\otimes  \widetilde T_{i_j} \left( I_{E_{i_j}}\otimes \widetilde T_{r+2,p}\right)\right)\left(\theta^{(n_{r+1})}_{i_j, i_{r+1}} \otimes  U^{n_{r+1}}_{i_j i_{r+1}}\right)\\
&=
\widetilde T_{1,r+1}\left(I_{E_{1,r+1}}\otimes  \widetilde T_{i_j} \left( I_{E_{i_j}}\otimes \widetilde T_{r+2,p}\right)\right)D_{r+1}[W_{i_j, i_{r+1}}]\\
\end{aligned}
\]
Substituting this into the induction hypothesis establishes the formula for $r+1$.
\begin{align*}
   \widetilde T_{i_j}\left(I_{E_{i_j}}\otimes \widetilde T_A^{(n)}\right)
&=
\widetilde T_{1,r}\left(I_{E_{1,r}}\otimes \widetilde T_{i_j}\left( I_{E_{i_j}}\otimes \widetilde T_{r+1,p}\right)\right) D_{\,r+1}[W_{i_j,i_{\,r+1}}]\prod_{m=1}^{r} D_{\,r+1-m}[W_{i_j,i_{\,r+1-m}}] \\
&= \widetilde T_{1,r}\left(I_{E_{1,r}}\otimes \widetilde T_{i_j}\left( I_{E_{i_j}}\otimes \widetilde T_{r+1,p}\right)\right) \prod_{m=1}^{r+1} D_{\,r+2-m}\!\big[W_{i_j,i_{\,r+2-m}}\big], \\
\end{align*}
Therefore, the result holds for $r+1$ proving our claim.
In particular, when $r=j-1$, that is when $\widetilde T_{i_j}$ crosses first $j-1$ blocks, we get
\[
\widetilde{T}_{i_j}\left(I_{E_{i_j}}\otimes \widetilde{T}_A^{(n)}\right)
=
\widetilde{T}_A^{(n+e_j)}
\prod_{m=1}^{j-1}
D_{\,j-m}[W_{i_j,i_{\,j-m}}],
\]
which completes the proof. The second claim follows similarly.
\end{proof}

\begin{lemma} \label{lemma: commutativity of Tlj and TnA}
For any $l \in A^c,$ $ \widetilde{T}_{l}\left(I_{E_l} \otimes \widetilde{T}^{(n)}_A\right) 
 = \widetilde{T}^{(n)}_A \left(I_{E^n_A} \otimes \widetilde{T}_l\right)\displaystyle \prod_{r=1}^{p} D_{p+1-r}[W_{l,i_{p+1-r}}].$
\end{lemma}
\begin{proof}
    The proof follows from Lemma \ref{lemma:commutativity-tij-Uij-TA} along with Lemma \ref{lemma:commutativity-Ti-Tj-n}.
\end{proof}

\begin{lemma}\label{lem:sigma-reduces-coreA-kernels}
For each $A \subseteq I_k$, the following statements hold true.
\begin{enumerate}
    \item $\mathcal W_A$ and $\mathcal D_A$ reduce $\sigma(\mathcal A)$.
    \item For all $i \neq j,$ $U_{ij} \clw_A=\clw_A,$ and $U_{ij} \cld_A=\cld_A.$
    \item For all $j \in A^c$, $\widetilde T_j (E_j \otimes \mathcal W_A) \subseteq \mathcal W_A$, and $\widetilde T^*_j  \mathcal W_A \subseteq E_j \otimes \mathcal W_A$. Also, $\widetilde T_j (E_j \otimes \mathcal D_A) \subseteq \mathcal D_A$, and $\widetilde T^*_j  \mathcal D_A \subseteq E_j \otimes \mathcal D_A$.
    \end{enumerate}
\end{lemma}

\begin{proof}
\emph{Proof of (1):} For each $i \in I_k$, the covariance condition gives $\sigma(a)\,\widetilde T_i=\widetilde T_i\,(\varphi_i(a)\otimes I_{\mathcal H}), \, \forall a\in\mathcal A$.
Taking adjoints on both sides, 
\begin{equation}\label{{eq:adjoint-intertwine}}
    \widetilde T_i^{*}\sigma(a) = (\varphi_i(a)\otimes I_{\mathcal H})\widetilde T_i^{*}.
\end{equation}
For $h\in\mathcal W_A$, $\widetilde T_j^{*}h=0, \, \forall j \in A$. Therefore, for each $a \in \mathcal{A},\, \widetilde T_j^{*}\sigma(a)h=(\varphi_j(a)\otimes I)\widetilde T_j^{*}h=0.$
Hence $\sigma(a)\mathcal W_A\subseteq\mathcal W_A$, for all $a \in \mathcal{A}$, proving our claim.

Now we claim that $\cld_A$ reduces $\sigma(\mathcal{A})$. Let $B:=I_k\setminus A$ and $\varphi_B^{(m)}$ denote the left action of $\mathcal A$ on $E_B^{m}$.  
Then $\sigma(a)\widetilde T_B^{(m)} 
= \widetilde T_B^{(m)}\left(\varphi_B^{(m)}(a)\otimes I_{\clh}\right)$ implies that
\[
\left(\widetilde T_B^{(m)}\right)^{*}\sigma(a)
= \left(\varphi_B^{(m)}(a)\otimes I_{\clh}\right)\left(\widetilde T_B^{(m)}\right)^{*}.
\]
Let $x \in \cld_A$. Then $x=\widetilde T_B^{(m)}\left(\xi \otimes w\right)$ with $\xi\in E_B^{m}$ and $w\in\mathcal W_A$.
\[
\sigma(a)x= \sigma(a) \widetilde T_B^{(m)}\left(\xi \otimes w\right)=\widetilde T_B^{(m)}\left(\varphi_B^{(m)}(a) \xi\otimes w\right)
\in \widetilde T_B^{(m)}\left(I_{E_B^m}\otimes\mathcal W_A\right).
\]
Thus each $\widetilde T_B^{(m)}\left(I_{E_B^m}\otimes\mathcal W_A\right)$ is invariant under
$\sigma(a)$ and $\sigma(a)^*$. Therefore, $\mathcal D_A$ reduces $\sigma(\mathcal A)$.

\emph{Proof of (2):} Note that $\widetilde T_k (I_{E_k} \otimes U_{ij}) = U_{ij}\widetilde T_k$, and hence  
$ (I_{E_k} \otimes U_{ij}) \widetilde T^*_k =  \widetilde T^*_k U_{ij}$, for all $i\neq j$ and for all $k\in I_k$. 
And hence $\clw_A, \, \cld_A$ reduce $U_{ij}$.

\emph{Proof of (3):}  Suppose $h \in \clw_A$. Then, for every $x \in E_j, \, j \in A^c$, we have
\[
\widetilde T_i^{\,*} \widetilde T_j (x \otimes h)
= (I_{E_i} \otimes \widetilde T_j)(t_{ji} \otimes U_{ij})(I_{E_j} \otimes  \widetilde T_i^{\,*}) x \otimes h
= 0, \qquad \forall i \in A.
\]
Hence $\widetilde T_j( E_j \otimes \mathcal W_A)\subseteq \clw_A$. 
Also, the doubly twisted relations imply,  
$(I_{E_i} \otimes \widetilde T_j^{\,*} )\widetilde T_i^{\,*}
= (t_{ij} \otimes U_{ij})^* (I_{E_j} \otimes \widetilde T_i^{\,*})\widetilde T_j^{\,*}$, 
and therefore, 
$\widetilde T_j^{\,*}\clw_A \subseteq  E_j \otimes \mathcal W_A.$

Now we claim that $\cld_A$ reduces $T_j, \, j \in A^c.$
Suppose first that $A = \emptyset$. Then, by convention, $\clw_A = \clh$, and hence 
$\cld_A = \clh_A$ by Theorem~\ref{thm: representation of summands}. In this case, $\cld_A$ trivially reduces 
$\widetilde T_j$ for all $j \in I_k$.  
If $A = I_k$, then $\cld_A = \clw_{I_k}$, and the claim follows from the first part.  
Now assume $A = \{i_1, \ldots, i_p\}$ with $1 \le p < k$.
Let $A^{\mathrm c}=\{l_1,\ldots,l_{k-p}\}$ and fix an index $s\in\{1,\ldots,k-p\}$.
Set $j:=l_s$. 
Given $x\in E_j$ and $h\in\mathcal D_A$, we claim that
$\widetilde T_j(x\otimes h)\in\mathcal D_A$.
Recall that $\mathcal D_A
:=
\bigcap_{m\in\mathbb Z_+^{k-p}}
\widetilde T_{A^{\mathrm c}}^{(m)}(I_{E_{A^{\mathrm c}}^{\,m}}\otimes \mathcal W_A).$
Fix $m=(m_1,\ldots,m_{k-p})\in\mathbb Z_+^{k-p}$.
Since $h\in\mathcal D_A$, we have
$h=\widetilde T_{A^{\mathrm c}}^{(m)}(\xi\otimes w)$ for some
$\xi\in E_{A^{\mathrm c}}^{\,m}$ and $w\in\mathcal W_A$.
Using Lemma~\ref{lemma:commutativity-Tij-TA} (applied to the ordered set $A^{\mathrm c}$),
\[
\widetilde T_{l_s}\bigl(I_{E_{l_s}}\otimes \widetilde T_{A^{\mathrm c}}^{(m)}\bigr)
=
\widetilde T_{A^{\mathrm c}}^{(m+e_{l_s})}
\prod_{r=1}^{s-1} D_{\,s-r}\!\big[W_{l_s,l_{\,s-r}}\big].
\]
and therefore,
\[
\widetilde T_j(x\otimes h)
=
\widetilde T_{A^{\mathrm c}}^{(m+e_j)}
\left(\prod_{r=1}^{s-1} D_{\,s-r}\!\big[W_{l_s,l_{\,s-r}}\big]\right)
(x\otimes \xi\otimes w).
\]
Since $\mathcal W_A$ reduces each $U_{ij}$,
the operator
$\prod_{r=1}^{s-1} D_{\,s-r}[W_{l_s,l_{\,s-r}}]$
maps $E_{j}\otimes E_{A^{\mathrm c}}^{\,m}\otimes \mathcal W_A$
into $E_{A^{\mathrm c}}^{\,m+e_j}\otimes \mathcal W_A$.
Hence
\[
\widetilde T_j(x\otimes h)\in
\widetilde T_{A^{\mathrm c}}^{(m+e_j)}
\left(I_{E_{A^{\mathrm c}}^{m+e_j}}\otimes \mathcal W_A\right)
\subseteq
\widetilde T_{A^{\mathrm c}}^{(m)}
\left(I_{E_{A^{\mathrm c}}^{\,m}}\otimes \mathcal W_A\right),
\]
and since $m$ was arbitrary, we conclude that
$\widetilde T_j(E_j\otimes \mathcal D_A)\subseteq \mathcal D_A$.
A similar argument (using the adjoint relation and Lemma~\ref{lem:sigma-reduces-coreA-kernels}(2))
shows that $\widetilde T_j^{\,*}\mathcal D_A\subseteq E_j\otimes \mathcal D_A$.

\end{proof}

Now assume that for this $A\subseteq I_k$, the corresponding $\clh_A \neq \{0\}.$ By Theorem~\ref{thm: representation of summands}, $\clh_A
= \displaystyle\bigoplus_{n \in \mathbb{Z}_+^{|A|}}
\widetilde{T}_A^{(n)}
\left(I_{E_A^{\,n}} \otimes \cld_A\right).$
Define the map $\Pi_A: \clh_A \rightarrow \mathcal{F}(E_A) \otimes \cld_A $
by
\begin{equation}\label{eq:PiA-def}
\Pi_A\left(\widetilde {T}_A^{(n)}(\xi \otimes h) \right) 
\;=\; \xi\otimes h,
\end{equation}
The map $\Pi_A$ is clearly well-defined.  
For $\xi,\eta\in E_A^{n}$ and $h,g\in\cld_A$, the isometric property of each
$(\sigma,T_i)$ yields
\[
\big\langle 
\widetilde T_A^{(n)}(\xi\otimes h),
\, \widetilde T_A^{(n)}(\eta\otimes g)
\big\rangle
=\big\langle 
h,\,
\sigma\!\left(\langle \xi,\eta\rangle_{\mathcal A}\right)
g
\big\rangle
=\big\langle 
\xi\otimes h,\,\eta\otimes g
\big\rangle.
\]
Hence each $\widetilde T_A^{(n)}$ is an isometry.
Surjectivity is immediate from the definition. 
Thus $\Pi_A$ is unitary.

\begin{prop}\label{prop:construction of MAi}
Let $(\sigma, T_1, \ldots, T_k)$ be a doubly twisted isometric covariant representation of $\mathbb E$ on $\clh$, and let $A=\{i_1<\cdots<i_p\}\subseteq I_k$ with $\clh_A\neq0$. Then with the canonical unitary $\Pi_A: \clh_A \rightarrow \mathcal{F}(E_A) \otimes \cld_A $, we have
\begin{enumerate}
\item For $x_i\in E_i$, we have
\[
\Pi_A T_i(x_i)\Pi_A^* =
\begin{cases}
T_{x_{i_1}}\otimes I_{\cld_A},
& i=i_1,\\[8pt]
\displaystyle
\Bigg(\prod_{r=1}^{j-1}
D_{j-r}\!\big[W_{i_j,i_{j-r}}\big]\Bigg)
(T_{x_{i_j}}\otimes I_{\cld_A}),
& i=i_j,\; 2\le j\le p,\\[10pt]
\displaystyle
\left(I_{\mathcal F(E_A)}\otimes \widetilde T_i\right)
\left(\prod_{r=1}^{p}
D_{p+1-r}\!\left[W_{i,i_{p+1-r}}\right]\right)
(T_{x_i}\otimes I_{\cld_A}),
& i\in A^{\mathrm c}.
\end{cases}
\]

\item For $a\in\mathcal A$,
\[
\Pi_A\sigma(a)\Pi_A^*
=\sigma_A(a)\otimes I_{\cld_A},
\qquad
\sigma_A(a):=\varphi_{\infty}(a)=\bigoplus_{n\in\mathbb Z_+^{|A|}}\varphi_A^{(n)}(a).
\]

\item For $i\ne j$ in $I_k$,
\[
U_{A,ij}:=\Pi_AU_{ij}\Pi_A^*
=I_{\mathcal F(E_A)}\otimes U_{ij}.
\]
\end{enumerate}
In particular,
if we set $M_{A,i}:=\Pi_AT_i(\cdot)\Pi_A^*$, then $\left(\sigma_A,\{M_{A,i}\}_{i\in I_k}\right) $ is again a doubly twisted isometric covariant representation with twist
$\{\,t_{ij}\otimes U_{A,ij}\,\}_{i\ne j}$ which is unitarily equivalent to \(\restr{(\sigma, \{ T_i\}_{i \in I_k})}{\clh_A}\).
Further $(\sigma_A,M_{A,i})$ is induced for $i\in A$ and fully coisometric for $i\in A^{\mathrm c}$.
\end{prop}
\begin{proof}
Let $\xi \otimes h \in \mathcal{F}(E_A) \otimes_\sigma \cld_A,$ where $\xi = \xi_{1}\otimes \xi_{2}\otimes\cdots\otimes \xi_{p} \in  E^n_A,$ and $h \in \cld_A$. 
Also, let $A^{\mathrm c} = \{l_1, \ldots, l_{k-p}\}$.

\begin{enumerate}
\item 
For $x_{i_1} \in E_{i_1}$, we compute
\begin{align*}
\Pi_A T_{i_1}(x_{i_1}) \Pi_A^*(\xi \otimes h)
&= \Pi_A T_{i_1}(x_{i_1})
   \left(T_{i_1}^{(n_1)}(\xi_{1})
   T_{i_2}^{(n_2)}(\xi_{2}) 
   \cdots 
   T_{i_p}^{(n_p)}(\xi_{p}) h\right)\\
&= \Pi_A \widetilde T_{A}^{(n+e_{i_1})}(x_{i_1}\otimes \xi \otimes  h) \\
&= x_{i_1}\otimes \xi \otimes h\\
&= (T_{x_{i_1}}\otimes I_{\cld_A})(\xi\otimes h).
\end{align*}
Hence
\[
T_{i_1}(x_{i_1}) 
\;\cong\;
T_{x_{i_1}}\otimes I_{\cld_A}.
\]
For $2\le j\le p$ and $x_{i_j}\in E_{i_j}$, we use Lemma~\ref{lemma:commutativity-Tij-TA} to get
\begin{align*}
\Pi_A T_{i_j}(x_{i_j}) \Pi_A^*(\xi \otimes h) 
&= \Pi_A T_{i_j}(x_{i_j}) 
   T_{i_1}^{(n_1)}(\xi_{1})
   \cdots 
   T_{i_p}^{(n_p)}(\xi_{p}) h\\
&= \Pi_A \widetilde{T}_A^{(n+e_{i_j})}
   \prod_{r=1}^{j-1} D_{j-r}[W_{i_j,i_{j-r}}]
   (x_{i_j}\otimes\xi\otimes h)\\
&= \left(\prod_{r=1}^{j-1}
   D_{j-r}[W_{i_j,i_{j-r}}]\right)
   \left(T_{x_{i_j}}\otimes I_{\cld_A}\right)(\xi\otimes h).
\end{align*}
Thus
\[
T_{i_j}(x_{i_j})
\;\cong\;
\left(\prod_{r=1}^{j-1}
D_{j-r}[W_{i_j,i_{j-r}}]\right)
\left(T_{x_{i_j}}\otimes I_{\cld_A}\right).
\]
Now, for $i\in A^{\mathrm c}$ and $x_i\in E_i$, we have
\begin{align*}
\Pi_A T_i(x_i)\Pi_A^*(\xi\otimes h)
&= \Pi_A T_i(x_i) \widetilde T_A^{(n)}(\xi\otimes h)\\
&= \Pi_A \widetilde T_A^{(n)} 
   \left(I_{E_A^{n}}\otimes \widetilde T_i\right)
   \prod_{r=1}^{p} D_{p+1-r}[W_{i,i_{p+1-r}}]
   (x_i\otimes\xi\otimes h)\\
&= \left(I_{E_A^{n}}\otimes \widetilde T_i\right)
   \prod_{r=1}^{p}
   D_{p+1-r}[W_{i,i_{p+1-r}}]
   (T_{x_i}\otimes I_{\cld_A})(\xi\otimes h).
\end{align*}
In general,
\[
T_i(x_i)
\;\cong\;
\left(I_{\mathcal{F}(E_A)}\otimes \widetilde T_i\right)
\prod_{r=1}^{p}
D_{p+1-r}[W_{i,i_{p+1-r}}]
\left(T_{x_i}\otimes I_{\cld_A}\right),
\]
Lemma \ref{lem:sigma-reduces-coreA-kernels}(3) implies that the above map is well-defined.
\item 
For any $a \in \mathcal{A},$
\[
(\Pi_A\sigma(a)\Pi_A^*)(\xi\otimes h)
= \Pi_A\sigma(a)\widetilde T_A^{(n)}(\xi\otimes h)
= \Pi_A\widetilde T_A^{(n)}\left(\varphi_A^{(n)}(a)\xi\otimes h\right)
= \varphi_A^{(n)}(a)\xi\otimes h.
\]
Hence 
\[
\Pi_A\sigma(a)\Pi_A^*
= \sigma_A(a)\otimes I_{\cld_A},
\text{ where }
\sigma_A(a):=\varphi_{\infty}(a)
.
\]
\item 
For any $i \neq j,$ we know that $U_{ij}\widetilde T_A^{(n)}
= \widetilde T_A^{(n)}(I_{E_A^{n}}\otimes U_{ij})$. Therefore,
\[
(\Pi_AU_{ij}\Pi_A^*)(\xi\otimes h)
= \Pi_AU_{ij}\widetilde T_A^{(n)}(\xi\otimes h)
= \Pi_A\widetilde T_A^{(n)}(\xi\otimes U_{ij}h)
= \xi\otimes U_{ij}h
= \left(I_{E^n_A}\otimes U_{ij}\right)(\xi\otimes h).
\]
Therefore,
\[
U_{A,ij}
= \Pi_AU_{ij}\Pi_A^*
= I_{\mathcal F(E_A)}\otimes U_{ij}.
\]
Lemma \ref{lem:sigma-reduces-coreA-kernels}(2) implies that the above map is well-defined.
\end{enumerate}

Since $\Pi_A$ is unitary and $\restr{(\sigma,\{T_i\}_{i \in I_k})}{\clh_A}$ satisfies the doubly twisted relations,
therefore, $\left(\sigma_A,\{M_{A,i}\}\right)$ is also a doubly twisted isometric covariant representation. This also implies that, $(\sigma_A,M_{A,i})$ is induced if $i\in A$, and  is fully coisometric, if $i\in A^{\mathrm c}$.
\end{proof}

\section{Unitary Extensions} \label{Unitary Extension}
The goal of this section is to construct a unitary extensions for specific classes via a direct limit construction. First we recall a categorical construction of the direct limit of Hilbert spaces \cite{Talker}.  
Let \( \clh = \{\clh_i\}_{i \in I} \) be a family of Hilbert spaces indexed by a non-empty directed set \( I \).  
Assume that for all \( i, j \in I \) with \( i \le j \) there is an isometry 
\(\varphi_{ji} : \clh_i \to \clh_j\) such that whenever \(i \le j \le k\),
\[
\varphi_{kj} \circ \varphi_{ji} = \varphi_{ki}, \quad \text{and} \quad \varphi_{ii} = \mathrm{id}_{\clh_i}.
\]  
The direct limit \((\clh_\infty, \psi_i)\) is a Hilbert space with isometries \(\psi_i : \clh_i \to \clh_\infty\) which satisfy
\[
\psi_j \circ \varphi_{ji} = \psi_i,
\]
and has the following universal property: for every family of contractions \(\{\pi_i\}_{i \in I}\) from \(\{\clh_i\}_{i \in I}\) into a Hilbert space \(\clk\) such that $\pi_j \circ \varphi_{ji} = \pi_i,$
there exists a unique contraction \(\pi_\infty : \clh_\infty \to \clk\) such that
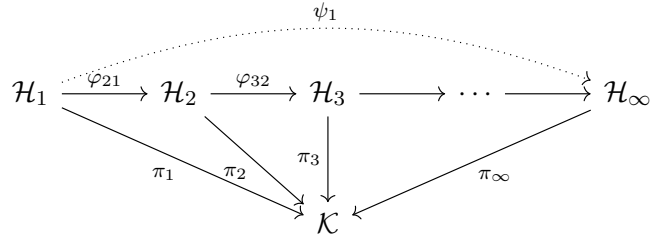
\[
\pi_i = \pi_\infty \circ \psi_i \quad \forall i \in I.
\]

\[
\begin{tikzcd}[column sep=3em, row sep=3em]
  \clh_1 \arrow[r, "\varphi_{21}"] \arrow[drr, "\pi_1"'] &
  \clh_2 \arrow[r, "\varphi_{32}"] \arrow[dr, "\pi_2"'] &
  \clh_3 \arrow[r] \arrow[d, "\pi_3"'] &
  \cdots \arrow[r] &
  \clh_{\infty} \arrow[dll, "\pi_{\infty}"] \\
  & & \clk & & 
  \arrow[from=1-1, to=1-5, bend left=20, dotted, "\psi_1" above]
\end{tikzcd}
\]

Now, given a direct system \((\clh_m, \varphi_{nm})\) with a direct limit \((\clh_\infty, \psi_m)\), 
let, for each $m \in I,$ \(\, 
(\sigma_m,\, T_{m,1}, \dots, T_{m,k})
\)
denote a doubly twisted isometric covariant representation of \(\mathbb{E}\) on \(\clh_m\), 
with corresponding twists \(\{\,t_{ij}\otimes U_{m,ij}\,\}_{i\ne j}\).
Suppose this tuple satisfies the following conditions:
For all \(m \leq n\) and \(i,\, j \in I_k,\, i\neq j\), and for each \(a \in \mathcal{A},\)
\begin{equation} \label{eq:intertwine-sigma-nm}
  \varphi_{nm}\, \widetilde{T}_{m,i} 
  = \widetilde{T}_{n,i}\, (I_{E_i}\otimes \varphi_{nm}), \quad
  \varphi_{nm}\, U_{m,ij} 
  = U_{n,ij}\, \varphi_{nm}, \quad
  \text{and} \quad
  \varphi_{nm}\, \sigma_m(a) = \sigma_n(a)\, \varphi_{nm}.
\end{equation}
Then Lemma~\ref{lemma: existence of Vinfnty} implies the existence of a contraction 
\(\sigma_{\infty}\), and for all \(i, j \in I_k,\) the existence of isometries 
\(\{\widetilde{T}_{\infty,i}\}_{i \in I_k}\), along with unitaries 
\(\{U_{\infty,ij}\}_{i\neq j}\), such that for each \(m\in I\) and \(a\in \mathcal{A},\)
\begin{equation} \label{eq:intertwine-sigma-infinity}
\psi_m \, \widetilde{T}_{m,i}
= \widetilde{T}_{\infty,i} (I_{E_i} \otimes \psi_m), \quad
\psi_m \, U_{m,ij}
= U_{\infty,ij}\,  \psi_m, \quad
\text{and} \quad
\psi_m\,\sigma_m(a)=\sigma_{\infty}(a)\,\psi_m.
\end{equation}
Therefore, for every \(x_i \in E_i\) and \(a \in \mathcal{A}\), the restriction to 
\(\psi_m(\mathcal{H}_m) \subseteq \mathcal{H}_\infty\) yields
\begin{equation} \label{eq:infinity-restricted-m}
  \restr{T_{\infty,i}(x_i)}{\clh_m} = T_{m,i}(x_i), \quad
  \restr{U_{\infty,ij}}{\clh_m} = U_{m,ij}, \quad
  \text{and} \quad
  \sigma_{\infty}(a)\big|_{\clh_m}= \sigma_m(a).
\end{equation}
Equation \ref{eq:intertwine-sigma-infinity} also implies that
\[
U^*_{\infty,ji}\, \psi_m= \psi_m U^*_{m,ji}= \psi_m U_{m,ij}= U_{\infty,ij} \psi_m.
\]
Similarly, for any $l \in I_k,$
\[
U_{\infty,ij}\, \widetilde{T}_{\infty,l}(I_{E_l} \otimes \psi_m)
=U_{\infty,ij}\,\psi_m\, \widetilde{T}_{m,l}
=\psi_m U_{m,ij}\, \widetilde{T}_{m,l}
=\psi_m \, \widetilde{T}_{m,l}(I_{E_l} \otimes U_{m,ij})
=\widetilde{T}_{\infty,l}(I_{E_l} \otimes U_{\infty,ij}).
\]
Since \(\psi_m(\mathcal{H}_m)\) is dense in \(\mathcal{H}_\infty\), these equalities extend to all of \(\mathcal{H}_\infty\). Therefore, we get
\begin{equation} \label{eq:unitary relations in defn}
     U_{\infty,ji}=U_{\infty,ij}^{*}, \quad \text{and} \quad
    U_{\infty,ij}\, \widetilde{T}_{\infty,l} 
    = \widetilde{T}_{\infty,l}\left(I_{E_l} \otimes U_{\infty,ij}\right).
\end{equation}

\begin{lemma}\label{lemma: sigma infinity is representation}
The map $\sigma_{\infty} : \mathcal{A} \to \mathcal{B}(\clh_{\infty})$ is a $C^*$-representation.
\end{lemma}

\begin{proof}
By the universal property of the direct limit,  
\(\sigma_{\infty} : \mathcal{A} \to \mathcal{B}(\clh_{\infty})\) is a contraction. 
Multiplicativity holds on the dense set $\displaystyle\bigcup_m\psi_m(\clh_m)$:
\[
\sigma_\infty(ab)\psi_m(h)
=\psi_m(\sigma_m(ab)h)
=\psi_m(\sigma_m(a)\sigma_m(b)h)
=\sigma_\infty(a)\sigma_\infty(b)\psi_m(h),
\]
and therefore, on all of $\clh_\infty$ by continuity.

For the $*$-property, fix $m,n\in I$ and choose $\ell\ge m,n$. Recall that $\varphi_{\ell n}\sigma_n(a)=\sigma_\ell(a)\varphi_{\ell n}$ and $\psi_\ell\varphi_{\ell n}=\psi_n$, $\psi_\ell\varphi_{\ell m}=\psi_m$. Then for $g\in\clh_m$, $h\in\clh_n$, along with the fact $\psi_m,\, \forall m \in I,$ is an isometry implies
\begin{align*}
\langle \sigma_\infty(a^*)\psi_n(h),\psi_m(g)\rangle
&=\langle \psi_n(\sigma_n(a^*)h),\psi_m(g)\rangle
=\langle \varphi_{\ell n}\sigma_n(a^*)h,\ \varphi_{\ell m}g\rangle_{\clh_\ell}\\
&=\langle \sigma_\ell(a^*)\varphi_{\ell n}h,\ \varphi_{\ell m}g\rangle
=\langle \varphi_{\ell n}h,\ \sigma_\ell(a)\varphi_{\ell m}g\rangle\\
&=\langle \psi_n(h),\ \psi_m(\sigma_m(a)g)\rangle
=\langle \psi_n(h),\ \sigma_\infty(a)\psi_m(g)\rangle,
\end{align*}
Hence $\sigma_\infty(a^*)=\sigma_\infty(a)^*$.

If each $\sigma_m$ is nondegenerate, then for every $m$,
\(\overline{\sigma_\infty(\mathcal A)\psi_m(\clh_m)}=\psi_m(\clh_m)\); density of $\displaystyle\bigcup_m\psi_m(\clh_m)$ in $\clh_\infty$ gives nondegeneracy of $\sigma_\infty$.
Therefore, $\sigma_\infty$ is a $*$-representation.
\end{proof}

Now we verify that the limiting representation 
\(\left(\sigma_\infty, \{T_{\infty,i}\}_{i\in I_k}\right)\)
preserves the covariance relations satisfied by each 
\((\sigma_m, \{T_{m,i}\}_{i\in I_k})\).

\begin{lemma}\label{lemma:Covariance-relations}
For all \(a,b \in \mathcal{A}\), \(i \in I_k\), and \(x_i \in E_i\), we have
\[
T_{\infty,i}(a x_i b)
\;=\;
\sigma_\infty(a)\, T_{\infty,i}(x_i)\, \sigma_\infty(b).
\]
\end{lemma}

\begin{proof}
For any \(m \in I\) and \(h \in \clh_m\), using covariance of \((\sigma_m, T_{m,i})\) along with Equation
\eqref{eq:intertwine-sigma-infinity},
\[
T_{\infty,i}(a x_i b)\,\psi_m(h)
= \psi_m\, T_{m,i}(a x_i b)\,h
= \psi_m\,\sigma_m(a)\, T_{m,i}(x_i)\, \sigma_m(b)
=\sigma_\infty(a)\, T_{\infty,i}(x_i)\, \sigma_\infty(b)\, \psi_m(h).
\]
Since the subspaces \(\psi_m(\clh_m)\) are dense in \(\clh_\infty\),
the equality holds on all of \(\clh_\infty\).
\end{proof}
Now we prove that the twisted commutation relations are preserved by the limiting representation 
\((\sigma_\infty, \{T_{\infty,i}\}_{i\in I_k})\).

\begin{lemma}\label{lemma:twisted-relations-infty}
For every \(i\neq j\), the following twisted relation holds: 
$$\widetilde T_{\infty,i}\left(I_{E_i}\otimes \widetilde T_{\infty,j}\right)
=
\widetilde T_{\infty,j}\left(I_{E_j}\otimes \widetilde T_{\infty,i}\right)
\,(t_{ij}\otimes U_{\infty,ij}).$$
\end{lemma}

\begin{proof}
Fix \(i\neq j\), and let \(m\in I\), \(\xi\in E_i\), \(\eta\in E_j\), \(g\in\clh_m\).
Then, by the intertwining relations \eqref{eq:intertwine-sigma-infinity},
\begin{align*}
\widetilde T_{\infty,i}\left(I_{E_i}\otimes \widetilde T_{\infty,j}\right)
(\xi\otimes \eta\otimes \psi_m g)
&= \psi_m\,\widetilde T_{m,i}\left(I_{E_i}\otimes \widetilde T_{m,j}\right)
(\xi\otimes \eta\otimes g)  \\
&= \psi_m\,\widetilde T_{m,j}\left(I_{E_j}\otimes \widetilde T_{m,i}\right)
\,(t_{ij}\otimes U_{m,ij})(\xi\otimes \eta\otimes g)\\
&= \widetilde T_{\infty,j}\left(I_{E_j}\otimes \widetilde T_{\infty,i}\right)
\,(t_{ij}\otimes U_{\infty,ij})(\xi\otimes \eta\otimes \psi_m g).
\end{align*}
By boundedness and density of \(\bigcup_m\psi_m(\clh_m)\), 
the identity extends to all of \(E_i\otimes E_j\otimes \clh_\infty\).
\end{proof}

\begin{lemma}\label{lemma: condition for fully coisometric}
Fix $j\in I_k$.  
Suppose that for every $m\in I$, there exists $n>m$ such that
\[
\widetilde T_{n,j}\,\widetilde T_{n,j}^{*}\,\varphi_{n,m}
\;=\;
\varphi_{n,m}.
\]
Then the direct limit representation $(\sigma_{\infty},\,T_{\infty,j})$ is fully coisometric.  
Moreover, if 
\[
\widetilde T_{n,j}\,\widetilde T_{n,j}^{*}
= I_{\clh_n}
\qquad \forall\,n\in I,
\]
then $(\sigma_{\infty}, T_{\infty,j})$ is fully coisometric.
\end{lemma}

\begin{proof}
We already know that $T_{\infty,j}: E_j \to \mathcal B(\clh_\infty)$ is an isometric representation, and 
$\operatorname{Ran}\,\widetilde T_{\infty,j}$ is a closed subspace of $\clh_\infty$.
Hence $\widetilde T_{\infty,j}$ is a coisometry if and only if 
$\operatorname{Ran}\,\widetilde T_{\infty,j} = \clh_\infty$. Because of  
the density of \(\bigcup_m\psi_m(\clh_m)\), it suffices to show that for each $m \in I,$
\(
\psi_m(\clh_m)\subseteq \operatorname{Ran}\,\widetilde T_{\infty,j}.
\)
Fix $m\in I$.  
By assumption, there exists $n>m$ such that
\(
\widetilde T_{n,j}\,\widetilde T_{n,j}^{*}\,\varphi_{n,m}
= \varphi_{n,m}.
\)
Therefore,
\(
\varphi_{n,m}(\clh_m)
\ \subseteq\
\operatorname{Ran}\,\widetilde T_{n,j}.
\)
Using $\psi_m=\psi_n\varphi_{n,m}$ together with
$\psi_n\,\widetilde T_{n,j}
=\widetilde T_{\infty,j}(I_{E_j}\otimes \psi_n)$, we get
\[
\psi_m(\clh_m)
=\psi_n\varphi_{n,m}(\clh_m)
\ \subseteq\
\psi_n\left(\operatorname{Ran}\,\widetilde T_{n,j}\right)
\ \subseteq\
\operatorname{Ran}\,\widetilde T_{\infty,j}.
\]
Hence $\psi_m(\clh_m)\subseteq \operatorname{Ran}\,\widetilde T_{\infty,j}$ for all $m\in I$, 
and therefore, 
$\operatorname{Ran}\,\widetilde T_{\infty,j}=\clh_\infty$.  
Equivalently,
\(
\widetilde T_{\infty,j}\,\widetilde T_{\infty,j}^{*}
=I_{\clh_\infty},
\)
that is, $(\sigma_{\infty}, T_{\infty,j})$ is fully coisometric.

Finally, if 
\(
\widetilde T_{n,j}\,\widetilde T_{n,j}^{*}=I_{\clh_n},
\, \forall\,n\in I,
\)
then each $\widetilde T_{n,j}^{*}$ is an isometry.  
By Lemma~\ref{lemma: existence of Vinfnty}, it follows that $\widetilde T_{\infty,j}^{*}$ is an isometry as well, 
and hence $(\sigma_{\infty}, T_{\infty,j})$ is fully coisometric.
\end{proof}
\begin{remark}
The hypothesis for all $m\in I,$ there exists $n>m$ such that $ 
\widetilde T_{n,j}\,\widetilde T_{n,j}^{*}\,\varphi_{n,m}
=\varphi_{n,m}$
is the natural (and minimal) condition ensuring that
$\psi_m(\clh_m)\subseteq \operatorname{Ran}\,\widetilde T_{\infty,j}$ for all $m\in I$, 
and hence that $\widetilde T_{\infty,j}$ is coisometric.  

The stronger assumption $\widetilde T_{n,j}\,\widetilde T_{n,j}^{*}=I_{\clh_n},$
for some $n$, implies the above relation for every $m\le n$, 
but does not, in general, guarantee coverage of $\psi_p(\clh_p)$ for all $p\ge n$.  
If one assumes instead that for every $m$ there exists $n\ge m$ such that $\widetilde T_{n,j}\,\widetilde T_{n,j}^{*}=I_{\clh_n},$
(that is, \emph{eventual coisometry}), then the required hypothesis holds automatically.
\end{remark}
\begin{lemma} \label{lemma: coiso plus twisted implies doubly twisted}
Let \((\sigma, T_1,\dots, T_k)\) be a c.c. covariant twisted isometric 
representation of \(\mathbb E\) on \(\mathcal H\) with corresponding twists $\{t_{ij} \otimes U_{ij}\}_{i \neq j},$ which is fully coisometric.
Then the representation is doubly twisted.
\end{lemma}
\begin{proof}
For each $i,$ we have
$\widetilde{T}_i\,\widetilde{T}_i^{\,*} = I_{\mathcal H}$. 
Also, for all $i\neq j$ in $I_k$, we have $ \widetilde T_i (I_{E_i}\otimes \widetilde T_j)
    = \widetilde T_j (I_{E_j}\otimes \widetilde T_i)\,(t_{ij}\otimes U_{ij}).$  This implies $(I_{E_i}\otimes \widetilde T_j) (t_{ij}\otimes U_{ij})^*
=  \widetilde T^*_i \widetilde T_j (I_{E_j}\otimes \widetilde T_i).$ Therefore, $$\widetilde T^*_i \widetilde T_j =(I_{E_i}\otimes \widetilde T_j) (t_{ji}\otimes U_{ji}) (I_{E_j}\otimes \widetilde T^*_i),$$
proving our claim.
\end{proof}
Combining \eqref{eq:unitary relations in defn} with 
Lemmas~\ref{lemma: sigma infinity is representation}-
\ref{lemma: condition for fully coisometric}, 
we obtain the following result:

\begin{theorem}\label{thm:inductive-limit}
Let $\{(\clh_m,\varphi_{n,m})\}_{m\in I}$ be a direct system of Hilbert spaces with direct limit
$(\clh_\infty,\psi_m)$.  
For each $m\in I$, let $(\sigma_m,\,T_{m,1},\dots,T_{m,k})$
be a doubly twisted isometric covariant representation of the product system 
$\mathbb E=\{E_i\}_{i=1}^k$ on $\clh_m$, with twists 
$\{t_{ij}\otimes U_{m,ij}\}_{i\neq j}$.  
Assume that for all $m\le n$, $a\in\mathcal A$, and all $i\neq j$,
\begin{equation*}
\varphi_{n,m}\widetilde T_{m,i}
=\widetilde T_{n,i}(I_{E_i}\otimes \varphi_{n,m}),\quad
\varphi_{n,m}U_{m,ij}=U_{n,ij}\varphi_{n,m},\quad
\varphi_{n,m}\sigma_m(a)=\sigma_n(a)\varphi_{n,m}.
\end{equation*}
Then there exists a twisted isometric covariant representation $(\sigma_\infty,\, T_{\infty,1}, \dots, T_{\infty,k})$
of the product system $\mathbb E$ on $\clh_\infty$, together with unitaries 
$\{U_{\infty,ij}\}_{i\neq j}$, such that 
$\{t_{ij} \otimes U_{\infty,ij}\}_{i\neq j}$ are the corresponding twists satisfying:
\begin{enumerate}
    \item For every $m\in I$, $a\in\mathcal A$, and $i\neq j$,
    \[
    \psi_m\,\widetilde T_{m,i}=\widetilde T_{\infty,i}(I_{E_i}\otimes \psi_m),\qquad
    \psi_m\,U_{m,ij}=U_{\infty,ij}\,\psi_m, \qquad
    \psi_m\,\sigma_m(a)=\sigma_\infty(a)\,\psi_m.
    \]
    Equivalently, for each $m\in I$,
    \[
    \restr{T_{\infty,i}(x_i)}{\clh_m}=T_{m,i}(x_i),\quad
    \restr{U_{\infty,ij}}{\clh_m}=U_{m,ij},\quad
    \restr{\sigma_\infty(a)}{\clh_m}=\sigma_m(a).
    \]
    
    \item
    For some $j\in I_k$, and for every $m\in I$, if there exists $n>m$ such that
    \[
    \widetilde T_{n,j}\,\widetilde T_{n,j}^{*}\,\varphi_{n,m}=\varphi_{n,m},
    \]
    then $(\sigma_\infty,T_{\infty,j})$ is fully coisometric. 
    Moreover, if 
    \[
    \widetilde T_{n,j}\,\widetilde T_{n,j}^{*}=I_{\clh_n}\qquad \text{for all }\,n\in I,
    \]
    then $\widetilde T_{\infty,j}$ is a coisometry.  
    In particular, if either of these conditions hold for every $j\in I_k$, 
    the representation $(\sigma_\infty,T_{\infty,1},\dots,T_{\infty,k})$ is fully coisometric, and hence doubly twisted.
\end{enumerate}
\end{theorem}

As an immediate consequence of the above theorem, we obtain the following elegant result for an $ n$-tuple of doubly twisted isometries \cite{RSS}. Recall that an $n$-tuple of isometries $(V_1, \ldots ,V_n)$ on $\clh$ is called doubly twisted with respect to a commuting family of unitaries $\{U_{ij}\}_{i<j}$ with $U_{ji} := U_{ij}^*$, if
$$V_iV_j = U_{ij} V_jV_i,\qquad V_i^*V_j=U_{ij}^*V_jV_i^*, \quad V_kU_{ij}=U_{ij}V_k \quad (i,j,k=I_n, \, i \neq j).$$ 

\begin{theorem}\label{thm:UE-doubly-twisted-central}
Let $V=(V_1,\ldots,V_k)$ be a doubly twisted tuple of isometries on a Hilbert space $\clh$ with corresponding twists $\{U_{ij}\}_{i\ne j}\subset \clb(\clh)$. Then $V$ admits a doubly twisted unitary extension.
\end{theorem}
\begin{proof}
Take $\mathcal A=\C$ and $E_i=\C$ with $t_{ij}=I$. Our aim is to construct a direct system ${(\clh_m,\varphi_{n,m})}_{m\in\Z+}$ of Hilbert spaces with direct limit $(\clh_\infty,\psi_m)$.
For each $m\in\mathbb Z_+$, let $\clh_m= \clh,$ and let $\sigma_m : \mathbb{C} \rightarrow \mathcal{B}(\clh)$ be given by
\[
\sigma_m(\lambda) = \lambda I_{\clh}, \quad \forall \lambda \in \mathbb{C}. 
\]
Also let $U_{m,ij} := U_{ij}.$ For $i \in I_k,$ let $Z_i := \displaystyle\prod_{j \neq i} U_{ji},$ and define
$T_{m,i} : E_i \to \mathcal B(\clh)$ by
\[
T_{m,i}(1) := M_{m,i} := Z_i^{\,m} V_i,
\qquad \text{so that} \qquad 
\widetilde T_{m,i}(\lambda \otimes h)
= \lambda\, Z_i^{\,m} V_i h,
\quad (\lambda\in\mathbb C,\; h\in\clh).
\]
Since each $V_i$ is an isometry and $Z_i$ is unitary, every 
$\widetilde T_{m,i}$ is an isometry.  
We now check that for each $m\in\mathbb Z_+$, the tuple $(\sigma_m,\,T_{m,1},\dots,T_{m,k})$
forms a doubly twisted isometric covariant representation of $\mathbb E$.
Let $a,b,\xi_i \in \mathbb C$, then
\[
T_{m,i}(a \xi_i b)
= a \xi_i b\, Z_i^{\,m} V_i
= \sigma_m(a)\, \xi_i\, Z_i^{\,m} V_i\, \sigma_m(b)
= \sigma_m(a)\, T_{m,i}(\xi_i)\, \sigma_m(b).
\]
We check the twisted relations for each $m\in\mathbb Z_+$: First notice that, since $\{U_{ij}\}_{i \neq j}$ is a commuting family, $Z_iZ_j=Z_jZ_i, \, \forall i, j \in I_k.$ Also $U_{ij}\in\{V_\ell\}',\, \forall i \neq j, \,\ell\in I_k.$ Therefore, for $\lambda_1,\,  \lambda_2\in \mathbb{C}$ and $ h \in \clh$,  
\begin{align*}
\widetilde T_{m,i}(I\otimes \widetilde T_{m,j})(\lambda_1\otimes \lambda_2\otimes h)
&= \lambda_1\, Z_i^{\,m}V_i\left(\lambda_2\, Z_j^{\,m}V_j h\right) 
= \lambda_1\lambda_2\, Z_i^{\,m} Z_j^{\,m} V_i V_j h 
= \lambda_1\lambda_2\, U_{ij}\, Z_j^{\,m} Z_i^{\,m} V_j V_i h \\
&= \widetilde T_{m,j}(I\otimes \widetilde T_{m,i})
      (I\otimes U_{ij})(\lambda_1\otimes \lambda_2\otimes h).\\
\end{align*}
For the adjoint relation, notice that for each $j \in I_k$,
\begin{align*}
  \big\langle \widetilde T_{m,j}^{*} h,\; \lambda \otimes h_1 \big\rangle
  &= \big\langle h,\; \widetilde T_{m,j}(\lambda \otimes h_1) \big\rangle 
  = \big\langle h,\; \lambda Z_j^{m} V_j h_1 \big\rangle 
  = \lambda \big\langle h,\; Z_j^{m} V_j h_1 \big\rangle 
  = \lambda \big\langle V_j^{*} Z_j^{*m} h,\; h_1 \big\rangle \\
  &= \big\langle 1 \otimes  V_j^{*} Z_j^{*m} h,\; \lambda \otimes h_1 \big\rangle.
\end{align*}
Hence
\(
\widetilde T_{m,j}^{*} h = 1 \otimes V_j^{*} Z_j^{*m} h.
\)
Therefore,
\begin{align*}
    \widetilde T_{m,j}^{*}\,\widetilde T_{m,i}(\lambda \otimes h) 
&= \lambda  \widetilde T_{m,j}^{*}  Z_i^m V_i h\\
&= \lambda(1  \otimes Z_j^{*m} V_j^{*} Z_i^m V_i h)\\
&=\lambda (1  \otimes  Z_i^m V_j^{*} V_i Z_j^{*m}h)\\
&=\lambda \otimes 1  \otimes  Z_i^m V_i U^*_{ji} V_j^{*}  Z_j^{*m}h\\
&= ( I_{E_j} \otimes  Z_i^m V_i U_{ij} )(  I_{E_i \otimes E_j} \otimes V_j^{*}  Z_j^{*m})(\lambda \otimes 1 \otimes h)\\
&=(I_{E_j} \otimes  Z_i^m V_i  )(t_{ij} \otimes U_{ij})(I_{E_i}\otimes \widetilde T_{m,j}^{*} ) (\lambda \otimes h)\\
&=(I_{E_j} \otimes  \widetilde T_{m,i}  )(t_{ij} \otimes U_{ij})(I_{E_i}\otimes \widetilde T_{m,j}^{*} ) (\lambda \otimes h).\\
\end{align*}
Therefore, $\widetilde T_{m,j}^{*}\,\widetilde T_{m,i}
=
(I_{E_j}\otimes \widetilde T_{m,i})
(I\otimes U_{ij})
(I_{E_i}\otimes \widetilde T_{m,j}^{*}).$
Also, one can easily show that $U_{m,ij} \widetilde T_{m,i} = T_{m,i}(I_{E_i} \otimes U_{m,ij}).$

We now construct the connecting maps $\varphi_{nm}, \, n\geq m$. Set $\Phi:=V_1V_2\cdots V_k$ and define the stationary connecting maps:
\[
\varphi_{m+1,m}\ :=\ \Phi \implies \varphi_{n,m}\ :=\ \Phi^{\,n-m}\quad(n\ge m).
\]
We claim that, for all $m\le n$, $a\in\mathcal A$, and all $i\neq j$, the following conditions hold:
\begin{equation*}
\varphi_{n,m}\widetilde T_{m,i}
=\widetilde T_{n,i}(I_{E_i}\otimes \varphi_{n,m}),\quad 
\varphi_{n,m}U_{m,ij}=U_{n,ij}\varphi_{n,m},\quad
\varphi_{n,m}\sigma_m(a)=\sigma_n(a)\varphi_{n,m}.
\end{equation*}
Since $U_{m,ij}=U_{ij}$ for all $m$, the compatibility condition $\varphi_{n,m} U_{m,ij} = U_{n,ij}\, \varphi_{n,m}$
holds trivially.  Likewise,
\(
\varphi_{n,m}\, \sigma_m(a) = \sigma_n(a)\, \varphi_{n,m}
\)
is automatic from the definition of $\sigma_m$.
Next, notice that for each $i, \, Z_i \Phi = \Phi Z_i.$ Therefore, for $\lambda\otimes h \in E_i\otimes \clh$,
\begin{align*}
\varphi_{n,m}\, \widetilde T_{m,i}(\lambda\otimes h)
&= \Phi^{\,n-m} \left(\lambda\, Z_i^{\,m} V_i h\right)  
= \lambda\, Z_i^{\,m} (V_1 \cdots V_k)^{\,n-m} V_i h 
= \lambda\, Z_i^{\,m} Z_i^{\,n-m} V_i \Phi^{\,n-m} h\\
&= \widetilde T_{n,i}\left(\lambda\otimes \varphi_{n,m}h\right) 
= \widetilde T_{n,i}(I_{E_i}\otimes \varphi_{n,m})(\lambda\otimes h).
\end{align*}
Therefore, all hypotheses of Theorem~\ref{thm:inductive-limit} are satisfied.  
Hence, we obtain a Hilbert space $\clh_\infty$ and a twisted isometric 
covariant representation $(\sigma_\infty,\, T_{\infty,1},\dots,T_{\infty,k})
\text{ on }\clh_\infty,$
with corresponding twists $\{t_{ij} \otimes U_{\infty,ij}\}_{i\neq j}$, such that
\[
\restr{T_{\infty,i}(\lambda)}{\clh}
  = T_{m,i}(\lambda),\qquad
\restr{U_{\infty,ij}}{\clh}=U_{ij},\qquad
\restr{\sigma_\infty(a)}{\clh}=\sigma(a).
\]
In particular, letting $\lambda =1$ and $m=0$, we get
\[
\restr{T_{\infty, i}(1)}{\clh} = V_i,\qquad
\restr{U_{\infty, ij}}{\clh} = U_{ij},\qquad
\restr{\sigma_{(\infty)}(a)}{\clh}= \sigma(a).
\]
Also the adjoint formula gives,
\[
\widetilde T_{n,j} \widetilde T_{n,j}^*
  = Z_j^{\,n} V_j V_j^* Z_j^{*n}
  = V_j V_j^*,
  \qquad \text{for all } n.
\]
So for each $m \in I$ and $j \in I_k$, the doubly twisted relations imply $\widetilde T_{m+1,j} \widetilde T_{m+1,j}^* \circ \varphi_{m+1,m}
  = \varphi_{m+1,m}.$
This implies that for each $j \in I_k$, the pair $(\sigma, \widetilde T_{\infty, j})$ is fully coisometric.  
Therefore, the representation $(\sigma, T_{\infty, 1},\dots,T_{\infty, k})$
is fully coisometric. Consequently, $\left( T_{\infty, 1}(1),\dots,T_{\infty, k}(1) \right)$
is a twisted unitary extension of $(V_1,\dots,V_k).$ Therefore, Lemma \ref{lemma: coiso plus twisted implies doubly twisted} implies it is a doubly twisted unitary extension.
\end{proof}

\begin{remark}
 If there exists a subset $A \subseteq I_n$ such that $V_i$ is a shift 
for each $i \in A$ (and unitary otherwise), then it suffices to take  
\[
\Phi := \prod_{i \in A} V_i,
\]
in the above proof.
\end{remark}
When $U_{ij}=z_{ij}$ for all $i\ne j$, the above theorem recovers the 
unitary extension for the class of doubly non-commuting isometries 
studied by Jeu and Pinto \cite{JP}.

\begin{cor}
Let $V=(V_1,\dots,V_n)$ be a $n$–tuple of isometries on $\clh$ satisfying
\[
V_j V_i \;=\; z_{ji}\, V_i V_j, 
\qquad 
V_i^* V_j \;=\; \overline{z_{ij}}\, V_j V_i^* 
\qquad (i\neq j),
\]
for scalars $z_{ij}\in\mathbb T$ with $z_{ij}\,z_{ji}=1$ and $z_{ii}=1$. Then $V$ admits a unitary extension.
\end{cor}
Further, when $U_{ij}=I$ for all $i\ne j$, the theorem yields a unitary 
extension for doubly commuting $k$-tuples of isometries, as studied by 
Sarkar \cite{S} and Słociński \cite{Slocinski80}.
\begin{cor}
If $V=(V_1,\dots,V_k)$ is a doubly commuting $k$–tuple of isometries on $\clh$, that is,
\[
V_iV_j=V_jV_i \quad\text{and}\quad V_i^*V_j=V_jV_i^* \qquad (i\neq j),
\]
then $V$ admits a unitary extension.
\end{cor}
The following example gives an immediate application of the above corollary.
\begin{example}
Multiplication operators $(M_{z_1},\cdots, M_{z_n})$ on $H^2(\mathbb D^n)$ admit a unitary extension.
\end{example}

Let $(\sigma, T_1,\dots, T_k)$ be a doubly twisted isometric covariant representation of a product system $\mathbb E$ of $C^*$\nobreakdash-correspondences over $\mathbb Z_+^k$ on a Hilbert space $\mathcal H$. 
By Theorem~\ref{thm: decompo for doubly twisted rep}, we can write $\mathcal H=\displaystyle\bigoplus_{A\subseteq I_k}\mathcal H_A.$ 
Fix $A=\{i_1,\dots,i_p\}\subseteq I_k$ with $\mathcal H_A\neq\{0\}$. 
Proposition~\ref{prop:construction of MAi} guarantees the existence of a unitary $\Pi_A:\ \mathcal H_A \xrightarrow{\ \cong\ } \mathcal F(E_A)\otimes_\sigma \mathcal D_A$
and a tuple $\left(\sigma_A, M_{A,1},\dots, M_{A,k}\right)$ on $\mathcal F(E_A)\otimes_\sigma \mathcal D_A$ that is unitarily equivalent to the restricted representation 
$\left(\restr{\sigma}{\mathcal H_A},\,\restr{T_1}{\mathcal H_A},\dots, \restr{T_k}{\mathcal H_A}\right)$ via $\Pi_A$. 

For each $i \in I_k$, let $M_i= \displaystyle\bigoplus_{A \subseteq I_k} M_{A,i}$, and $\sigma' = \displaystyle\bigoplus_{A \subseteq I_k}  \sigma_A$.  
For each $A\subseteq I_k$, if $\left(\sigma_A, M_{A,1},\dots, M_{A,k}\right)$ on $\mathcal F(E_A)\otimes_\sigma \mathcal D_A$ admits a doubly twisted unitary extension, then so does  $\left(\sigma', M_{1},\dots, M_{k}\right)$ on $\displaystyle\bigoplus_{A \subseteq I_k}\mathcal F(E_A)\otimes_\sigma \mathcal D_A$. Therefore, we get a doubly twisted unitary extension for $\left(\sigma,\,T_1,\dots, T_k\right).$ Consequently, it suffices to establish an analogue of Theorem~\ref{thm:inductive-limit}
for $A\subseteq I_k$. Since we are fixing a subset \(A\subseteq I_k\) throughout this construction, 
we shall henceforth suppress the subscript ‘‘\(A\)’’.  
Thus we write 
\((\sigma_m, M_{m,1},\dots,M_{m,k})\)
for the representation on \(\mathcal H_m\), with twists 
\(t_{ij}\otimes U_{m,ij}\).

\begin{cor}\label{cor:direct-limit-A}
Let $\{(\mathcal H_m,\varphi_{n,m})\}_{m\le n\in I}$
be a direct system of Hilbert spaces, with direct limit $(\mathcal K_A,\psi_m)$.  
For each $m\in I$, let $(\sigma_{m},\, M_{m,1},\dots, M_{m,k})$
be a doubly twisted isometric covariant representation of a product system $\mathbb E=\{E_i\}_{i\in I_k}$ on 
$\mathcal H_m$, with twists $\{\,t_{ij}\otimes U_{m,ij}\,\}_{i\ne j}.$
Let $\mathcal K_A=\mathcal F(E_A)\otimes_\sigma \mathcal D_A$, and the initial 
representation be given by
\(
(\sigma_{1},\, M_{1,1},\dots, M_{1,k})
  = (\sigma_A,\, M_{A,1},\dots, M_{A,k}).
\)
Assume that for all $m\le n$, all $a\in\mathcal A$, and all $i\ne j$, 
\[
\varphi_{n,m}\,\widetilde M_{m,i}
   = \widetilde M_{n,i}\,(I_{E_i}\otimes \varphi_{n,m}),\qquad
\varphi_{n,m}\, U_{m,ij}
   = U_{n,ij}\,\varphi_{n,m},\qquad
\varphi_{n,m}\,\sigma_{m}(a)
   = \sigma_{n}(a)\,\varphi_{n,m}.
\]
Then there exists a twisted isometric covariant representation $(\tau_A,\; V_{A,1},\dots,V_{A,k})$
of $\mathbb E$ on $\mathcal K_A$, with twists $\{\,t_{ij}\otimes W_{A,ij}\,\}_{i\ne j}$,
such that
\[
\restr{\widetilde V_{A,i}}{E_i\otimes \mathcal H_1}
   = \widetilde M_{A,i}, \qquad
\restr{W_{A,ij}}{\mathcal H_1}
   = U_{A,ij}, \qquad
\restr{\tau_A(a)}{\mathcal H_1}
   = \sigma_{A}(a).
\]
Moreover, for $j\in I_k$, if either for all $ m\in I$ there exists $n>m$ such that
\[
\widetilde M_{n,j}\, \widetilde M_{n,j}^{*}\,\varphi_{n,m} = \varphi_{n,m},
\]
or 
\[
\widetilde M_{n,j}\, \widetilde M_{n,j}^{*}=I_{\mathcal H_n}\qquad \forall\,n\in I,
\]
then $(\tau_A,V_{A,j})$ is coisometric.  
In particular, if either condition holds for every $j\in I_k$, then 
$(\tau_A,V_{A,1},\dots,V_{A,k})$ is fully coisometric (and hence doubly twisted).
\end{cor}
We illustrate this corollary through the following Example:
\begin{example}\label{ex: unitary-ext-commutative}
Let $\mathcal A$ be a unital commutative $C^*$\nobreakdash-algebra, and let 
$\mathbb E=\{E_i\}_{i\in I_k}$ be a product system of $C^*$\nobreakdash-correspondences
over $\mathbb Z_+^k$.  
For each $i\in I_k$, take $E_i=\mathcal A$ as a correspondence over $\mathcal A$ with the 
usual left and right actions
\(
a\cdot x\cdot b \;:=\; axb,
\) and
the structure maps \(
t_{ij}:E_i\otimes_{\mathcal A} E_j \longrightarrow E_j\otimes_{\mathcal A} E_i,
\)
\(
t_{ij}(a\otimes b)=b\otimes a, \, \forall a,x,b\in\mathcal A.
\)
Let $(\sigma,\, T_1,\dots,T_k)$ be a doubly twisted isometric covariant 
representation of the product system $\mathbb E$ on a Hilbert space $\mathcal H$.  
Then $(\sigma, T_1,\dots,T_k)$ admits a unitary extension.
\end{example}
The discussion preceding Corollary~\ref{cor:direct-limit-A} shows that it suffices to prove the statement for the tuple   $\left(\sigma_A, M_{A,1},\dots, M_{A,k}\right)$ on $\mathcal F(E_A)\otimes_\sigma \mathcal D_A$, where \(A =\{i_1, \cdots, i_p\} \subseteq I_n\) is arbitrary. Thus, our goal is to construct a direct system of Hilbert spaces
\(
\{(\mathcal H_m,\varphi_{n,m})\}_{m \le n \in I},
\)
and, for each \(m \in I\), a doubly twisted isometric covariant representation
\(
(\sigma_{m},\, M_{m,1},\dots, M_{m,k})
\)
of the product system \(\mathbb E=\{E_i\}_{i\in I_k}\) on \(\mathcal H_m\), with twists \(\{\, t_{ij}\otimes U_{m,ij}\,\}_{i\ne j}\), such that all the hypotheses of the corollary are satisfied.

As \(\mathcal A\) is a unital commutative \(C^{*}\)\nobreakdash-algebra and \(E_i = \mathcal A\) for each \(i\), we have
\(
E_i\otimes_{\mathcal A}E_j \cong \mathcal A
\)
via the map \(a\otimes b \mapsto ab\). Since \(ab = ba\), this identification is symmetric, and hence we may identify
\(t_{ij} = \mathrm{id}_{\mathcal A}\).
Moreover,
\[
\mathcal F(E_A)\otimes_\sigma \mathcal D_A
\;=\;
\ell^2(\mathbb Z_+^{|A|})\otimes \mathcal D_A.
\]
Let \(\{e_n\}_{n\in\mathbb Z_+^{|A|}}\) denote the canonical orthonormal basis of
\(\ell^2(\mathbb Z_+^{|A|})\).
Recall that for \(n\in\mathbb Z_+^{|A|}\), \(h\in\mathcal D_A\), and \(\xi \in \mathcal A\),
using the identification \(\xi \otimes h = 1\otimes \sigma(\xi)h\) in
\(\mathcal A \otimes_\sigma\mathcal D_A\), we have
\begin{equation}\label{eq:PiA-star}
\Pi_A\!\left(\widetilde T_A^{\,n}(\xi\otimes h)\right)
= \xi\otimes h
= e_n\otimes \sigma(\xi)\,h,
\qquad\text{and}\qquad
\Pi_A^*(e_n\otimes d)\;=\;\widetilde T_A^{\,n}(1\otimes d).
\end{equation}
By Proposition~\ref{prop:construction of MAi}\,(2),
for any \(a\in\mathcal A\), \(n\in\mathbb Z_+^{|A|}\) and \(h\in\mathcal D_A\), we obtain
\begin{align*}
(\Pi_A\sigma(a)\Pi_A^*)(e_n\otimes h)
&= \phi_{\infty}(a)e_n\otimes h 
= a e_n\otimes h
= e_n\otimes \sigma(a) h 
= (I_{\ell^2(\mathbb Z_+^{|A|})}\otimes \sigma(a))(e_n\otimes h).
\end{align*}
Thus,
\(
\sigma_A(a)
= \Pi_A\sigma(a)\Pi_A^*
= I_{\ell^2(\mathbb Z_+^{|A|})}\otimes \sigma(a).
\) Lemma~\ref{lem:sigma-reduces-coreA-kernels}\,(1) implies that \(\sigma(\mathcal A)\) reduces \(\mathcal D_A\).
Therefore, \(\sigma_A\) is well-defined. 
Also, for all $i \neq j,$ and again using Lemma \ref{lem:sigma-reduces-coreA-kernels} (2), we get a well-defined $U_{A, ij}$ defined by
\[
U_{A, ij}:=\Pi_A\,U_{ij}\,\Pi_A^*= I_{\ell^2(\mathbb Z_+^{|A|})} \otimes U_{ij}.
\]

Let \(S_i\in\mathcal B\left(\ell^2(\mathbb Z_+^{|A|})\right)\) be given by
\(S_i e_n =e_{\,n+e_i}\); this is the unilateral shift in the \(i\)th coordinate.
For any \(a\in\mathcal A\) and \(i\in I_k\), the operators
\(
M_{A,i}(a)\in\mathcal B\left(\ell^2(\mathbb Z_+^{|A|})\otimes\mathcal D_A\right)
\)
are explicitly given by
\[
M_{A,i}(a) =
\begin{cases}
S_{i_1}\otimes \sigma(a),
& i=i_1,\\[8pt]
\displaystyle
S_{i_j}\otimes \sigma(a)\,
\prod_{r=1}^{j-1}
D_{r}\!\big[U_{i_j,i_{r}}\big],
& i=i_j,\; 2\le j\le p,\\[10pt]
\displaystyle
I_{\ell^2(\mathbb Z_+^{|A|})}\otimes  T_l(a)\,
\prod_{r=1}^{p} D_r[U_{l,i_r}],
& i=l \in A^{\mathrm c}.
\end{cases}
\]
We obtain these formulas from Proposition~\ref{prop:construction of MAi}\,(1). Indeed,
\begin{align*}
M_{A,i_1}(a)(e_n\otimes h)
&= \Pi_A T_{i_1}(a)\, \Pi_A^{*} (e_n\otimes h)
= (T_{a}\otimes I_{\mathcal D_A}) (e_n\otimes h)
= a\otimes e_n \otimes h
= e_{n+e_{i_1}}\otimes \sigma(a)h.
\end{align*}
For \(2\le j\le p\),
\begin{align*}
(\Pi_A T_{i_j}(a)\Pi_A^*)(e_n\otimes h)
&=
\Bigg(\prod_{r=1}^{j-1}
D_{j-r}\big[W_{i_j,i_{j-r}}\big]\Bigg)
(T_{a}\otimes I_{\mathcal D_A})(e_n\otimes h)\\
&=
\Bigg(\prod_{r=1}^{j-1}
D_{j-r}\big[W_{i_j,i_{j-r}}\big]\Bigg)
(a\otimes e_n\otimes h) \\
&= e_{n+e_{i_j}} \otimes \sigma(a)
\prod_{r=1}^{j-1}
D_{r}\big[U_{i_j,i_{r}}\big]h.
\end{align*}
For \(l \in A^{\mathrm c}\),
\begin{align*}
(\Pi_A T_l(a)\Pi_A^*)(e_n \otimes h)
&=\left(I_{\mathcal F(E_A)}\otimes \widetilde T_l\right)
\Bigg(\prod_{r=1}^{p}
D_{p+1-r}\!\big[W_{l,i_{p+1-r}}\big]\Bigg)
(T_{a}\otimes I_{\mathcal D_A})(e_n \otimes h)\\
&=e_n \otimes  T_l(a)
\prod_{r=1}^{p}
D_{r}\big[U_{l,i_{r}}\big] h.
\end{align*}
Therefore, we get a \emph{doubly twisted isometric covariant representation}
\(\left(\sigma_A,\{M_{A,i}\}_{i\in I_k}\right)\) 
with twists \(\{\,t_{ij}\otimes U_{A,ij}\,\}_{i\neq j}\).

Now for each \(m \in I\), let $\clh_m=\ell^2(\mathbb Z_+^{|A|}) \otimes \cld_A$, and let us construct a doubly twisted isometric covariant representation
\(
(\sigma_{m},\, M_{m,1},\dots, M_{m,k})
\) 
of the product system \(\mathbb E=\{E_i\}_{i\in I_k}\) on \(\mathcal H_m\) with twists \(\{\, t_{ij}\otimes U_{m,ij}\,\}_{i\ne j}\). First let 
\[
X_{i_1}= I_{\ell^2(\mathbb Z_+^{|A|}) \otimes \mathcal D_A}.
\]
For each \(i_j \in A \setminus \{i_1\}\), and for each \(l \in A^{\mathrm c}\), let
\[
X_{i_j}
=
I_{\ell^2(\mathbb Z_+^{|A|})}
\otimes
\left(\prod_{r<j} U_{i_r i_j}\right),
\qquad
Y_l
=
I_{\ell^2(\mathbb Z_+^{|A|})}
\otimes
\left(\prod_{i \in A} U_{i l}\right).
\]
For each \(m \in I\) and \(i \in I_k\), define
\(
M_{m,i} : \mathcal A \longrightarrow 
\mathcal B\!\left(\ell^2(\mathbb Z_+^{|A|}) \otimes \mathcal D_A\right)
\)
by
\[
M_{m,i}(a)
\;=\;
M_{A,i}(a)\, Z_i^{\,m},
\qquad\text{where}\qquad
Z_i =
\begin{cases}
X_i, & i\in A,\\[4pt]
Y_i, & i\in A^{\mathrm c}.
\end{cases}
\]
Therefore, we obtain a map
\(
\widetilde M_{m,i} : 
\mathcal A \otimes (\ell^2(\mathbb Z_+^{|A|}) \otimes \mathcal D_A)
\longrightarrow
\ell^2(\mathbb Z_+^{|A|}) \otimes \mathcal D_A,
\)
such that for \(a\in \mathcal A\) and 
\(h \in \ell^2(\mathbb Z_+^{|A|}) \otimes \mathcal D_A\),
\[
\widetilde M_{m,i}(a\otimes h)
\;=\;
M_{A,i}(a)\, Z_i^{\,m} h.
\]
As \(\{U_{ij}\}_{i\neq j}\) is a commuting family, $U_{ij}\,\sigma(a) = \sigma(a)\,U_{ij}$ for all $a\in\mathcal A$, and
\(
U_{ij}\,\widetilde T_k
=
\widetilde T_k\, (I_{\mathcal A}\otimes U_{ij}), 
\ \forall k\in I_k,
\)
using definition of $Z_l$, we get
\begin{equation}\label{eq: Z_i commutes MAi}
M_{A,i}(a)\, Z_j^{\,m} = Z_j^{\,m}\, M_{A,i}(a),
\qquad
Z_\ell\, U_{A,ij} = U_{A,ij}\, Z_\ell,
\qquad
\forall\, i,j,\ell\in I_k,\; i\neq j.
\end{equation}
Finally, for all $m \in I$, define
\[
\sigma_{m,A} := \sigma_A = I \otimes \sigma,
\qquad
U_{m, A, ij} := U_{A,ij}.
\]
 Notice that, for $m=0,$ we recover our doubly twisted isometric covariant representation
\(\left(\sigma_A,\{M_{A,i}\}_{i\in I_k}\right)\) 
with twists \(\{\,t_{ij}\otimes U_{A,ij}\,\}_{i\neq j}\).

\begin{prop} \label{prop: mth level doubly twisted}
For each \(m \in I\), the tuple 
\(\left(\sigma_{m},\, M_{m,1},\dots, M_{m,k}\right)\)
is a doubly twisted isometric covariant representation of the product system
\(\mathbb E\)
with twists \(\{\,t_{ij}\otimes U_{m,ij}\,\}_{i\ne j}\).
\end{prop}
\begin{proof}
Fix \(m\in I\). 
By construction, each \(M_{A,i}\) is an isometry (unitary if \(i\in A^c\)), hence each
\(M_{m,i} = M_{A,i} Z_i^{\,m}\) is also an isometry (unitary if \(i\in A^c\)).
Since \([\sigma(\mathcal A), U_{ij}]=0\) for all \(i\neq j\), therefore, \([\sigma_m(a), Z_i]=0\) for all \(a\in\mathcal A\) and \(i\in I_k\). 
Using the covariance of \(M_{A,i}\) and this commutation, for all \(a,x,b\in\mathcal A\),
\[
M_{m,i}(axb)
= M_{A,i}(axb) Z_i^{\,m}
= \sigma_A(a)\,M_{A,i}(x)\,\sigma_A(b)\,Z_i^{\,m}
= \sigma_m(a)\,M_{m,i}(x)\,\sigma_m(b).
\]
Thus \((\sigma_m,M_{m,1},\dots,M_{m,k})\) is covariant.
Now, for \(l \in I_k\) and \(i \ne j\), Equation \ref{eq: Z_i commutes MAi} implies
\[
U_{m,ij} M_{m,l}(a)
= U_{m,ij} M_{A,l}(a) Z_l^{\,m}
= M_{A,l}(a) Z_l^{\,m} U_{m,ij}
= M_{m,l}(a)\, U_{m,ij}, \qquad a \in \mathcal{A}.
\]
Hence,
\(
U_{m,ij}\, \widetilde M_{m,l}
= \widetilde M_{m,l}\, (I_{\mathcal A} \otimes U_{m,ij}).
\) For all $i \ne j, \, \forall a \in \mathcal{A}$, the relations $U_{m,ij}^*= U_{m,ji},$ and $\sigma_m(a) U_{m,ij}= U_{m,ij}\sigma_m(a)$ follow trivially from their definitions.
To prove the twisted relations, let \(a,b\in\mathcal A\) and \(h\in\mathcal H_m\),
\begin{align*}
\widetilde M_{m,i}(I_{\mathcal A}\otimes \widetilde M_{m,j})(a\otimes b\otimes h)
&= M_{m,i}(a)\, M_{m,j}(b)\, h \\
&= M_{A,i}(a)\, Z_i^{\,m} M_{A,j}(b)\, Z_j^{\,m} h \\
&= M_{A,i}(a) M_{A,j}(b)\, Z_i^{\,m} Z_j^{\,m} h
   \qquad\text{(by \eqref{eq: Z_i commutes MAi})} \\
&= \widetilde M_{A,i}(I_{\mathcal A}\otimes \widetilde M_{A,j})
   (a\otimes b\otimes Z_i^{\,m} Z_j^{\,m} h) \\
&= \widetilde M_{A,j}(I_{\mathcal A}\otimes \widetilde M_{A,i})
   (t_{ij}\otimes U_{A,ij})(a\otimes b\otimes Z_i^{\,m} Z_j^{\,m} h)\\
&=  M_{A,j}(b) M_{A,i}(a)  U_{A,ij} Z_i^m Z^{m}_jh\\
&=  M_{A,j}(b)Z^{m}_j M_{A,i}(a) Z_i^m U_{A,ij}  h \qquad\text{(by \eqref{eq: Z_i commutes MAi})}\\
&=  \widetilde M_{m,j}(b \otimes  M_{m,i}(a) U_{A,ij}  h) \\
&= \widetilde M_{m,j} (I_{\mathcal{A}} \otimes \widetilde M_{m,i})(t_{ij}  \otimes U_{m,ij})(a \otimes b \otimes h )\\
\end{align*}
Therefore, $\widetilde M_{m,i} (I_{\mathcal{A}} \otimes \widetilde M_{m,j})=\widetilde M_{m,j} (I_{\mathcal{A}} \otimes \widetilde M_{m,i})(t_{ij}  \otimes U_{m,ij}).$

To prove doubly twisted relations, first notice that
\begin{equation} \label{eq: M*_m,i}
    \widetilde M^{*}_{m,i}
= (I_{\mathcal A}\otimes Z_i^{\,m})^{*}\, \widetilde M_{A,i}^{*}.
\end{equation}
Equation \eqref{eq: Z_i commutes MAi} implies that, for any $i, j \in I_k,$
\[
Z^m_j \widetilde M_{A,i} =  \widetilde M_{A,i} (I_{\mathcal{A}} \otimes Z^m_j) \implies  Z^{*m}_j \widetilde M_{A,i}=\widetilde M_{A,i}(I_{\mathcal{A}} \otimes Z^{*m}_j).
\]
Using this equality along with doubly twisted relations, we get
\begin{align*}
\widetilde M^*_{m,i}\,\widetilde M_{m,j}(b \otimes h)
&= \widetilde M^*_{m,i}\left(M_{A,j}(b)\, Z_j^{\,m} h\right)\\
&= (I_{\mathcal A}\otimes Z_i^{\,m})^{*}\,
   \widetilde M_{A,i}^{*}\left(M_{A,j}(b)\, Z_j^{\,m} h\right)\\
&= (I_{\mathcal A}\otimes Z_i^{\,*m})\, 
   \widetilde M_{A,i}^{*}\, \widetilde M_{A,j}(b \otimes Z_j^{\,m} h)\\
&= (I_{\mathcal A}\otimes Z_i^{\,*m})\,
   (I_{\mathcal A}\otimes \widetilde M_{A,j})\,
   (t_{ji}\otimes U_{m,ji})\,
   (I_{\mathcal A}\otimes \widetilde M_{A,i}^{*})\,
   (b \otimes Z_j^{\,m} h)\\
&= (I_{\mathcal A}\otimes \widetilde M_{A,j})\,
   (t_{ji}\otimes U_{m,ji})\,
   (I_{\mathcal A}\otimes I_{\mathcal A}\otimes Z_i^{\,*m})\,
   (I_{\mathcal A}\otimes \widetilde M_{A,i}^{*})\,
   (b\otimes Z_j^{\,m}h)
   \quad(\text{ by } \eqref{eq: Z_i commutes MAi}\\
&= (I_{\mathcal A}\otimes \widetilde M_{A,j})\,
   (t_{ji}\otimes U_{m,ji})\,
   \left(I_{\mathcal A}\otimes (I_{\mathcal A}\otimes Z_i^{\,m})^{*}\,
   \widetilde M_{A,i}^{*}\right)\,(b\otimes Z_j^{\,m}h)\\
&= (I_{\mathcal A}\otimes \widetilde M_{A,j})\,
   (t_{ji}\otimes U_{m,ji})\,
   (I_{\mathcal A}\otimes \widetilde M_{m,i}^{*})\,
   (I_{\mathcal A}\otimes Z_j^{\,m})(b\otimes h) \quad(\text{ by } \eqref{eq: M*_m,i})\\
&= (I_{\mathcal A}\otimes \widetilde M_{A,j})\,
   (t_{ji}\otimes U_{m,ji})\,
   (I_{\mathcal A}\otimes I_{\mathcal A}\otimes Z_j^{\,m})\,
   (I_{\mathcal A}\otimes \widetilde M_{m,i}^{*})\,
   (b\otimes h)\\
&= (I_{\mathcal A}\otimes \widetilde M_{m,j})\,
   (t_{ji}\otimes U_{m,ji})\,
   (I_{\mathcal A}\otimes \widetilde M^*_{m,i})\,
   (b\otimes h).
\end{align*}
Thus for each \(m\in I\), the tuple
\(\left(\sigma_m, M_{m,1},\dots,M_{m,k}\right)\)
is a doubly twisted isometric covariant representation of
\(\mathbb E=\{E_i=\mathcal A\}_{i\in I_k}\) with twists
\(\{t_{ij}\otimes U_{m,ij}\}_{i\neq j}\).
\end{proof}

Now, for any \(n,m \in I\) with \(n \geq m\), define
\(
\varphi_{n,m} : \ell^2(\mathbb Z_+^{|A|}) \otimes \mathcal D_A 
\longrightarrow \ell^2(\mathbb Z_+^{|A|}) \otimes \mathcal D_A
\)
by
\[
\varphi_{m+1,m} := S_A \otimes I_{\mathcal D_A},
\quad\text{where } S_A := \prod_{i_r \in A} S_{i_r}.
\]
Therefore,
\(
\varphi_{n,m} := (S_A \otimes I_{\mathcal D_A})^{\,n-m}, \, n \ge m.
\)
Since \(S_A\) is an isometry, each \(\varphi_{n,m}\) is an isometry, and for all
\(m \le n \le k\) we have
\[
\varphi_{k,n} \circ \varphi_{n,m} = \varphi_{k,m},
\qquad
\varphi_{m,m} = I_{\ell^2(\mathbb Z_+^{|A|}) \otimes \mathcal D_A}.
\]
Thus \(\{(\mathcal H_m,\varphi_{n,m})\}_{m \le n \in I}\), with
\(\mathcal H_m := \ell^2(\mathbb Z_+^{|A|}) \otimes \mathcal D_A\), 
forms a direct system of Hilbert spaces.  
We denote its direct limit by \((\mathcal K_A,\psi_m)_{m\in I}\), that is, there exists 
a Hilbert space \(\mathcal K_A\) and isometries 
\(\psi_m : \mathcal H_m \to \mathcal K_A\) such that for all $m \le n$,
\(
\psi_n \circ \varphi_{n,m} = \psi_m,
\)
and
\(
\mathcal K_A 
= \overline{\mathrm{span}}\big\{\psi_m(\mathcal H_m) : m \in I\big\}.
\)
By Proposition~\ref{prop: mth level doubly twisted}, for each \(m \in I\), $(\sigma_m,\, M_{m,1},\dots, M_{m,k})$ is a
doubly twisted isometric covariant representation with twists
\(\{\,t_{ij}\otimes U_{m,ij}\,\}_{i\ne j}\). With the maps defined above, we now claim the following.

\begin{prop} \label{prop: twisted extension}
A doubly twisted isometric covariant representation
\(\left(\sigma_A,\{M_{A,i}\}_{i\in I_k}\right)\) 
with twists \(\{\,t_{ij}\otimes U_{A,ij}\,\}_{i\neq j}\) admits a twisted isometric covariant extension.
\end{prop}

\begin{proof}
For any $i \in I_k,$ first we claim that $\varphi_{n,m}\,\widetilde M_{m,i}
=
\widetilde M_{n,i}\,(I_{\mathcal A}\otimes \varphi_{n,m}).$
Fix \(e_s \otimes d \in \ell^2(\mathbb Z_+^{|A|}) \otimes \mathcal D_A\),
where \(s=(s_{i_1},\dots,s_{i_p}) \in \mathbb Z_+^{|A|}\).
For \(j=1\), \(a\in\mathcal A\), and \(m\in I\), we have
\[
\widetilde M_{m,i_1}(a\otimes e_s \otimes d)
= (S_{i_1}\otimes \sigma(a))(e_s\otimes d)
= e_{\,s+e_{i_1}} \otimes \sigma(a)d.
\]
Therefore, with \(e(A):=\sum_{i\in A }e_i\),
\begin{align*}
\varphi_{n,m}\,\widetilde M_{m,i_1}(a\otimes e_s \otimes d)
= (S_A\otimes I)^{\,n-m}(e_{\,s+e_{i_1}} \otimes \sigma(a)d) 
= e_{\,s+e_{i_1} + (n-m)e({A})} \otimes \sigma(a)d.
\end{align*}
On the other hand,
\[
\widetilde M_{n,i_1}(I_{\mathcal A}\otimes \varphi_{n,m})(a\otimes e_s\otimes d)
= \widetilde M_{n,i_1}(a\otimes e_{\,s+(n-m)e(A)}\otimes d)
= e_{\,s+e_{i_1}+(n-m)e(A)}\otimes \sigma(a)d.
\]
Hence,
\[
\varphi_{n,m}\,\widetilde M_{m,i_1}
=
\widetilde M_{n,i_1}\,(I_{\mathcal A}\otimes \varphi_{n,m}).
\]
Now for all \(i_j \in A\setminus \{i_1\}\),
\[
\widetilde M_{A,i_j}(a \otimes e_s\otimes d)
= e_{s+e_{i_j}}\otimes \sigma(a)
  \left(\prod_{r<j}U_{i_j,i_r}^{\,s_{i_r}}\right)d,
\]
and hence
\begin{align*}
  \widetilde M_{m,i_j}(a \otimes e_s\otimes d)
&= M_{m,i_j}(a)(e_s\otimes d)\\
&= M_{A,i_j}(a)\,Z_{i_j}^{\,m}(e_s\otimes d)\\
&= e_{s+e_{i_j}}\otimes 
\sigma(a)\, \left(\prod_{r<j}U_{i_j,i_r}^{\,s_{i_r}}\right)
          \left(\prod_{r<j}U_{i_r,i_j}^{\,m}\right)d\\
&= e_{s+e_{i_j}}\otimes 
\sigma(a)\, \left(\prod_{r<j}U_{i_j,i_r}^{\,s_{i_r}-m}\right)d.
\end{align*}
Applying \(\varphi_{n,m}=(S_A\otimes I)^{n-m}\) gives
\[
\varphi_{n,m}\,\widetilde M_{m,i_j}(a \otimes e_s\otimes d)
= e_{s+e_{i_j}+(n-m) e(A)} \otimes 
\sigma(a)\, \left(\prod_{r<j}U_{i_j,i_r}^{\,s_{i_r}-m}\right)d.
\]
On the other hand,
\[
\begin{aligned}
\widetilde M_{n,i_j}(I_{\mathcal{A}} \otimes \varphi_{n,m})(a \otimes e_s\otimes d)
&=\widetilde M_{n,i_j}(a \otimes e_{s+(n-m)e(A)}\otimes d)\\
&= M_{A,i_j}(a)\,Z_{i_j}^{\,n}(e_{s+(n-m)e(A)}\otimes d)\\
&= e_{s+e_{i_j}+(n-m)e(A)}\otimes \sigma(a)
  \left(\prod_{r<j}U_{i_j,i_r}^{\,s_{i_r}+(n-m)}\right)
  \left(\prod_{r<j}U_{i_r,i_j}^{\,n}\right)d\\
&= e_{s+e_{i_j}+(n-m)e(A)}\otimes \sigma(a)
  \left(\prod_{r<j}U_{i_j,i_r}^{\,s_{i_r}-m}\right)d.
\end{aligned}
\]
Therefore, for all \(m \leq n\) and all \(i_j \in A\),
\[
\varphi_{n,m}\,\widetilde M_{m,i_j}
= \widetilde M_{n,i_j}\, (I_{\mathcal{A}} \otimes\varphi_{n,m}).
\]
For \(l \in A^{\mathrm c}\),  
\(
M_{A,l}(a)(e_s\otimes d)
= e_s\otimes 
  T_l(a)\,
  \left(\displaystyle\prod_{r=1}^p U_{\,l,i_r}^{\,s_{i_r}}\right)d.
\)
Therefore,
\begin{align*}
\varphi_{n,m}\,\widetilde M_{m,l}(a \otimes e_s \otimes d)
&= (S_A \otimes I)^{\,n-m}
   \left(I \otimes 
      T_l(a)
      \left(\prod_{i_r \in A} U_{\,l,i_r}^{\,s_{i_r}}\right)\right)
      Y_l^{\,m}(e_s \otimes d)\\[6pt]
&= (S_A \otimes I)^{\,n-m}
   \left(e_s \otimes
      T_l(a)\,
      \left(\prod_{i_r \in A} U_{\,l,i_r}^{\,s_{i_r}}\right)
      \left(\prod_{i_r \in A} U_{\,i_r,l}^{\,m}\right)
      d\right)\\[6pt]
&=
   e_{\,s+(n-m)e(A)} \otimes
   T_l(a)\,
   \left(\prod_{i_r \in A} U_{\,l,i_r}^{\,s_{i_r}-m}\right)d .
\end{align*}

On the other hand,
\[
\begin{aligned}
\widetilde M_{n,l}( I \otimes \varphi_{n,m})(a \otimes e_s\otimes d)
&= \widetilde M_{n,l}(a \otimes e_{\,s+(n-m)e(A)}\otimes d)\\
&= M_{A,l}(a)\, Y_l^{\,n}(e_{\,s+(n-m)e(A)}\otimes d)\\
&= e_{\,s+(n-m)e(A)}\otimes 
   T_l(a)\,
   \left(\prod_{i_r \in A} U_{\,l,i_r}^{\,s_{i_r}+ (n-m)}\right)
   \left(\prod_{i_r \in A} U_{\,i_r,l}^{\,n}\right)d\\[3pt]
&= e_{\,s+(n-m)e(A)}\otimes 
   T_l(a)\,
   \left(\prod_{i_r \in A} U_{\,l,i_r}^{\,s_{i_r}-m}\right)d .
\end{aligned}
\]
Therefore, for all \(m\le n\) and $l \in A^c,$
\[
\varphi_{n,m}\,\widetilde M_{m,l}
\;=\;
\widetilde M_{n,l}\,( I_{\mathcal A} \otimes \varphi_{n,m}).
\]
Hence, the claim follows.
Finally, for all \(a \in \mathcal A\) with \(m \le n\), 
\[
\varphi_{n,m}\,\sigma_{m}(a)
=(S_A\otimes I)^{\,n-m}(I\otimes\sigma(a))
=(I\otimes\sigma(a))(S_A\otimes I)^{\,n-m}
=\sigma_{n}(a)\,\varphi_{n,m},
\]
and  for all \(i \neq j\) in \(I_k\),
\[
\varphi_{n,m}\, U_{m,ij}
= (S_A\otimes I)^{\,n-m}(I\otimes U_{ij})
= (I\otimes U_{ij})(S_A\otimes I)^{\,n-m}
= U_{n,ij}\,\varphi_{n,m}.
\]
Therefore, all the hypotheses of Corollary~\ref{cor:direct-limit-A} are satisfied, 
and we obtain a twisted isometric covariant extension
\((\tau_A,\, V_{A,1},\dots,V_{A,k})\)
of \(\mathbb E\) on \(\mathcal K_A\), with twists 
\(\{\,t_{ij}\otimes W_{A,ij}\,\}_{i\ne j}\) such that
\[
\restr{\widetilde V_{A,i}}{E_i\otimes \mathcal H_0}
   = \widetilde M_{A,i}, \qquad
\restr{W_{A,ij}}{\mathcal H_0}
   = U_{A,ij}, \qquad
\restr{\tau_A(a)}{\mathcal H_0}
   = \sigma_{A}(a).
\]
This completes the proof of the claim.
\end{proof}
We now claim that this extension is doubly twisted and, in fact, unitary.  
To verify this, we need to compute the adjoint of \(\widetilde M_{A,i}\) for each \(i \in A\). First notice that, any 
\( \zeta \in \mathcal A\otimes_\sigma
   (\ell^2(\mathbb Z_+^{|A|})\otimes\mathcal D_A)\)
can be written in the form
\[
\zeta=\sum_{k,l} b_{kl}\,(a_{kl}\otimes e_k\otimes h_l)=\sum_{k,l} b_{kl}\,\left(1\otimes e_k\otimes \sigma(a_{kl})h_l\right),
\qquad b_{kl}\in\mathbb C,\; a_{kl}\in\mathcal A.
\]
Let \(\sigma(a_{kl})h_l=\sum_m d^{(k,l)}_m h_m\), and \(c_{km}=\sum_l b_{kl}d^{(k,l)}_m\). Then
\[
\zeta
=\sum_{k,m} c_{km}\,(1\otimes e_k\otimes h_m), \qquad
\sum_{k,m} |c_{km}|^2 < \infty.
\]
Under the identification 
\(\mathcal A\otimes_\sigma(\ell^2(\mathbb Z_+^{|A|})\otimes\mathcal D_A)
 \cong \ell^2(\mathbb Z_+^{|A|})\otimes\mathcal D_A\),
given by \(a\otimes(\xi\otimes h)\longmapsto \xi\otimes\sigma(a)h\),
the vectors \(\{\,1\otimes e_k\otimes h_m\,\}_{k,m}\) correspond to the
standard orthonormal basis \(\{e_k\otimes h_m\}_{k,m}\), and hence form 
an orthonormal basis of 
\(\mathcal A\otimes_\sigma(\ell^2(\mathbb Z_+^{|A|})\otimes\mathcal D_A)\).
Thus \(\{\,1\otimes e_k\otimes h_m\,\}_{k,m}\) is an orthonormal basis of
\(\mathcal A\otimes_\sigma(\ell^2(\mathbb Z_+^{|A|})\otimes\mathcal D_A).\)

\begin{prop}\label{Adjoint of the twisted creation operator}
Let \(\{\,e_n \otimes h_m : n \in \mathbb Z_+^{|A|},\, m \in \mathbb Z_+\,\}\) be an 
orthonormal basis of 
\(\ell^2(\mathbb Z_+^{|A|}) \otimes \mathcal D_A\).
For each \(i_j \in A\), the adjoint
\[
\widetilde M_{A,i_j}^* :
\ell^2(\mathbb Z_+^{|A|}) \otimes \mathcal D_A
\;\longrightarrow\;
\mathcal A \otimes 
\ell^2(\mathbb Z_+^{|A|}) \otimes \mathcal D_A
\]
is given by
\[
\widetilde M_{A,i_j}^*(e_n \otimes h_m)
=
\begin{cases}
\displaystyle 
1 \otimes 
\left(e_{\,n-e_{i_j}} \otimes 
      \prod\limits_{r=1}^{j-1} 
      U_{\,i_j, i_r}^{\,*\, n_{i_r}}\; h_m\right),
& \text{if } n_{i_j} \ge 1,\\[10pt]
0, & \text{if } n_{i_j} = 0.
\end{cases}
\]
\end{prop}
\begin{proof}
    Let
\(
\widetilde M_{A,i_j}^*(e_n\otimes h_m)
= \sum_{k,p} a_{k, p}\,\left(1\otimes e_k \otimes h_p\right),
\,  a_{k,p}\in\mathbb C.
\)
Then 
\begin{align*}
 a_{k,p} 
 &= \left\langle \widetilde M_{A,i_j}^*(e_n\otimes h_m),\, 1 \otimes e_k \otimes h_{p} \right\rangle   \\
 &= \left\langle e_n\otimes h_m,\, \widetilde M_{A,i_j}(1\otimes e_k\otimes h_p)
            \right\rangle\\
&= \left\langle e_n\otimes h_m,\, e_{k+ e_{i_j}} \otimes 
\prod_{r=1}^{j-1}
D_{r}\!\left[U_{i_j,i_{r}}\right]h_{p}
            \right\rangle\\
&= \delta_{\,n,\;k+e_{i_j}} \left\langle  h_m, 
\prod_{r=1}^{j-1}
D_{r}\!\left[U_{i_j,i_{r}}\right]h_{p}
            \right\rangle.
\end{align*}

Therefore, $a_{k,p} \neq 0$ if and only if $k=n-e_{i_j}$.  
Let $k=n-e_{i_j}$. Then
\begin{align*}
 a_{n-e_{i_j},\,p} 
 &=  \left\langle  h_m, 
\prod_{r=1}^{j-1}
D_{r}\!\big[U_{i_j,i_{r}}\big]h_{p}
            \right\rangle \\
&= \left\langle  h_m, 
\prod_{r=1}^{j-1}
U^{k_{i_r}}_{i_j,i_{r}}\,h_{p}\right\rangle \\
&= \left\langle \prod_{r=1}^{j-1}
U^{*k_{i_r}}_{i_j,i_{r}}\, h_m, 
h_{p} \right\rangle.
\end{align*}

This further implies, \(
\widetilde M_{A,i_j}^*(e_n\otimes h_m)
= \displaystyle\sum_{k,p} a_{k, p}\,\left(1\otimes e_k \otimes h_p\right)
=  \sum_{p} a_{n-e_{i_j}, p}\,\left(1\otimes e_{n-e_{i_j} } \otimes h_p\right).
\)
Hence, 
\begin{align*}
 \widetilde M_{A,i_j}^*(e_n\otimes h_m)
&=\sum_{p}\left\langle \prod_{r=1}^{j-1}
U^{*{(n-{e_{i_j}})}_{r}}_{i_j,i_{r}} h_m, 
h_{p} \right\rangle \left(1\otimes e_{n-e_{i_j} } \otimes h_p\right)  \\
&=1\otimes e_{n-e_{i_j} } \otimes \sum_{p} \left\langle \prod_{r=1}^{j-1}
U^{*{(n-{e_{i_j}})}_{r}}_{i_j,i_{r}} h_m, 
h_{p} \right\rangle h_p  \\
&=1\otimes e_{n-e_{i_j} } \otimes  \prod_{r=1}^{j-1}
U^{*{(n-{e_{i_j}})}_{r}}_{i_j,i_{r}} h_m. \\
\end{align*}
Since \(s = n-e_{i_j}\), we have \(s_{i_r} = n_{i_r}\) for all \(r<j\), and we may write
\[
\widetilde M_{A,i_j}^*(e_n\otimes h_m)
=
\begin{cases}
1\otimes e_{\,n-e_{i_j}}\otimes  
   \displaystyle\prod_{r=1}^{j-1}
      U_{i_j i_r}^{*\,n_{i_r}}\, h_m,
&\text{if } n_{i_j}\ge 1,\\[6pt]
0, &\text{if } n_{i_j}=0.
\end{cases}
\]
\end{proof}

In order to prove that the extension constructed in
Proposition~\ref{prop: twisted extension} is a doubly twisted unitary
representation, we begin by showing that any twisted representation which is
fully coisometric must in fact satisfy the doubly twisted relations.

\begin{theorem} 
    A doubly twisted isometric covariant representation 
\(\left(\sigma_A,\{M_{A,i}\}_{i\in I_k}\right)\) 
with twists \(\{\,t_{ij}\otimes U_{A,ij}\,\}_{i\neq j}\) admits a doubly twisted unitary covariant extension.

\end{theorem}
\begin{proof}
    Now, for each \(i_j \in A\), we claim that
\[
\widetilde M_{A,i_j}^{(n)}\, \widetilde M_{A,i_j}^{*(n)} \circ \varphi_{n,m}
\;=\;
\varphi_{n,m}.
\]
Let \(e_s \otimes h \in \ell^2(\mathbb Z_+^{|A|}) \otimes \mathcal D_A\). Then
\begin{align*}
 \widetilde M_{A,i_j}^{(n)}\, \widetilde M_{A,i_j}^{*(n)} \circ \varphi_{n,m}(e_s \otimes h)
 &= \widetilde M_{A,i_j}\, X_{i_j}^{\,n}\, X_{i_j}^{*n}\,
    \widetilde M_{A,i_j}^{*}\left( e_{\,s+(n-m)e(A)} \otimes h \right)\\[4pt]
 &= \widetilde M_{A,i_j}\left(
      1 \otimes e_{\,s+(n-m)e(A) - e_{i_j}}
      \otimes 
      \prod_{r=1}^{j-1}
      U_{i_j i_r}^{*\, (s+(n-m)e(A) - e_{i_j})_r}
      h
    \right)\\[4pt]
 &= e_{\,s+(n-m)e(A)} \otimes
    \prod_{r=1}^{j-1}
    U_{i_j i_r}^{\, (s+(n-m)e(A) - e_{i_j})_r}\,
    \prod_{r=1}^{j-1}
    U_{i_j i_r}^{*\, (s+(n-m)e(A) - e_{i_j})_r}\,
    h\\[4pt]
 &= e_{\,s+(n-m)e(A)} \otimes h\\[4pt]
 &= \varphi_{n,m}(e_s \otimes h).
\end{align*}

Thus, for each \(i_j \in A\), the pair \((\sigma_A, M_{A,i_j})\) is coisometric.  
For every \(j \in A^{\mathrm c}\), we already know that \((\sigma_A, M_{A,j})\) is coisometric.  
Therefore, by Corollary~\ref{cor:direct-limit-A}, we obtain the required unitary extension. Lemma \ref{lemma: coiso plus twisted implies doubly twisted} further implies that the extension is doubly twisted.

\end{proof}

\section*{Acknowledgements}
The authors acknowledge the Department of Mathematics,
Technion—Israel Institute of Technology, Haifa, Israel,
for providing institutional support and a conducive research
environment during the preparation of this work.

\end{document}